\documentclass[12pt, twoside]{article}
\usepackage{amsmath,amsthm,amssymb}
\usepackage{amsfonts}
\usepackage{times}
\usepackage{enumerate}
\usepackage{graphicx}
\usepackage{tikz}
\usepackage{csvsimple}
\usepackage{subcaption}
\usepackage{verbatim}
\usepackage{hyperref}
\usepackage{mathrsfs}
\usepackage{mathtools}
\usepackage[justification=centering]{caption}
\usepackage[top=1in, bottom=1.25in, left=1.25in, right=1.25in]{geometry}

\def\titlerunning#1{\gdef\titrun{#1}}
\makeatletter
\def\author#1{\gdef\autrun{\def\and{\unskip, }#1}\gdef\@author{#1}}
\def\address#1{{\def\and{\\\hspace*{18pt}}\renewcommand{\thefootnote}{}%
\footnote {#1}}%
\markboth{\autrun}{\titrun}}
\makeatother
\def\email#1{e-mail: #1}
\def\subjclass#1{{\renewcommand{\thefootnote}{}%
\footnote{\emph{Mathematics Subject Classification (2010):} #1}}}
\def\keywords#1{\par\medskip
\noindent\textbf{Keywords.} #1}

\usepackage[
backend=biber,
style=alphabetic,
sorting=anyt,
maxbibnames=99
]{biblatex}
\renewbibmacro{in:}{}
\newtheorem{thm}{Theorem}[section]
\newtheorem{cor}[thm]{Corollary}
\newtheorem{lem}[thm]{Lemma}

\newtheorem{prop}[thm]{Proposition} 

\theoremstyle{definition}

\numberwithin{equation}{section}

\begin{document}



\titlerunning{Spectral Decimation for Families of Laplacians on the Sierpinski Gasket}

\title{Spectral Decimation for Families of Self-Similar Symmetric Laplacians on the Sierpinski Gasket}

\author{Sizhen Fang, Dylan A. King, Eun Bi Lee, Robert S. Strichartz}

\date{}

\maketitle

\address{S. Fang: Department of Mathematics and Statistics, Mount Holyoke College, Clapp Laboratory, South Hadley, MA, 01075; \email{\href{mailto:fang22s@mtholyoke.edu}{fang22s@mtholyoke.edu}}\and
D. A. King: Department of Mathematics and Statistics, Wake Forest University, Manchester Hall, Winston-Salem, NC, 27109; \email{\href{mailto:kingda16@wfu.edu}{kingda16@wfu.edu}}\and
E. Lee: Department of Mathematics, University of Chicago, Eckhart Hall, Chicago, IL, 60637; \email{\href{mailto:seraphina@math.uchicago.edu}{seraphina@math.uchicago.edu}}\and
R. S. Strichartz: Department of Mathematics, Cornell University, Malott Hall, Ithaca, NY, 14853; \email{\href{mailto:str@math.cornell.edu}{str@math.cornell.edu}}}


\begin{abstract}
We construct a one-parameter family of Laplacians on the Sierpinski Gasket that are symmetric and self-similar for the 9-map iterated function system obtained by iterating the standard 3-map iterated function system. Our main result is the fact that all these Laplacians satisfy a version of spectral decimation that builds a precise catalog of eigenvalues and eigenfunctions for any choice of the parameter. We give a number of applications of this spectral decimation. We also prove analogous results for fractal Laplacians on the unit Interval, and this yields an analogue of the classical Sturm-Liouville theory for the eigenfunctions of these one-dimensional Laplacians.


\end{abstract}

\section{Introduction}
In the theory of analysis on fractals, the standard Laplacian on the Sierpinski gasket first presented by Kigami \cite{Kigami1} stands as a kind of ``poster child", as it is nontrivial but completely understandable (See \cite{Strichartz} for an elementary exposition). Just as there are families of Laplacians associated to manifolds (usually described in terms of Riemannian metrics), so there are families of Laplacians on the Sierpinski gasket (SG). Note that SG is a self-similar fractal, characterized by the self-similar identity
\begin{equation}\label{sgdef}
SG = \bigcup_{j=0}^2 F_j(SG)
\end{equation}
where $F_j$ are the contractions of the plane
\begin{equation}
F_j(x) = \frac{1}{2}(x-q_j) + \frac{1}{2}q_j
\end{equation}
with $(q_0, q_1, q_2)$ the vertices of an equilateral triangle. The standard Laplacian $\Delta$ is self-similar, meaning that \begin{equation}\label{renormdef}
\Delta(u\circ F_j) = r_j^{-1}(\Delta u)\circ F_j
\end{equation}
for some positive coefficients $r_j$ (in this case, all $r_j = 5$), and symmetric with respect to the dihedral symmetry group $D_3$ of the triangle (and hence SG). Moreover, the standard Laplacian is characterized, up to a constant, by these two properties. Is this the end of the story?

Actually not. Already in \cite{Cucuringu} it was noted that you can modify the iterated function system (IFS) $\{F_j\}$ to another one $\{\widetilde{F}_j\}$ that composes each $F_j$ with the reflection preserving $q_j$, and still generate SG by the analog of (\ref{sgdef}). This allows the construction of another family of Laplacians in a very explicit fashion, but still the standard Laplacian is the unique one that is both self-similar and symmetric.

Another way to generate SG is to take the IFS consisting of all nine compositions $F_j \circ F_k$. The sets $F_jF_k(SG)$ gives a level two subdivision with respect to the original IFS that becomes the level one subdivision for the composite IFS. The symmetry condition says that the three outer cells $F_jF_j(SG)$ are equivalent, as are the six inner cells $F_jF_k(SG)$ for $j\neq k$. We then can construct a two-parameter family of Laplacians that are both self-similar and symmetric. If we add one simplifying condition that makes the renormalization coefficients in the analog of (\ref{renormdef}) equal for all nine cells, then we end up with a one-parameter family of symmetric, self-similar Laplacians. It is this family of Laplacians that we examine in detail in this paper. The standard Laplacian belongs to this familiy for the parameter choice $r=1$.

A remarkable property of the standard Laplacian on SG, called \textit{spectral decimation}, was discovered by Fukushima and Shima \cite{Fukushima}. It is natural to consider SG as the limit of a sequence of graphs, and the standard Laplacian as a limit of graph Laplacians. In particular, there is a straightforward algorithm to construct harmonic functions on the graph approximations. Spectral decimation allows you to explicitly construct eigenfunctions and eigenvalues on SG as limits of eigenfunctions and eigenvalues on the graph approximations. A key result in this paper is the discovery of an analog of spectral decimation for the whole family of Laplacians we consider. This is quite surprising, since there are many fractal Laplacians that are extremely symmetric but do not satisfy spectral decimation (see \cite{Pentagasket} for the case of the pentagasket). Using spectral decimation, we are able to answer many interesting questions about the spectra of our family of Laplacians.

The construction of the standard Laplacian on SG is an exact analog of the construction of the second derivative on the unit Interval as a limit of second difference quotients. In a similar way, there are analogs to our twice-iterated gasket construction on the unit Interval, based on the self-similar identity
\begin{equation}
I=\bigcup_{i=0}^3 F_i(I)
\end{equation}
for
\begin{equation}
F_i(x) = \frac{1}{4}x + \frac{i}{4}
\end{equation}
and a system of weights that treats inner and outer maps separately. In fact, we present our results for these one-dimensional fractal Laplacians first since the description is simpler and we can say more in this context. In particular, the eigenfunctions for these Laplacians satisfy analogs of Sturm-Liouville theory concerning locations of zeros and local extrema. It would be fair to think of them as forming a family of special functions analogous to $\{\sin (k\pi x)\}$.

Sections 2, 3, and 4 of this paper are devoted to the family of Laplacians on the Interval, with the description of the Laplacians in section 2, the theory of spectral decimation in section 3, and numerical data in section 4. This data is selected from the website \cite{Website}, which also contains the programs used to generate the data. In section 5 we prove the Sturm-Liouville properties of our one-dimensional eigenfunctions. Sections 6, 7, and 8 present the analogs of sections 2, 3, and 4 for the SG Laplacians. In section 9 we discuss a method we call \textit{threshold subdivision} to create different Laplacians using the same parameters but subdividing (or not) cells at one level to create cells of the next level based on the measure of the cell. We present experimental evidence that the Laplacians obtained are different. In section 10 we study hierarchical Laplacians which are not self-similar but vary the parameters at different levels of the construction, as in \cite{Drenning}. In section 11 we present data for solutions of spacetime equations, such as the heat equation and the wave equation. See also \cite{ACBMT} for related results.

Although we present a large amount of numerical data, most of our important results are given complete proofs. To some extent this disguises the experimental nature of our work, since the numerical data led us to conjecture the results that we were then able to prove. The reader should consult \cite{Kigami2} or \cite{Strichartz} for the standard theory of the Laplacian on SG.

\section{Laplacians on the Interval}

To construct the Interval model, we start with a unit Interval $I = [0,1]$ and the IFS $\{F_{i}|F_i(x) = \frac{x}{4} + \frac{i}{4}, i=0,1,2,3\}$. This gives us the following self-similar identity on the Interval  
\begin{equation}
I=\bigcup_{i=0}^{3} F_{i}(I).
\end{equation}
In other words, we divide the Interval into subintervals of length $\left(\frac{1}{4}\right)^{m}$ on each level, so that at level $m$ the subintervals are $I_{k}^{(m)}=[\frac{k}{4^{m}},\frac{k+1}{4^{m}}]$, where $0 \leqslant k \leqslant 4^{m}-1$ (see Figure \ref{fig:adjmcells}).

After constructing the model we can assign measure and resistance to the Interval. On level 1, we assign a measure $\frac{p}{2}$ (where $0 < p < 1$) to the outer intervals $I_{0}^{(1)}$ and $I_{3}^{(1)}$ and a measure $\frac{1-p}{2}$ to the inner intervals $I_{1}^{(1)}$ and $I_{2}^{(1)}$ (as in Figure \ref{fig:intlev1}). Every time we take a cell from level $m$ and subdivide it into 4 cells on level $m+1$, we split the measure of the $m$-cell in the same proportions
\begin{equation}
\mu(I_{4k+j}^{(m+1)})= \begin{cases} \frac{p}{2}\mu(I_{k}^{(m)})  &j=0,3, \\ \frac{1-p}{2}\mu(I_{k}^{(m)})  &j=1,2. \end{cases}
\end{equation}

Let the function $i$ be defined as the following. If $A = F_{i_1}\circ ... \circ F_{i_m}(I)$ is a cell on level $m$,
\begin{equation}
i(A) = \#\{i_j | i_j = 0, 3\}
\end{equation}
i.e. $i(A)$ is the number of ``outside" choices during contraction to cell $A = I_k^{(m)}$. Then
\begin{equation}
\mu(A) = \left(\frac{p}{2}\right)^{i(A)}\left(\frac{1-p}{2}\right)^{m-i(A)} = \frac{p^{i(A)}(1-p)^{m-i(A)}}{2^{m}}.
\end{equation}

Now we assign a point mass $\mu_{k}^{(m)}$ to each point $x=\frac{k}{4^{m}}$. By averaging the measures of the intervals that meet at $x$ we get 
\begin{equation}
\mu_{k}^{(m)}=\frac{\mu(I_{k}^{(m)})+\mu(I_{k-1}^{(m)})}{2}.
\end{equation}
We will approximate $\int f d\mu$ by $\sum_{k=1}^{4^{m}-1} f(\frac{k}{4^{m}})\mu_{k}^{(m)}$.

Next choose another parameter $q$ with  $0 < q < 1$ and define resistance $R(I_{k}^{(m)})$ in the same manner as the measure
\begin{equation}
R(I_{k}^{(m)})=\frac{q^{i(k)}(1-q)^{m-i(k)}}{2^{m}}.
\end{equation} 

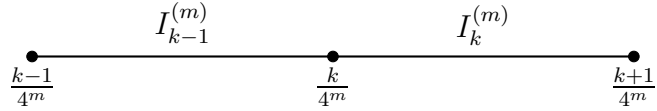
\begin{figure}
    \centering
    \begin{tikzpicture}
    \draw[black, thick] (0,0) -- (8,0);
    \filldraw[black] (0,0) circle (2pt) node[anchor=north] {$\frac{k-1}{4^{m}}$};
    \filldraw[black] (2,0) circle (0pt) node[anchor=south] {$I_{k-1}^{(m)}$};
    \filldraw[black] (4,0) circle (2pt) node[anchor=north] {$\frac{k}{4^{m}}$};
    \filldraw[black] (6,0) circle (0pt) node[anchor=south] {$I_{k}^{(m)}$};
    \filldraw[black] (8,0) circle (2pt) node[anchor=north] {$\frac{k+1}{4^{m}}$};
    \end{tikzpicture}
    \caption{Two adjacent $m$-cells}
    \label{fig:adjmcells}
\end{figure}

\begin{figure}
    \centering
    \begin{tikzpicture}
    \draw[black, thick] (0,0) -- (8,0);
    \filldraw[black] (0,0) circle (2pt) node[anchor=west]{};
    \filldraw[black] (1,0) circle (0pt) node[anchor=south]{$\frac{p}{2}$};
    \filldraw[black] (1,0) circle (0pt) node[anchor=north]{$\frac{q}{2}$};
    \filldraw[black] (2,0) circle (2pt) node[anchor=west]{};
    \filldraw[black] (3,0) circle (0pt) node[anchor=south]{$\frac{1-p}{2}$};
    \filldraw[black] (3,0) circle (0pt) node[anchor=north]{$\frac{1-q}{2}$};
    \filldraw[black] (4,0) circle (2pt) node[anchor=north] {};
    \filldraw[black] (5,0) circle (0pt) node[anchor=south]{$\frac{1-p}{2}$};
    \filldraw[black] (5,0) circle (0pt) node[anchor=north]{$\frac{1-q}{2}$};
    \filldraw[black] (6,0) circle (2pt) node[anchor=north]{}; 
    \filldraw[black] (7,0) circle (0pt) node[anchor=south]{$\frac{p}{2}$};
    \filldraw[black] (7,0) circle (0pt) node[anchor=north]{$\frac{q}{2}$};
    \filldraw[black] (8,0) circle (2pt) node[anchor=west] {};
    \end{tikzpicture}
    \caption{Level 1}
    \label{fig:intlev1}
\end{figure}

The conductances are the reciprocals of resistance, namely
\begin{equation}
c(I_{k}^{(m)})=\frac{1}{R(I_{k}^{(m)})}.
\end{equation}

The energy is defined
\begin{align}
\mathcal{E}_{m}(f)&=\sum_{k=0}^{4^{m}-1} c(I_{k}^{(m)})\left(f\left(\frac{k}{4^{m}}\right)-f\left(\frac{k+1}{4^m}\right)\right)^{2}, \\
\mathcal{E}(f) &= \lim_{m\to\infty} \mathcal{E}_m.
\end{align}

With this definition of energy, we have the weak formulation of the Laplacian:
\begin{equation}
\mathcal{E}(u,v) = -\int fv d\mu
\end{equation}
where $u, v \in \text{dom}\mathcal{E}$ and $f = \Delta^{(p)}u$, the Laplacian with parameter $p$. In addition, the pointwise Laplacian is given by 
\begin{multline}
-\Delta_{m}^{(p)}f\left(\frac{k}{4^{m}}\right)=\frac{1}{\mu_{k}^{m}}
\biggr[
c\left(I_{k}^{(m)}\right)
\left(f\left(\frac{k}{4^{m}}\right)-f\left(\frac{k+1}{4^{m}}\right)\right)\\
+c\left(I_{k-1}^{(m)}\right)
\left(f\left(\frac{k}{4^{m}}\right)-f\left(\frac{k-1}{4^{m}}\right)\right)
\biggr].
\end{multline}
This is a weighted average of the changes in $f$ over the two intervals intersecting at $\frac{k}{4^{m}}$. Let $A_0 = [\frac{k-1}{4^m}, \frac{k}{4^m}]$ and $A_1 = [\frac{k}{4^m}, \frac{k+1}{4^m}]$, the $m$-cells containing $\frac{k}{4^m}$. Then,
\begin{multline}\label{intlap}
-\Delta_{m}^{(p)}f\left(\frac{k}{4^{m}}\right)=\frac{2 \cdot 4^{m}}{p^{i(A_1)}(1-p)^{m-i(A_1)}+p^{i(A_0)}(1-p)^{m-i(A_0)}} \cdot \\ \biggr[\frac{1}{q^{i(A_1)}(1-q)^{m-i(A_1)}}\left(f\left(\frac{k}{4^{m}}\right)-f\left(\frac{k+1}{4^{m}}\right)\right)\\+\frac{1}{q^{i(A_0)}(1-q)^{m-i(A_0)}}\left(f\left(\frac{k}{4^{m}}\right)-f\left(\frac{k-1}{4^{m}}\right)\right)\biggr].
\end{multline}

Note that the Laplacian is renormalized by $\frac{p}{2} \cdot \frac{q}{2}$ or $\frac{1-p}{2} \cdot \frac{1-q}{2}$ depending on the location within the Interval. In order for the renormalization factor to be constant across the Interval, we need $\frac{p}{2} \cdot \frac{q}{2}=\frac{1-p}{2} \cdot \frac{1-q}{2}$, or $p+q=1$. We will be using this property throughout this paper for the Interval model.

We will impose Dirichlet boundary conditions, namely $f(0)=0$ and $f(\frac{4^{m}}{4^{m}})=0$. Then $-\Delta_{m}^{(p)}$ can be represented as a self adjoint (with respect to the pointmasses) matrix of size $(4^{m}-1) \times (4^{m}-1)$, which has $4^{m}-1$ eigenvectors with positive eigenvalues. In other words, 
\begin{equation}
-\Delta_{m}^{(p)}f\left(\frac{k}{4^{m}}\right)=\lambda f\left(\frac{k}{4^{m}}\right).
\end{equation}

We will abbreviate $\Delta_{m}^{(p)}$ to $\Delta_m$ throughout the rest of the text when the choice of parameter $p$ is clear.

We wish to study the continuous eigenfunctions and eigenvalues of the Laplacian as the limit of the discrete eigenfunctions and eigenvalues given by the graph approximations of the unit Interval.

\section{Spectral Decimation on the Interval} \label{intdec}

Our aim is to replicate the spectral decimation on the standard Interval model for our twice-iterated Interval model. Here we start by simplifying the pointwise Laplacian formula (\ref{intlap}). Letting $y_1 < z < y_2 \in V_m$, as diagrammed in Figure \ref{fig:intmodel} and $p+q=1$, there are three cases for pointwise Laplacian, 
\begin{align}\label{simpintlap}
-\Delta_m f(z) &= \left(\frac{4}{pq}\right)^m (2f(z) - f(y_1) - f(y_2)) &&\text{if }i([z, y_2]) = i([y_1, z]), \\
-\Delta_m f(z) &= \left(\frac{4}{pq}\right)^m (2f(z) - 2qf(y_1) - 2pf(y_2)) &&\text{if } i([z, y_2]) = i([y_1, z])+1, \\
-\Delta_m f(z) &= \left(\frac{4}{pq}\right)^m (2f(z) - 2pf(y_1) - 2qf(y_2)) &&\text{if } i([z, y_2]) = i([y_1, z])-1.
\end{align}

We know that the cases are exhaustive by the following lemma:
\begin{lem} \label{i(k)lemma}
Let $x\in V_{m}$. Then $|i(A_{0})-i(A_{1})|\leq1$ for $A_{0},A_{1}$ the two cells with junction point $x$.
\end{lem}
\begin{proof}
Proof by induction:

Base case: $V_{1}$ satisfies this property.

Inductive Step: Given that the claim holds on level $V_{n}$, the extension to level $V_{n+1}$ involves subdividing each $n$-cell via the process defined in section 2. Letting $x\in V_{n+1}$, there are two cases.
\begin{itemize}
    \item[(a)] $x\in V_{n}$. Let $A_{0},A_{1}$ be the two $n$-cells with $x$ as their junction point. By inductive hypothesis, $|i(A_{0})-i(A_{1})|\leq1$. By the design of the subdivision process, the two ($n+1$)-cells with $x$ as junction point in $V_{n+1}$ are $F_{j}A_{0}$ and $F_{k}A_{1}$ with $j\neq k$. Then $i(F_{j}A_{0}) = i(A_{0})+1$ and $i(F_{k}A_{1}) = i(A_{1})+1$, so $|i(F_{j}A_{0})-i(F_{k}A_{1})|\leq1$ and the claim holds on $V_{n+1}$.
    
    \item[(b)] $x \notin V_{n}$. Therefore $x$ must be in a subdivided $n$-cell, $A$. The two $(n+1)$-cells with junction point $x$ are $F_{j}A$ and $F_{k}A$ by the subdivision scheme. Since $i$ is additive over words, ie $i(F_{j}A) = i(F_{j}I)+i(A)$, computing $i$ on these cells yields $i(F_{j}I)+i(A)$ and $i(F_{k}I)+i(A)$. The value of $i(F_{j}I)$ must either be 0 ($i=0,3$) or 1 ($i=1,2$), and similarly for $i(F_{k}I)$. Then $|(i(F_{j}I)+i(A))-(i(F_{k}I)+i(A))| = |i(F_{j}I)-i(F_{k}I)| \leq 1$ and so the claim holds on $V_{n+1}$.
\end{itemize}
\end{proof}

For any given eigenvalues $\lambda_{m}$ and eigenfunctions $f_{m}(x)$ of the Laplacian on level $m$, we want to be able to extend the eigenfunctions to level $m+1$, as well as give a new eigenvalue $\lambda_{m+1}$ such that the following equation holds

\begin{equation}
-\Delta_{m+1} f_{m+1}(x) = \lambda_{m+1}f_{m+1}(x)\quad \forall \ x\in V_{m+1}\setminus V_0. \label{eqn:eigint}
\end{equation}

For now, we will omit the renormalization factor $\left(\frac{4}{pq}\right)^m$, but we will rescale our eigenvalues later.  

As shown in Figure \ref{fig:intmodel}, $x_{1}$, $x_{2}$ are our points on the previous level, and $y_{1}$, $y_{2}$ and $z$ are new points that are born on the next level. Note that we will drop the $f$ in front of the variables from here for the sake of simplicity. Evaluating \ref{eqn:eigint} on the three new points gives us three equations and six variables, meaning we can solve for the values at $y_{1}$, $y_{2}$ and $z$ in terms of $x_{1}$, $x_{2}, \lambda_{m+1}$ and $p$ (where $q = 1-p$ as usual). Solutions for these values are as follows. 

\begin{align}
y_1(x_1, x_2, \lambda_{m+1}, p) &= \frac{-4pqx_2-2px_1(2+2p-4\lambda_{m+1}+\lambda_{m+1}^2)}{(-4q+(\lambda_{m+1}-2)^2)(\lambda_{m+1}-2)} \nonumber\\
z(x_1, x_2, \lambda_{m+1}, p) &= \frac{2p(x_1+x_2)}{(\lambda_{m+1}-2)^2-4q} \nonumber\\
y_2(x_1, x_2, \lambda_{m+1}, p) &= \frac{-4pqx_1 - 2px_2(2+2p-4\lambda_{m+1}+\lambda_{m+1}^2)}{(-4q+(\lambda_{m+1}-2)^2)(\lambda_{m+1}-2)} \label{eqn:intdecfunc}
\end{align}

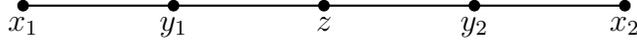
\begin{figure}
    \centering
    \begin{tikzpicture}
    \draw[black, thick] (0,0) -- (8,0);
    \filldraw[black] (0,0) circle (2pt) node[anchor=north] {$x_{1}$};
    \filldraw[black] (2,0) circle (2pt) node[anchor=north] {$y_{1}$};
    \filldraw[black] (4,0) circle (2pt) node[anchor=north] {$z$};
    \filldraw[black] (6,0) circle (2pt) node[anchor=north] {$y_{2}$};
    \filldraw[black] (8,0) circle (2pt) node[anchor=north] {$x_{2}$};
    \end{tikzpicture}
    \caption{A single $m$-cell}
    \label{fig:intmodel}
\end{figure}

In order for the equations above to hold we need the denominators to be nonzero, and thus $\lambda = 2(1-\sqrt{q}), 2(1+\sqrt{q}), 2$ are our \textit{forbidden eigenvalues}.

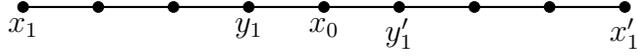
\begin{figure}
    \centering
    \begin{tikzpicture}
    \draw[black, thick] (0,0) -- (8,0);
    \filldraw[black] (0,0) circle (2pt) node[anchor=north] {$x_1$};
    \filldraw[black] (1,0) circle (2pt) node[anchor=north] {};
    \filldraw[black] (2,0) circle (2pt) node[anchor=north] {};
    \filldraw[black] (3,0) circle (2pt) node[anchor=north] {$y_1$};
    \filldraw[black] (4,0) circle (2pt) node[anchor=north] {$x_0$};
    \filldraw[black] (5,0) circle (2pt) node[anchor=north] {$y_1'$};
    \filldraw[black] (6,0) circle (2pt) node[anchor=north] {};
    \filldraw[black] (7,0) circle (2pt) node[anchor=north] {};
    \filldraw[black] (8,0) circle (2pt) node[anchor=north] {$x_1'$};
    \end{tikzpicture}
    \caption{The $m$-cells around $x_0$}
    \label{fig:intextmodel}
\end{figure}

We must verify that this extension is still an eigenfunction at $x_0 \in V_m$ with the new eigenvalue $\lambda_{m+1}$. We must consider the two neighboring $m$-cells around $x_0$ as in Figure \ref{fig:intextmodel}. Although there are three different cases depending on the two $m$-cells around $x_0$, the algebraic result is the same for all three cases. In the simplest case, we are given

\begin{equation}
\lambda_m x_0 = \left(\frac{4}{pq}\right)^m(2x_0 - x_1-x_1')
\end{equation}

by the $m$-level eigenfunction equation, and want to verify

\begin{equation}
\lambda_{m+1} x_0 = \left(\frac{4}{pq}\right)^{m+1}(2x_0 - y_2(x_1, x_0, \lambda_{m+1}, p)-y_1(x_0, x_1', \lambda_{m+1}, p))
\end{equation}

These two conditions yield $\lambda_{m}$ as a quartic function of $\lambda_{m+1}$ and $p$.

\begin{equation}
\lambda_{m}(\lambda_{m+1}, p) = \frac{(4-\lambda_{m+1})(\lambda_{m+1}-2)^2\lambda_{m+1}}{4pq}
\end{equation}

All eigenvalues need to be scaled by a factor of $\left(\frac{4}{pq}\right)^{m+1}$. Since $p$ and $q$ are interchangeable in the quartic equation and the scaling factor, we will see a pattern of equal eigenvalues when we interchange $p$ and $q$ in section 4. 

Now we want to determine the eigenfunctions and eigenvalues that are born on each level. We can describe all the eigenfunctions that are born with the following proposition and corollary. 

\begin{prop}\label{squishprop}
If $f_{m}$ is a Dirichlet eigenfunction with eigenvalue $\left(\frac{4}{pq}\right)^{m}\lambda$ on level $m$, then $f_{m+1}$ defined the following way is an eigenfunction with eigenvalue $\left(\frac{4}{pq}\right)^{m+1}\lambda$ on level $m+1$: \\
If $f_{m}$ is skew symmetric about $x = \frac{1}{2}$,
\begin{equation*}
f_{m+1}(x)=
\begin{cases}
f_{m}\circ F_0^{-1}(x) &\text{if }x\in F_0(I)\\
\frac{p}{q}f_{m}\circ F_1^{-1}(x)&\text{if }x\in F_1(I)\\
\frac{p}{q}f_{m}\circ F_2^{-1}(x)&\text{if }x\in F_2(I)\\
f_{m}\circ F_3^{-1}(x) &\text{if }x\in F_3(I)
\end{cases}
\end{equation*}
If $f_{m}$ is symmetric about $x = \frac{1}{2}$,
\begin{equation*}
f_{m+1}(x)=
\begin{cases}
f_{m}\circ F_0^{-1}(x) &\text{if }x\in F_0(I)\\
-\frac{p}{q}f_{m}\circ F_1^{-1}(x)&\text{if }x\in F_1(I)\\
\frac{p}{q}f_{m}\circ F_2^{-1}(x)&\text{if }x\in F_2(I)\\
-f_{m}\circ F_3^{-1}(x) &\text{if }x\in F_3(I)
\end{cases}
\end{equation*}
\end{prop}
 
\begin{proof}
First, we note that the two cases given (symmetric or skew symmetric) are exhaustive, since we see that by the eigenfunction extension equations from decimation, symmetric (or skew symmetric) functions will be extended so that the extension is symmetric or skew symmetric, and all eigenfunctions born also satisfy that condition.

In both cases, $f_{m+1}$ is (locally) skew symmetric about $y \in V_1$ for all $m$. For all $y \in V_{m+1}$, we want to verify that $f_{m+1}(y) = \left(\frac{4}{pq}\right)^{m+1}\lambda_{m}$. There are two cases for $y\in V_{m+1}$: $y\not\in V_1$ or $y\in V_1$.

\begin{itemize}

\item[(a)] Consider $y\not\in V_1$ and let $f_m$ be an eigenfunction on level $m$ with eigenvalue $\left(\frac{4}{pq}\right)^{m}\lambda_m$. For all such $y$, there is some $i$ and $x\neq 0, 1$ such that $y = F_i(x)$. If $i=0, 3$,
\begin{align*}
\Delta_{m+1} f_{m+1}(F_i(x)) &= \frac{4}{pq}\Delta_{m+1}(f_{m+1}\circ F_i)(x) &\text{by the scaling law}\\
&=\frac{4}{pq}\Delta_{m}(\pm f_{m}(x)) &\text{by definition of }f_{m+1}\\
&= \left(\frac{4}{pq}\right)\left(\frac{4}{pq}\right)^{m}(-\lambda) (\pm f_{m}(x)) &\text{by assumption}\\
&= \left(\frac{4}{pq}\right)^{m+1}(-\lambda) f_{m+1}(F_i(x)) &\text{by definition of }f_{m+1}
\end{align*}
A similar computation holds for $i= 1, 2$.

\item[(b)] Consider $y \in V_1$ and let $f_m$ be an eigenfunction on level $m$ with eigenvalue $\left(\frac{4}{pq}\right)^{m}\lambda_m$ as before. Recall the pointwise Laplacian formula (\ref{simpintlap}) and consider $y = \frac{1}{4}$ and its neighbors, $y_1$ and $y_2$ (where $y_1 < y < y_2$).
\begin{align*}
-\Delta_{m+1} f_{m+1}\left(\frac{1}{4}\right) &= \left(\frac{4}{pq}\right)^{m+1} \left(2f_{m+1}\left(\frac{1}{4}\right) - 2pf_{m+1}(y_1) - 2qf_{m+1}(y_2)\right) \\
&= \left(\frac{4}{pq}\right)^{m+1} \left(2f_{m+1}\left(\frac{1}{4}\right) - 2pf_{m+1}(y_1) + 2q\left(\frac{p}{q}\right)f_{m+1}(y_1)\right) \\
&= \left(\frac{4}{pq}\right)^{m+1} 2f_{m+1}\left(\frac{1}{4}\right) \\
&= 0
\end{align*}
Since we know that $f_{m+1}(y_2) = -\frac{p}{q}f_{m+1}(y_1)$ by skew symmetry, and $f_{m+1}(\frac{1}{4}) = 0$ by the definition of $f_{m+1}$. Similarly, the eigenvalue equation holds at $x = \frac{3}{4}$.
At $x = \frac{1}{2}$, since by definition, $f_{m+1}$ is skew symmetric about $x = \frac{1}{2}$ and $f_{m+1}(\frac{1}{2}) = 0$, the equation holds as well. 
\end{itemize}
\end{proof}

\begin{cor}
Eigenvalues $2(1-\sqrt{q})$, $2$, $2(1+\sqrt{q})$, scaled appropriately, are born on each level $m$.\label{intborn}
\end{cor} 
\begin{proof}
To confirm that $2(1-\sqrt{q})$ and $2(1+\sqrt{q})$ never arise from as a solution to the quartic from decimation, we see that both $\lambda_m = 2(1-\sqrt{q})$ and $\lambda_m =2(1+\sqrt{q})$ give $\lambda_{m-1} = 4$, meaning that $\lambda_{m-2} = 0$. $\lambda_m=2$ gives $\lambda_{m-1} = 0$. Therefore $2(1-\sqrt{q})$ and $2(1+\sqrt{q})$ are forbidden eigenvalues (at least for Dirichlet eigenfunctions).

Proof by induction on $m$. Base case: We give three eigenfunctions for level $m=1$.
\begin{align*}
g_1(x) = \begin{cases}
0 & x=0 \\
\sqrt{q} & x=\frac{1}{4} \\
1 & x=\frac{1}{2}\\
\sqrt{q} & x=\frac{3}{4} \\
0 & x=1
\end{cases}
\quad
g_2(x) = \begin{cases}
0 & x=0 \\
1 & x=\frac{1}{4} \\
0 & x=\frac{1}{2} \\
-1 & x=\frac{3}{4} \\
0 & x=1
\end{cases}
\quad
g_3(x) = \begin{cases}
0 & x=0 \\
\sqrt{q} & x=\frac{1}{4} \\
-1 & x=\frac{1}{2}\\
\sqrt{q} & x=\frac{3}{4} \\
0 & x=1
\end{cases}
\end{align*}
By the proposition above, we can construct new eigenfunctions on any level $m$ for which the eigenvalues stay the same, i.e. $2(1-\sqrt{q}), 2, 2(1+\sqrt{q})$, scaled up, so these are the eigenvalues that are born at every level.
\end{proof}

As usual, we finish with a counting argument to ensure that all eigenvalues and eigenfunctions have been accounted for. On level $m+1$, we should have $4^{m+1}-1$ pairs of eigenvalues and eigenfunctions since $\#(V_{m+1} \setminus V_{0}) = 4^{m+1}-1$. From decimating the $4^{m}-1$ eigenfunctions from level $m$, we obtain $4(4^{m}-1) = 4^{m+1}-4$ new eigenfunctions. On each level, we account for $3$ new eigenvalues and functions that are born. Adding these give $4^{m+1}-4+3 = 4^{m+1}-1 = \#V_{m+1}$, meaning that this process accounts for all eigenfunctions and values on level $m+1$.

These are the four maps that take $\lambda_{m}$ to four values of $\lambda_{m+1}$.

\begin{align}
\Phi_1(\lambda_{m}) &= 2-\sqrt{2+2\sqrt{1-p\lambda_{m} +p^2\lambda_{m}}} \nonumber\\
\Phi_2(\lambda_{m}) &= 2-\sqrt{2-2\sqrt{1-p\lambda_{m} +p^2\lambda_{m}}} \nonumber\\
\Phi_3(\lambda_{m}) &= 2+\sqrt{2-2\sqrt{1-p\lambda_{m} +p^2\lambda_{m}}} \nonumber\\
\Phi_4(\lambda_{m}) &= 2+\sqrt{2+2\sqrt{1-p\lambda_{m} +p^2\lambda_{m}}} \label{eqn:intdecval}
\end{align}

In the following proposition, we study how the eigenvalues on level $m+1$ are ordered.

\begin{figure}
\centering
  \includegraphics[width=0.75\textwidth]{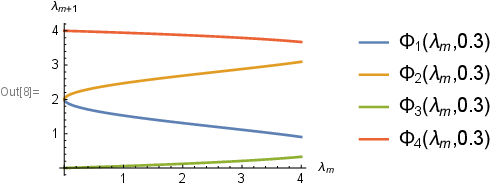}
  \caption{$\Phi_1(\lambda), \Phi_2(\lambda), \Phi_3(\lambda), \Phi_4(\lambda)$ for $p=0.3$}
  \label{fig:phigraph}
\end{figure}

\begin{prop}\label{phiorder}
Let $s$ be the number of Dirichlet eigenvalues on level $m$, given by $s=\#(V_{m} \setminus V_{0})=4^{m}-1$. Then the sequence of eigenvalues on level $m+1$ in strictly increasing order is the following:
\begin{multline}
\Phi_1(\lambda_{1, p}^{(m)}) , ..., \Phi_1(\lambda_{s, p}^{(m)}), 2(1-\sqrt{q}), \Phi_2(\lambda_{s, p}^{(m)}), ..., \Phi_2(\lambda_{1, p}^{(m)}), 2, \\
\Phi_3(\lambda_{1, p}^{(m)}), ..., \Phi_3(\lambda_{s, p}^{(m)}), 2(1+\sqrt{q}), \Phi_4(\lambda_{s, p}^{(m)}), ..., \Phi_4(\lambda_{1, p}^{(m)})
\end{multline}
Furthermore, all eigenvalues have multiplicity 1.
\end{prop}
\begin{proof}
We see that the domain of these functions are $\lambda \in (0, 4)$ (since the maximum that can be attained from applying the functions on each level approaches $4$), and since $\Phi_1, \Phi_3$ are increasing and $\Phi_2, \Phi_4$ are decreasing, the minimum for $\Phi_1, \Phi_3$ and the maximum for $\Phi_2, \Phi_4$ occur at $\lambda = \frac{1}{p(1-p)}$. Also, $\Phi_4 \geq \Phi_3 \geq \Phi_2 \geq \Phi_1$, i.e. for all $x, y$ in the domain, $\Phi_4(x) \geq \Phi_3(y)$, etc. See Figure \ref{fig:phigraph} for a graph showing this in the case of $p=0.3$.

Inductively, since on level 1, the eigenvalues $2, 2(1-\sqrt{q}), 2(1+\sqrt{q})$ all have multiplicity 1, all eigenvalues generated by the above 4 equations on level $m$ will also have multiplicity 1, and we need only to check where the eigenvalues born on level $m$ fall into the spectrum. First, for $\lambda = 2$: we can see that since $\Phi_3(0) = \Phi_2(0) = 2$, and since $\lambda_{m-1} \neq 0$, we see that $2$ has multiplicity 1, and that $\Phi_3 > 2 > \Phi_2$. Similarly, consider $\Phi_2, \Phi_1,$ and $2(1-\sqrt{q})$. If $2p-1 > 0$, since $\Phi_2$ is decreasing, consider $\Phi_2(\max\Phi_i) = \Phi_3(4) = 2-\sqrt{2-2\sqrt{1-4p +4p^2}} = 2-\sqrt{2-2(2p-1)} = 2-\sqrt{4-4p} = 2(1-\sqrt{q})$, and since $\lambda_{m-1}=4 \implies \lambda_{m-2} = 0$, this means that $2(1-\sqrt{q}) < \min\Phi_2$. Otherwise, if $2p-1 < 0$, we see that $2(1-\sqrt{q}) > \max\Phi_1$ by computations similar to those before. Similar results can also be shown for $2(1+\sqrt{q})$. Finally,
\begin{align}
\Phi_1(\lambda) <  2(1-\sqrt{q}) < \Phi_2(\lambda) < 2 < \Phi_3(\lambda) < 2(1+\sqrt{q}) < \Phi_4(\lambda)
\end{align}
for all relevant $\lambda$. We know that $\Phi_1, \Phi_3$ are increasing and $\Phi_2, \Phi_4$ are decreasing, so if $\{\lambda_{n, p}^{(m)}\}$ denote an ordered sequence of eigenvalues on level $m$ for a particular choice of $p$, the sequence for level $m+1$ would look like this:
\begin{multline}
\Phi_1(\lambda_{1, p}^{(m)}) , ..., \Phi_1(\lambda_{s, p}^{(m)}), 2(1-\sqrt{q}), \Phi_2(\lambda_{s, p}^{(m)}), ..., \Phi_2(\lambda_{1, p}^{(m)}), 2, \\
\Phi_3(\lambda_{1, p}^{(m)}), ..., \Phi_3(\lambda_{s, p}^{(m)}), 2(1+\sqrt{q}), \Phi_4(\lambda_{s, p}^{(m)}), ..., \Phi_4(\lambda_{1, p}^{(m)})
\end{multline}
where $s$ is the number of eigenvalues on level $m$.
\end{proof}


We would like to define
\begin{equation}
\lambda = \lim_{m \to \infty} \left(\frac{4}{pq}\right)^{m}\lambda_{m}
\end{equation}

with $\lambda_{m}$ a sequence defined by repeated application of the $\Phi$ mappings, and all but a finite number $\Phi_{1}$. Expressing $\Phi_{1}$ in Taylor Series form yields

\begin{equation}
\Phi_{1}(x) = \frac{4}{pq}x+O(x^{2})
\end{equation}

and so as $\lambda \to 0$, the higher order terms will fall away, causing $\lambda_{m} = O(\big(\frac{pq}{4}\big)^{m})$ as $m \to \infty$. Then the limit defined above clearly exists.

With the above information, we may state the following theorem summarizing these results. The proof lies in the work shown above.

\begin{thm}[Interval Spectral Decimation]

For any $p$, given $u_{m}$, an eigenfunction with eigenvalue $\lambda_{m}$ on $V_{m}$, we may choose $\lambda_{m+1}$ as one of the values given in \ref{eqn:intdecval}, assuming that $\lambda_{m+1} \neq 2,2(1\pm\sqrt{q})$. We can then extend $u_{m}$ to $V_{m+1}$ according to \ref{eqn:intdecfunc} to obtain an eigenfunction on level $m+1$. Using Corollary \ref{intborn} and counting arguments, this process produces a complete spectra on level $m+1$.

\end{thm}

\section{Data on the Interval}
In this section, we will present the experimental data produced for the Interval for $p = 0.1, 0.9$.
\subsection{Eigenvalues and Eigenfunctions}

Below we have two tables of eigenvalues of the Dirichlet Laplacian on $I$ the first three levels of graph approximation. The table on the left presents the eigenvalues for $p=0.1$, and the table on the right presents the eigenvalues for $p=0.9$.

\begin{center}
\tiny
    \begin{tabular}{l|l|l|l}%
    $n$ & $m = 1$ & $m = 2$ & $m = 3$ 
    \begin{filecontents*}{intp01.csv}
,,,
1,4.561485,4.574716,4.575014
2,88.88889,94.42371,94.55094
3,173.2163,196.7487,197.3024
4,,202.7326,203.3207
5,,2719.420,2832.159
6,,3092.044,3239.248
7,,3760.552,3982.276
8,,3950.617,4196.609
9,,4140.683,4412.348
10,,4809.191,5182.667
11,,5181.815,5620.117
12,,7698.502,8744.386
13,,7704.486,8752.202
14,,7806.811,8886.180
15,,7896.660,9004.313
16,,,9010.340
17,,,120076.9
18,,,120432.6
19,,,120839.3
20,,,120863.1
21,,,131514.8
22,,,133237.8
23,,,136467.8
24,,,137424.2
25,,,138399.9
26,,,142011.9
27,,,144173.8
28,,,167135.6
29,,,167261.5
30,,,169821.6
31,,,174315.5
32,,,175583.0
33,,,176850.5
34,,,181344.4
35,,,183904.5
36,,,184030.4
37,,,206992.2
38,,,209154.1
39,,,212766.1
40,,,213741.8
41,,,214698.2
42,,,217928.2
43,,,219651.2
44,,,230302.9
45,,,230326.7
46,,,230733.3
47,,,231089.1
48,,,342155.6
49,,,342161.7
50,,,342279.8
51,,,342413.8
52,,,342421.6
53,,,345545.9
54,,,345983.3
55,,,346753.6
56,,,346969.4
57,,,347183.7
58,,,347926.7
59,,,348333.8
60,,,350962.7
61,,,350968.7
62,,,351071.4
63,,,351161.4
\end{filecontents*}
\csvreader[head to column names]{intp01.csv}{}
    {\\\hline\csvcoli&\csvcolii&\csvcoliii&\csvcoliv}
    \end{tabular}
    \quad
    \begin{tabular}{l|l|l|l}%
    $n$ & $m = 1$ & $m = 2$ & $m = 3$ 
    \begin{filecontents*}{intp09.csv}
,,,
1,60.77975,63.28154,63.33865
2,88.88889,94.42371,94.55094
3,116.9980,126.9341,127.1642
4,,2701.322,2812.513
5,,2957.227,3091.390
6,,3092.044,3239.248
7,,3246.346,3409.282
8,,3950.617,4196.609
9,,4654.888,5003.240
10,,4809.191,5182.667
11,,4944.007,5340.254
12,,5199.912,5641.515
13,,7774.300,8843.549
14,,7806.811,8886.180
15,,7837.953,8927.074
16,,,120058.8
17,,,120309.2
18,,,120432.6
19,,,120561.7
20,,,131432.3
21,,,132608.6
22,,,133237.8
23,,,133966.5
24,,,137424.2
25,,,141150.4
26,,,142011.9
27,,,142780.0
28,,,144282.1
29,,,168901.7
30,,,169821.6
31,,,170867.2
32,,,175583.0
33,,,180298.7
34,,,181344.4
35,,,182264.3
36,,,206883.9
37,,,208386.0
38,,,209154.1
39,,,210015.5
40,,,213741.8
41,,,217199.5
42,,,217928.2
43,,,218557.4
44,,,219733.7
45,,,230604.3
46,,,230733.3
47,,,230856.8
48,,,231107.2
49,,,342238.9
50,,,342279.8
51,,,342322.4
52,,,345524.5
53,,,345825.7
54,,,345983.3
55,,,346162.7
56,,,346969.4
57,,,347756.7
58,,,347926.7
59,,,348074.6
60,,,348353.5
61,,,351038.8
62,,,351071.4
63,,,351102.6
\end{filecontents*}
\csvreader[head to column names]{intp09.csv}{}
    {\\\hline\csvcoli&\csvcolii&\csvcoliii&\csvcoliv}
    \end{tabular}
    \captionof{table}{Eigenvalues of the Laplacian on the Interval. Left: $p=0.1$, Right: $p=0.9$}
    \label{intervaleigval}
\end{center}

Notice that, at each level, there exist identical eigenvalues for $p=0.1$ and $p=0.9$, and the locations ($n$) of identical eigenvalues are $8k-2$, $8k+2$, $8(2k+1)$, $8(8k+4)$ with periodic period 64. This pattern occurs for all interchanged pairs of $p$ and $q$ as far as the property $p+q=1$ is preserved. We give a generalized description of this pattern in the following corollary of Proposition \ref{phiorder}.

\begin{cor}
Let $\lambda_{n,p}^{(m)}$ denote the $n$th eigenvalue of $\Delta_m^{(p)}$ and let $p+q=1$. Then $\lambda_{n, p}^{(m)} = \lambda_{n, q}^{(m)}$ if $n \equiv \frac{4^a}{2} \mod 4^a$ for some $a \leq m$.
\end{cor}
\begin{proof}
By induction: Base case is trivial on level $1$, where $\lambda_{2, p}^{(1)} = \lambda_{2, q}^{(1)}$. Assume the claim holds on level $m-1$. If $n\equiv \frac{4^a}{2} \mod 4^a$ where $a < m$, it is easy to see that this will still hold in level $m$ for all $a < m$ by the way that the four mappings above act on the eigenvalues on level $m-1$. If $i=1, 3$,
\begin{align*}
\Phi_i(\lambda_{n, p}^{(m-1)}) &= \lambda_{n+(i-1)(s+1), p}^{(m)} \\
&= \lambda_{n+(i-1)(s+1), q}^{(m)}
\end{align*}
and if $i=2, 4$,
\begin{align*}
\Phi_i(\lambda_{n, p}^{(m-1)}) &= \lambda_{-n+i(s+1), p}^{(m)} \\
&= \lambda_{-n+i(s+1), q}^{(m)}
\end{align*}
where $s$ denotes the number of eigenvalues on level $m-1$. By assumption, $n+(i-1)(s+1) \equiv \frac{4^a}{2} \mod 4^a$ and $-n+i(s+1) \equiv \frac{4^a}{2} \mod 4^a$ where $a < m$, since $s = 4^{m-1}-1$.

Now we only need to account for the new eigenvalue born on level $m$, which will occur at the middle of the spectrum. This is  $\lambda_{\frac{(4^m-1)+1}{2}, p}$, which is equal to $\lambda_{\frac{(4^m-1)+1}{2}, q}$ since both are equal to $2\left(\frac{4}{pq}\right)^m$. Since this is $\frac{4^m}{2} \mod 4^m$, we see that the corollary holds for all $a \leq m$.
\end{proof}

The graph of $\Phi_i$ in Figure \ref{fig:phigraph} shows that there are gaps in the mappings of eigenvalues from level $m$ to $m+1$ at $(\Phi_1(4), \Phi_2(4)) \cup (\Phi_3(4), \Phi_4(4))$, and we know that no eigenvalues are born in those intervals at any level. These gaps can be observed in the experimental data as well.

\begin{figure}
\centering
\includegraphics[width=0.75\textwidth]{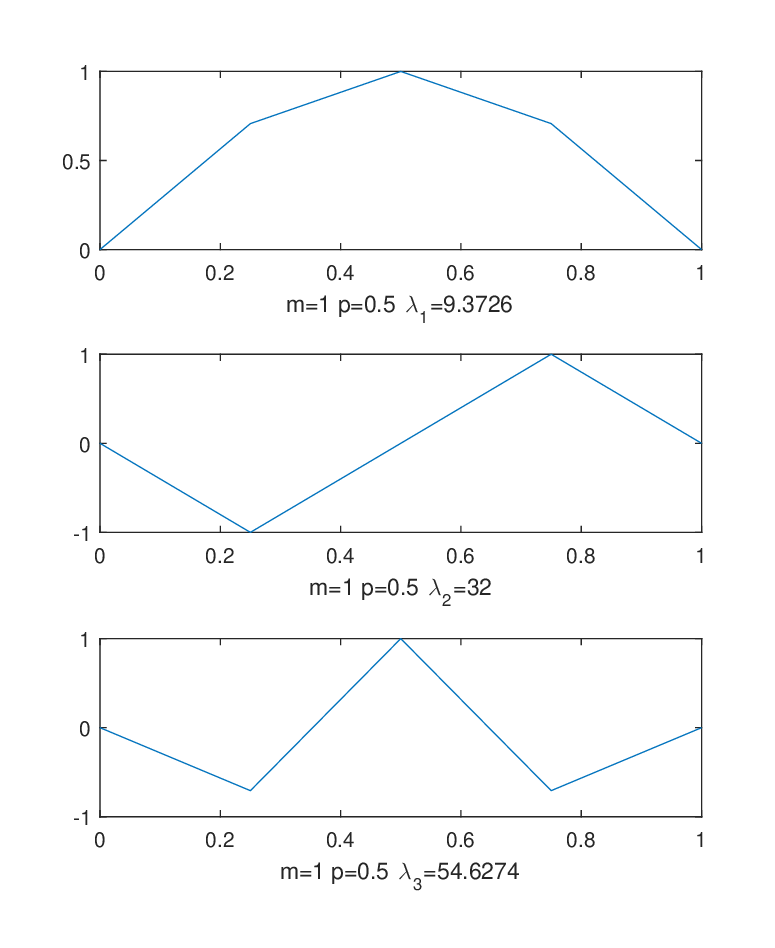}
\caption{First 3 Dirichlet Eigenfunctions on the Interval on level $m=1$. $p=0.5$ (standard case, appx. Sine)}
\label{fig:intervaleigfunction}
\end{figure}

\begin{figure}
\centering
\includegraphics[width=\textwidth]{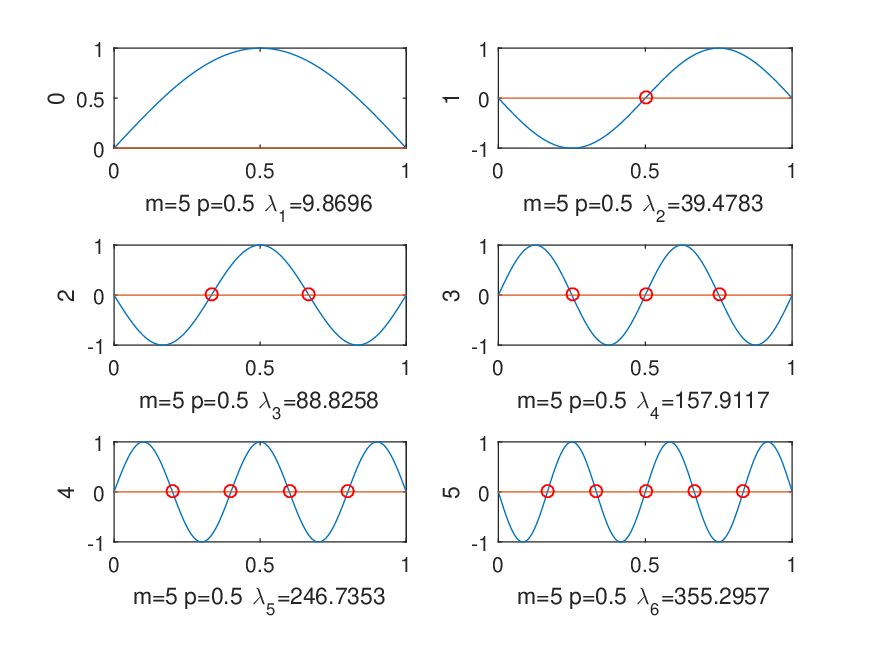}
\caption{First six eigenfunctions on the Interval on level $m=5$ for $p=0.5$ (standard case)}
\label{fig:intervaleigfunction2}
\end{figure}

\begin{figure}
\centering
\includegraphics[width=0.85\textwidth]{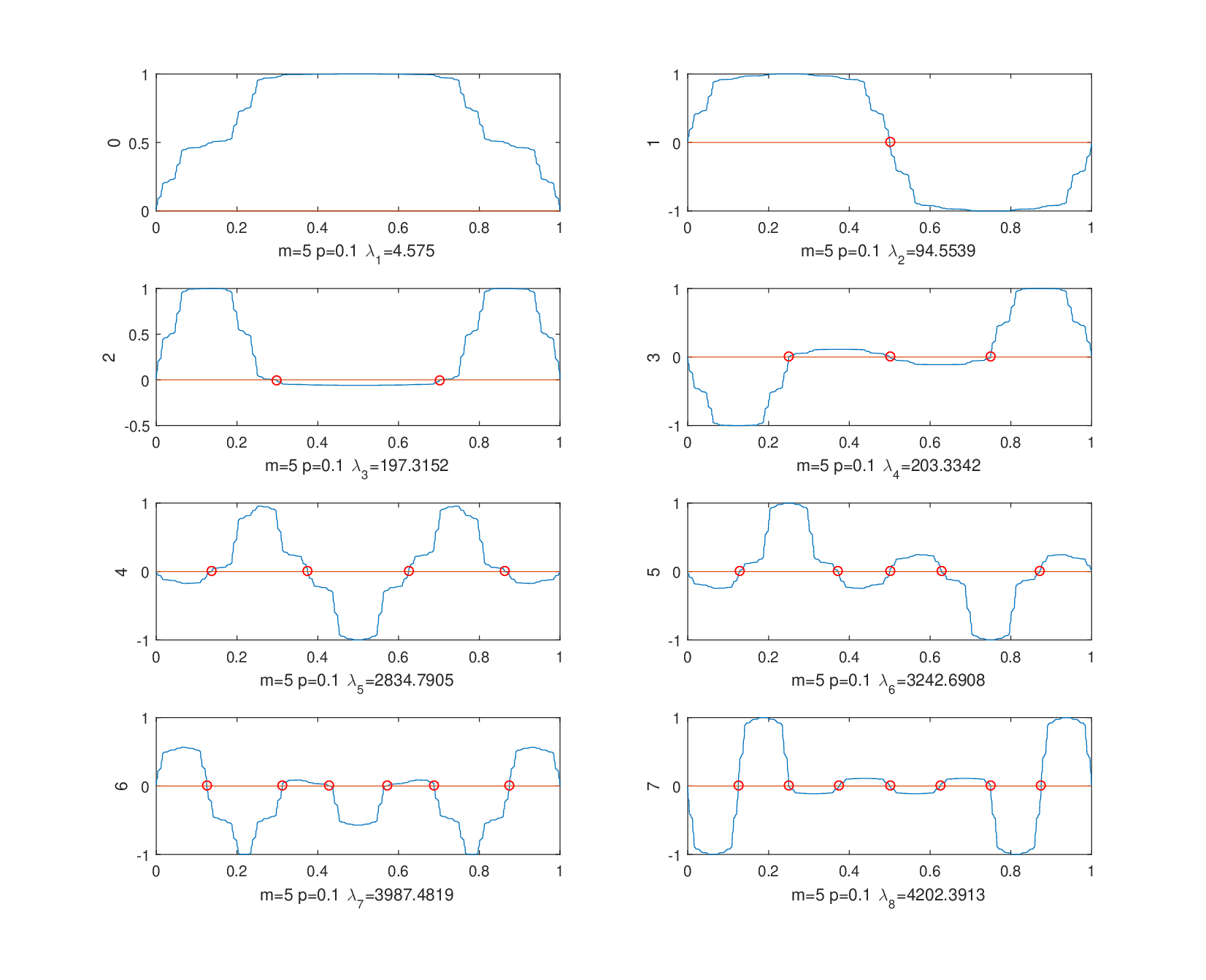}
\includegraphics[width=0.85\textwidth]{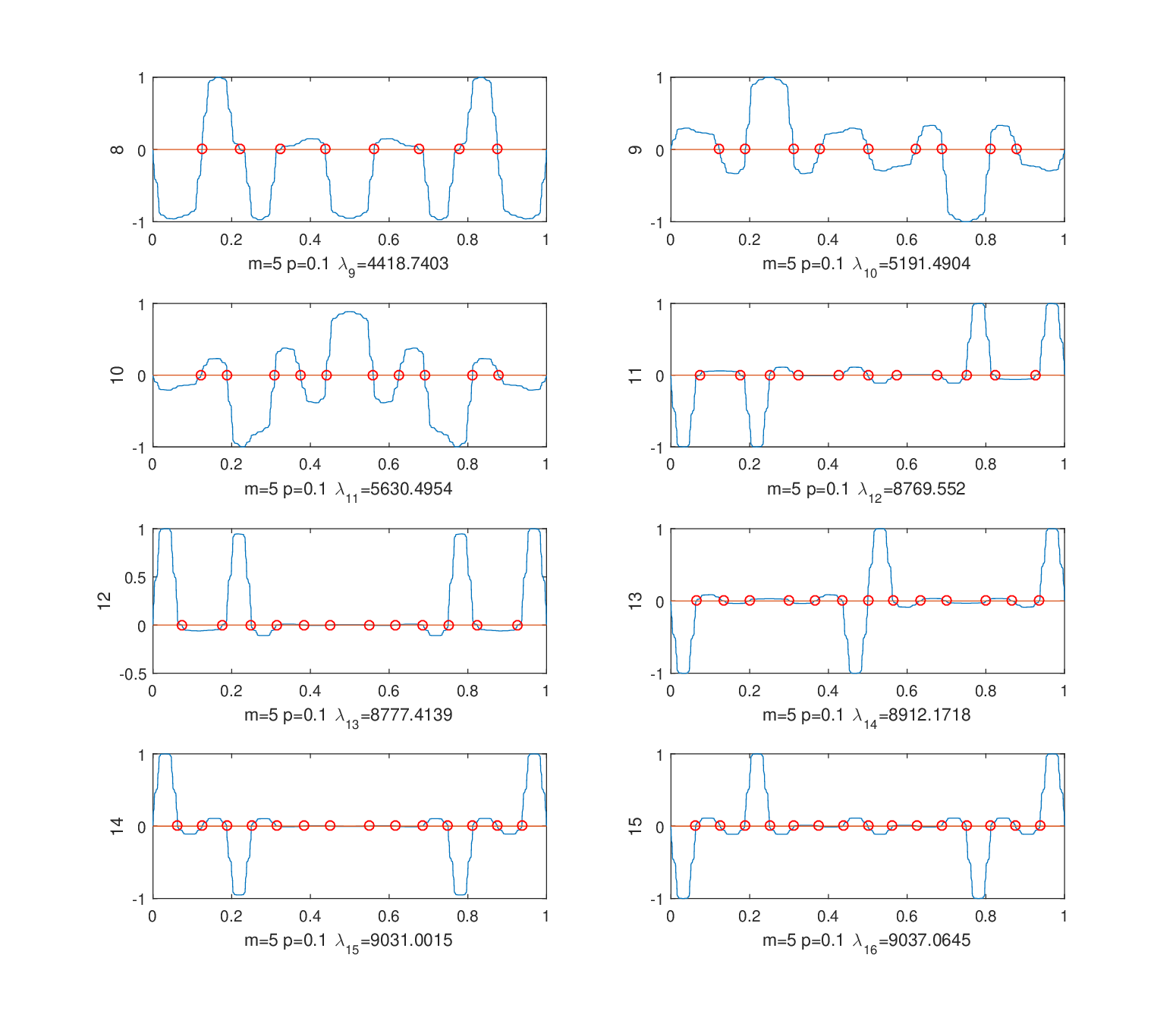}
\caption{First 16 eigenfunctions on the Interval on level $m=5$ for $p=0.1$}
\label{fig:intervaleigfunction4}
\end{figure}

\begin{figure}
\centering
\includegraphics[width=0.85\textwidth]{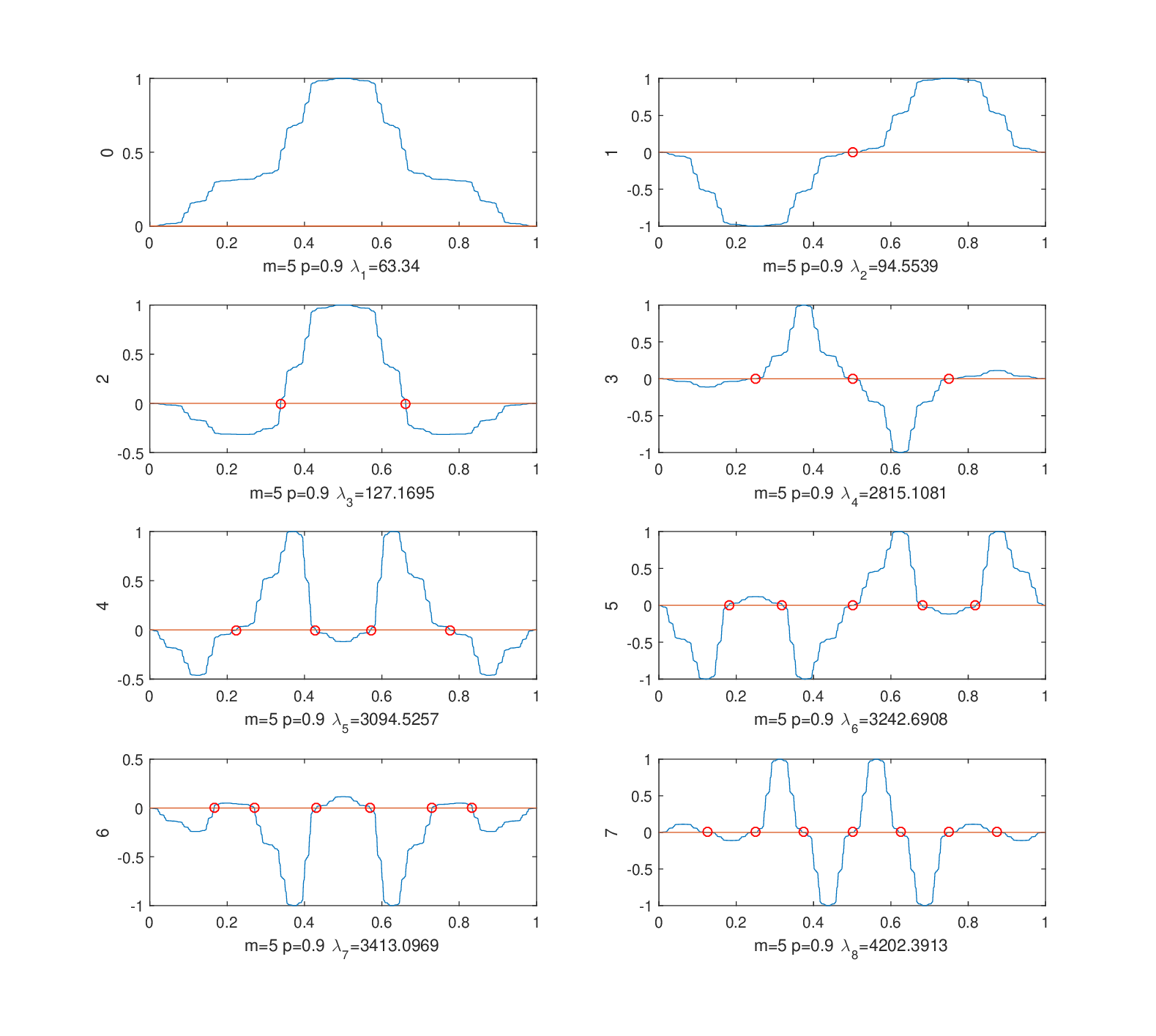}
\includegraphics[width=0.85\textwidth]{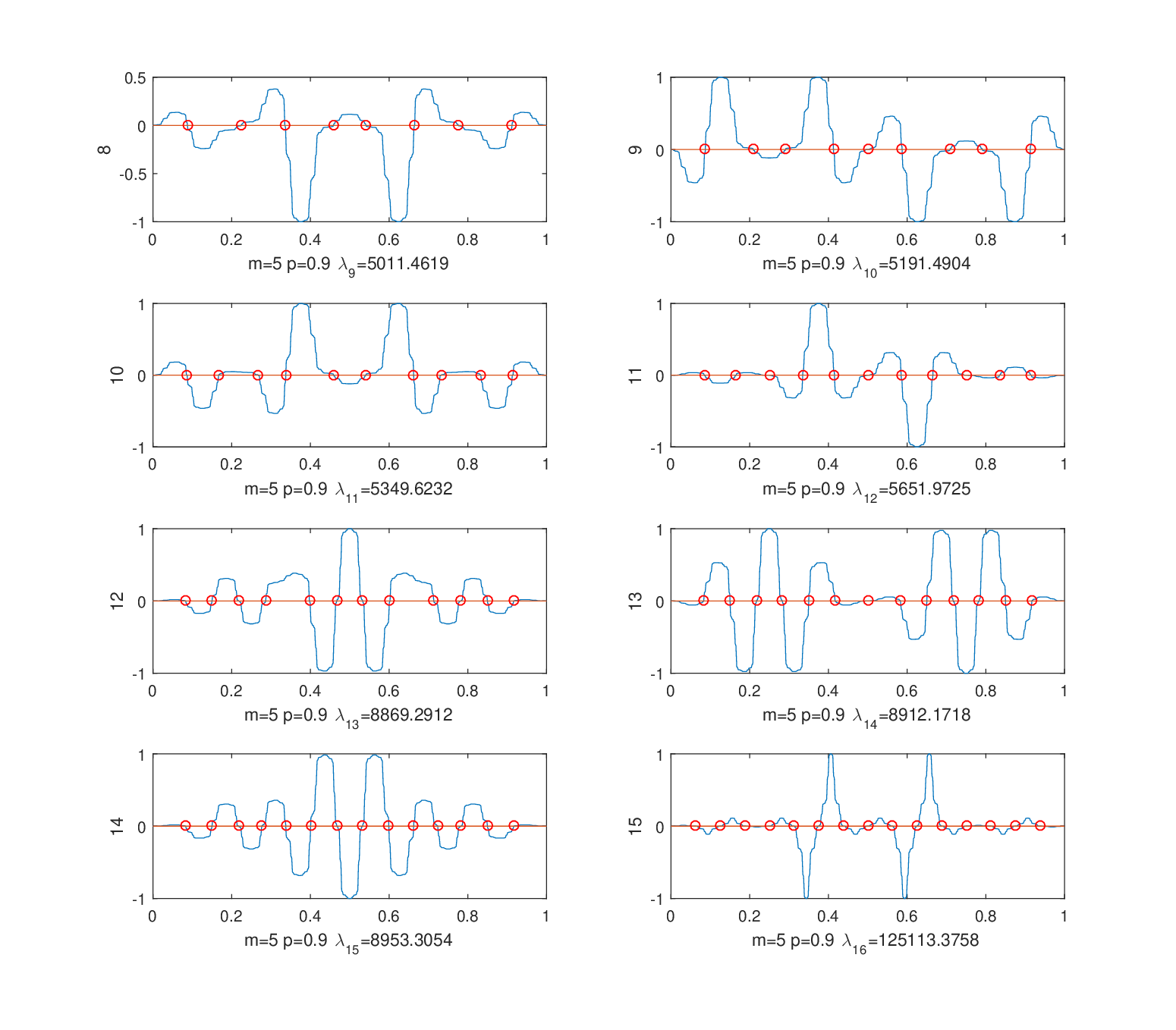}
\caption{First 16 eigenfunctions on the Interval on level $m=5$. $p=0.9$}
\label{fig:intervaleigfunction6}
\end{figure}

Figure \ref{fig:intervaleigfunction} shows all three eigenfunctions of the graph Laplacian at level 1 of the standard case when $p=0.5$, where $\lambda_{i}$ labeled below each graph are the eigenvalues associated with each eigenfunction. Figure \ref{fig:intervaleigfunction2} presents the first six eigenfunctions for the same value of $p$ but on a higher level, $m=5$. As we can see, the eigenfunctions on the Interval of the standard case are trigonometric functions. 

Figure \ref{fig:intervaleigfunction4} and \ref{fig:intervaleigfunction6} are the eigenfunction plots for the Interval with $m=5$ but different values of $p$. The value of $p$ for Figures \ref{fig:intervaleigfunction4} and \ref{fig:intervaleigfunction6} are $0.1$ and $0.9$, respectively. The eigenfunctions still resembles the trigonometric functions although not as much as the standard case.

\subsection{Eigenvalue Counting Function and Weyl Plot}

We define the eigenvalue counting function: 

\begin{equation}
N(x) = \#\{\lambda \ |  \ \lambda\leq x, \text{ for all eigenvalues }\lambda\}
\end{equation}

Figure \ref{fig:intervalcounting} shows eigenvalue counting functions of $p=0.1$ and $p=0.9$ at level 5. The eigenvalue counting function plots of every pair of interchanged $p$ and $q$ (for example, $p=0.1$ and $p=0.9$) look almost identical due to the set of matching eigenvalues. 

\begin{figure}
\centering
\includegraphics[width=0.49\textwidth]{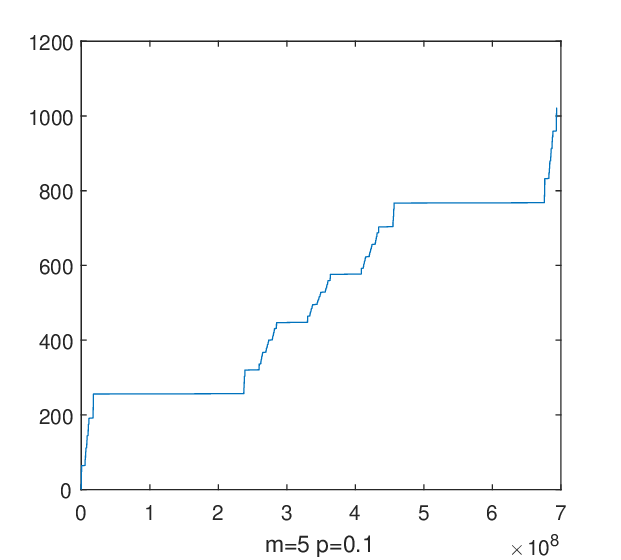}
\includegraphics[width=0.49\textwidth]{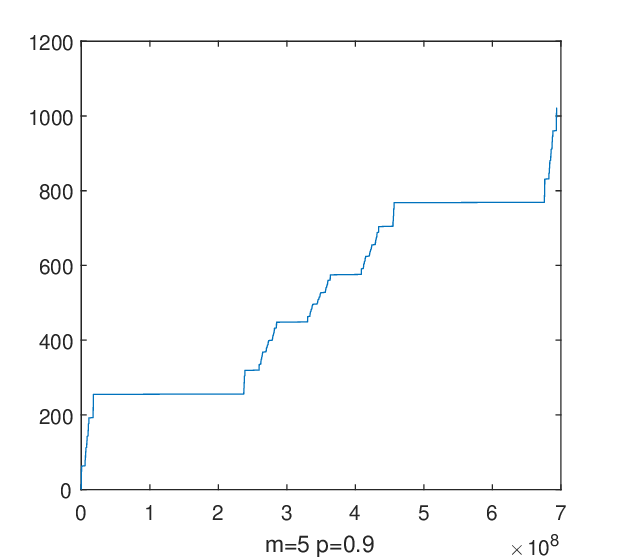}
\caption{Counting functions on the Interval on level $m=5$. Left: $p=0.1$, Right: $p=0.9$}
\label{fig:intervalcounting}
\end{figure}

By linear regression on a log-log plot, we can obtain numerical estimates for the classic associated power ratio $\alpha$, and then we can compute the Weyl ratio $W(x)=\frac{N(x)}{x^{\alpha}}$.

However, besides the numerical approach, we also have an algebraic expression for $\alpha$. At level $m$ the Laplacian renormalization factor is $(\frac{4}{pq})^{m}$, and thus $N((\frac{4}{pq})^{m}) \approx 4^{m}$. Therefore $(\frac{4}{pq})^{m\alpha} \approx 4^{m}$, and we get $\alpha = \frac{log4}{log(\frac{4}{pq})}$. Figure \ref{fig:intervalweyl} shows the Weyl plots of $p=0.1$ and $p=0.9$ at level 5. The power ratio $\alpha$ is labeled below the x-axis of each graph. Note that $p$ and $q$ also enter the expression for $\alpha$ symmetrically. 

\begin{figure}
\centering
  \includegraphics[width=0.49\textwidth]{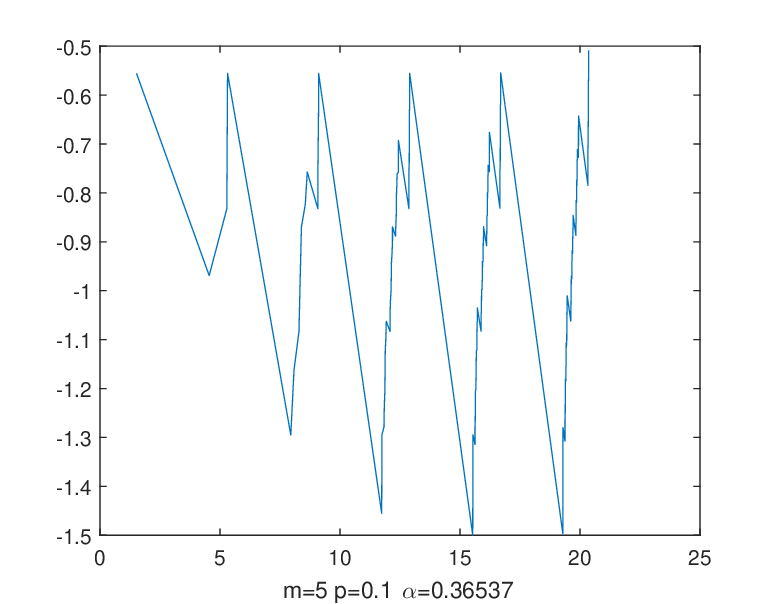}
  \includegraphics[width=0.49\textwidth]{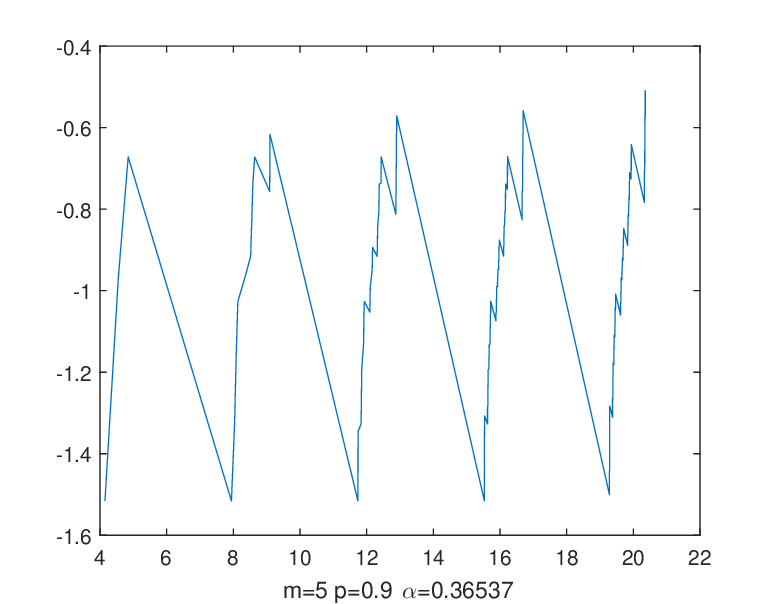}
  \caption{Weyl plots on the Interval at level $m=5$. Left: $p=0.1$, Right: $p=0.9$}
  \label{fig:intervalweyl}
\end{figure}

\subsection{Limiting Laplacians}

Having developed spectral decimation on the Interval, we would like to describe the behavior in the limiting cases, where $p \to 0$ or $p \to 1$. We begin with an analysis of the eigenvalue distribution.

For both limiting cases, the renormalization constant $\frac{4}{pq}$ is unbounded, and so any eigenvalues that are to remain bounded (with respect to $p$) must be very small. Considering the eigenvalues that are born on each level, $\big(\frac{4}{pq}\big)^{m}2$ and $\big(\frac{4}{pq}\big)^{m}2(1+\sqrt{q})$ are both unbounded for both limiting cases. In the case of the first eigenvalue,

\begin{equation}\lim_{p \to 0,m \to \infty} \left(\frac{4}{pq}\right)^{m}\Phi_{1}^{m}(2(1-\sqrt{q}))= 4
\end{equation}

\begin{equation}
\lim_{p \to 1,m \to \infty} \left(\frac{4}{pq}\right)^{m}\Phi_{1}^{m}(2(1-\sqrt{q}))= \infty
\end{equation}

the behavior is different for the two limiting cases. This is verified experimentally in Table \ref{tab:intlim}.

\begin{center}
\tiny
    \begin{tabular}{c|c|c|c}%
    $ $ & $p=10^{-2}$ & $p=10^{-4}$ & $p=10^{-5}$\\
    \hline
    \hline
    $\lambda_{1}$ & $4.0507$ & $4.0005$ & $4.0000$  \\
    \hline
    $\lambda_{2}$ & $813.1334$ & $8.0013*10^{4}$ & $8.0001*10^{5}$ \\
    \hline
    $\lambda_{3}$ & $1632.434$ & $1.6003*10^{5}$ & $1.6000*10^{6}$  \\
    \hline
    $\lambda_{4}$ & $1636.588$ & $1.60004*10^{5}$ & $1.6000*10^{6}$   \\
    \hline
    \end{tabular}
    \quad
    \begin{tabular}{c|c|c|c}%
    $ $ & $p=1-10^{-2}$ & $p=1-10^{-4}$ & $p=1-10^{-5}$\\
    \hline
    \hline
    $\lambda_{1}$ & $731.361$ & $7.9213*10^{4}$ & $7.9748*10^{5}$  \\
    \hline
    $\lambda_{2}$ & $813.1334$ & $8.0013*10^{4}$ & $8.0001*10^{5}$ \\
    \hline
    $\lambda_{3}$ & $895.009$ & $8.0813*10^{4}$ & $8.0254*10^{5}$  \\
    \hline
    $\lambda_{4}$ & $2.9385*10^{5}$ & $3.1686*10^{9}$ & $3.1899*10^{11}$   \\
    \hline
    \end{tabular}
    \captionof{table}{Eigenvalues of limiting Laplacians on the Interval}
    \label{tab:intlim}
\end{center}

Patterns also develop in the relationships of the unbounded eigenvalues -- these will be addressed in the upcoming section. The behavior of the eigenfunctions associated with these eigenvalues is complex and is connected to the sequence of $\Phi$ maps used to reach each individual eigenvalue. However, here we provide a clear analysis of a single example -- the eigenfunction associated with $\lambda_{1} = 4$ as $p \to 0$. We can explicitly take the limits of the four eigenfunction extension formulas as $p \to 0$. Writing the extensions to $y_{1},z,$ and $y_{2}$ as functions of values $x_{1},x_{2},p,\lambda$,

\begin{align}
    \lim_{p\to 0} y_{1}(x_{1},x_{2},p,\Phi_{1}(2(1-\sqrt{1-p}))) &= \frac{x_{1}+x_{2}}{2} \nonumber\\
    \lim_{p\to 0} z(x_{1},x_{2},p,\Phi_{1}(2(1-\sqrt{1-p}))) &= \frac{x_{1}+x_{2}}{2} \nonumber\\
    \lim_{p\to 0} y_{2}(x_{1},x_{2},p,\Phi_{1}(2(1-\sqrt{1-p}))) &= \frac{x_{1}+x_{2}}{2}
\end{align}

Here we use $\Phi_{1}(2(1-\sqrt{1-p})$ as the eigenvalue because this will lead us to $\lambda = 4$. Since the Dirichlet eigenfunction associated with this eigenvalue on level 1 is uniformly 1 on $\frac{1}{4},\frac{1}{2},\frac{3}{4}$, this extension algorithm produces two Cantor functions, joined by an interval of uniformly 1. These theoretical results are supported numerically in Figure \ref{fig:cantor}, where the Devil's Staircase is clearly visible on the Interval $[0,\frac{1}{4}]$ (and again in reverse on $[\frac{3}{4},1]$).

Another observation is that the miniaturization algorithm on eigenfunctions that unites 4 copies of an eigenfunction has interesting limiting behavior. The ratio $\frac{p}{q}$ used in the piecewise definition will go to $0$ with $p \to 0$ and go to $\infty$ with $p \to 1$. After normalization, this causes eigenfunctions produced in this manner to have support limited to the inner half of the Interval as $p \to 1$ and support limited to the outer quarters of the Interval as $p \to 0$.

\begin{figure}
\centering
  \includegraphics[width=0.5\textwidth]{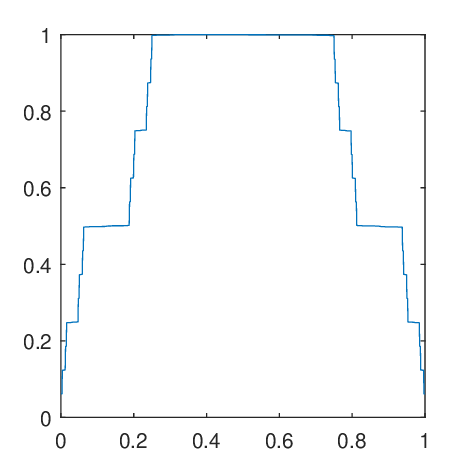}
  \caption{Ground state eigenfunction for $p=0.001$, $m=4$}
  \label{fig:cantor}
\end{figure}

\subsection{Ratios of Eigenvalues}

We may tackle the behavior of unbounded eigenvalues by instead examining their behavior using the ratios between eigenvalues -- this effectively removes the renormalization constant by division. In the standard case, taking ratios of squares of integers forms a dense set. The eigenfunctions that are born on level $m$, $2(1-\sqrt{q}),2,2(1+\sqrt{q})$ will converge to $0,2,4$ as $p \to 0$ and $2,2,2$ as $p \to 1$. The eigenvalues produced through decimation are given by $\Phi_{1},\Phi_{2},\Phi_{3},\Phi_{4}$ applied to eigenvalues, but explicitly taking limits of these maps gives us values independent of the eigenvalue. Specifically

\begin{align}
    \lim_{p\to 0} \{\Phi_{1},\Phi_{2},\Phi_{3},\Phi_{4}\} &= \{0,2,2,4\} \nonumber\\
    \lim_{p\to 1} \{\Phi_{1},\Phi_{2},\Phi_{3},\Phi_{4}\} &= \{0,2,2,4\}
\end{align}

Then, excepting larger ratios derived from comparing 2 or 4 to values near 0, we expect the set of ratios to approach the set of fractions obtained by choosing from $\{2,4\}$, namely $\{\frac{1}{2},1,2\}$. Indeed, we find numerical evidence supporting this claim, visible in Figure \ref{fig:intratio}.

\begin{figure}
\centering
    \includegraphics[width=0.3\textwidth]{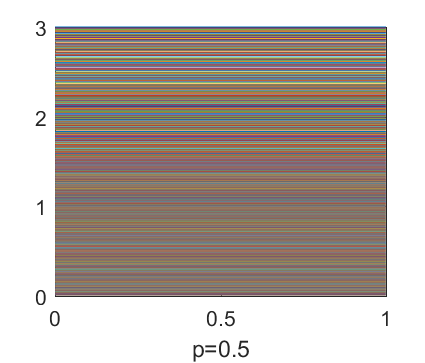}
    \includegraphics[width=0.3\textwidth]{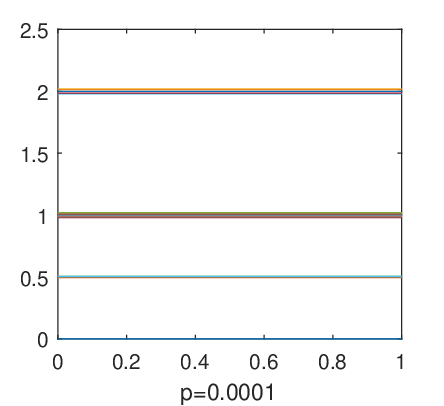}
    \includegraphics[width=0.3\textwidth]{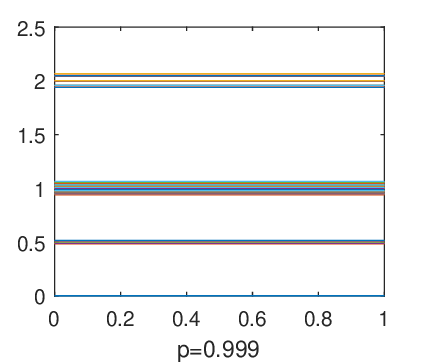}
  \caption{Eigenvalue ratios for $p=\frac{1}{2}$, $p=0.0001$, $p=0.999$}
  \label{fig:intratio}
\end{figure}

\section{Zeros and Extrema of Eigenfunctions on the Interval}
The eigenfunctions of the standard Laplacian on the Interval are well studied as the eigenfunctions of the second derivative. They are a special case of Sturm-Liouville equations, which are second-order homogeneous linear differential equations of form
\begin{equation}
\frac{d}{dx}\left[p(x) \frac{du}{dx}\right] + [\lambda\rho(x) -q(x)]u=0
\end{equation}
which yield the Laplacian equation if $p(x) = 1$, $\rho(x) = 1$, and $q(x) = 0$. Studying eigenfunctions, i.e. sine curves, in this way allows us to characterize them by their zeros and extrema. Classic Sturm-Liouville theory proves the following theorem to do this:
\begin{thm}[Sturm Comparison Theorem]\label{cmpthm}
Let $P(x) \geq P_1(x) > 0$ and $Q_1(x) \geq Q(x)$ in the differential equations
\begin{align}
\frac{d}{dx}\left(P(x)\frac{du}{dx}\right) + Q(x)u &= 0 \\
\frac{d}{dx}\left(P_1(x)\frac{du_1}{dx}\right) + Q_1(x)u_1 &= 0
\end{align}
Then, between any two zeros of a nontrivial solution $u(x)$ of the first differential equation, there lies at least one zero of every solution of the second differential equation, except when $u(x) \equiv cu_1(x)$. This implies $P\equiv P_1$ and $Q\equiv Q_1$, except possibly in intervals where $Q \equiv Q_1 \equiv 0$.
\end{thm}

It follows from the Sturm Comparison theorem that if we have two eigenfunctions $u_1, u_2$ of the standard Laplacian on the Interval such that
\begin{align}
-\Delta u_1 &= \lambda_1 u_1\\
-\Delta u_2 &= \lambda_2 u_2
\end{align}
where $\lambda_1, \lambda_2$ are constants, if $\lambda_2 > \lambda_1 > 0$, between every pair of zeros of $u_2$, $u_1$ will also have at least one zero. This is a special case of the theorem. It is easy to verify this result since we know that the eigenfunctions for the standard Laplacian are of form $f(x) = \sin(k\pi x)$.

We establish an analogous result to the special case of the Sturm Comparison Theorem for all $\Delta^{(p)}$ in the theorem below. The proof of Theorem \ref{cmpthm} involves linear operators that, when twice iterated, equal the Laplacian (i.e. the second derivative) and classic trignometric functions that allow one to exploit useful facts about their zeros. We were not able to employ such strategies, as we lacked analogous notions of derivative and trignometric functions that would help us. Therefore, the proof below uses different techniques than those used in classic Sturm-Liouville theory. We note that similar results for related Laplacians have been obtained in \cite{Bird}.

\begin{thm}\label{slanalog}
Let $\lambda_i$ be the $i$th eigenvalue and $f_i$ be the eigenfunction for $\lambda_i$.
\begin{itemize}
\item[(a)] For any eigenfunction $f$ of the Interval, there is exactly one local extremum between two consecutive zeros.
\item[(b)] $f_i$ has $i-1$ zeros.
\item[(c)] If $\lambda_i < \lambda_j$ and $x_k, x_{k+1}$ are consecutive zeros of $f_i$, then $f_j$ has at least one zero in $[x_k, x_{k+1}]$.
\end{itemize}
\end{thm}
\begin{proof}
(a) If $f(x)$ is a local maximum, then $-\Delta f(x) > 0$. Since $\lambda > 0$ and $-\Delta f(x) = \lambda f(x)$, $f(x) > 0$. Similarly, if $f(x)$ is a local minimum, $f(x) < 0$. Since if $z, w$ are consecutive zeros, $f(a) > 0$ for all $z < a < w$ or $f(a) < 0$ for all $z < a < w$, there can be either only maxima or only minima between two zeros, meaning there can only be one extrema. 

(b) The result is true for $p=0.5$ by Theorem \ref{cmpthm}. We claim that as $p$ varies continuously, $\lambda_{p,i}$ and $f_{p, i}(x)$ for all points $x$, the $i$th eigenvalue and the value at the $i$th eigenfunction of $\Delta^{(p)}$, also vary continuously. This means that if there exists $p$ such that $f_{p, i}$ does not have $i-1$ zeros, $f_{p, i}$ must have morphed continuously from $f_{0.5, i}$ to do so. We know that in order for $\Delta^{(p)} f_{p,i}(x) = 0$, on some neighborhood $A$ of $x$, for all $x_0, x_1\in A$ where $x_0<x$ and $x_1>x$, $f(x_0) < 0 < f(x_1)$ or $f(x_0) > 0 > f(x_1)$. However, in order for the number of zeros of $f_{p,i}$ to change from that of $f_{0.5, i}$, there has to exist $p'$ such that $f_{p', i}$ has a zero $x$ where on some neighborhood $A$ of $x$, $f_{p',i}(y) \geq 0$ for all $y\in A$ or $f_{p',i}(y) \leq 0$ for all $y\in A$. This is a contradiction. Therefore, the number of zeros for the $i$th eigenfunction stays constant as $p$ varies. 

(c) Proof by contradiction. Consider $\lambda_i < \lambda_j$, with eigenfunctions $f_i, f_j$ respectively, of $\Delta_m^{(p)}$ given $m$ and $p$. If the statement were false, then there would exist consecutive zeros of $f_i$, $x_k, x_{k+1}$, such that $f_j$ does not have a zero in $A = [x_k, x_{k+1}]$. Since $f_i, f_j$ are eigenfunctions,
\begin{align}
-\Delta f_i &= \lambda_i f_i \\
-\Delta f_j &= \lambda_j f_j \\
\implies (\lambda_j -\lambda_i)f_if_j &= f_i(-\Delta f_j)-f_j(-\Delta f_i) &&\text{by algebra}
\end{align}
We can assume without loss of generality that $f_i$ and $f_j$ are both positive on $A$. By the Gauss-Green formula,
\begin{align}
\int_A f_i(-\Delta f_j)d\mu-\int_A f_j(-\Delta f_i) d\mu &= \sum_{x \in \partial A} \big((f_j \partial_n f_i)(x)-(f_i \partial_n f_j)(x)\big)
\end{align}

So, 
\begin{align}
\int_A (\lambda_j - \lambda_i) f_i f_j d\mu &= \int_A f_i(-\Delta f_j)d\mu-\int_A f_j(-\Delta f_i) d\mu \\
&= \sum_{x \in \partial A} \big((f_j \partial_n f_i)(x)-(f_i \partial_n f_j)(x)\big) \\
&= \sum_{x \in \partial A} \big((f_j \partial_n f_i)(x)-(0)(\partial_n f_j)(x)\big) &&\text{because } f_i(\partial A) = 0\\
&= (f_{j}\partial_n f_i)(x_k) + (f_j \partial_n f_i)(x_{k+1})
\end{align}

Since $f_i$ is positive on $A$, the normal derivative must be negative at the boundaries of A, i.e. $\partial_n f_i(x_k), \partial_n f_i(x_{k+1}) < 0$. We assumed that $f_j(x_k), f_j(x_{k+1}) > 0$, so the RHS is negative. But the LHS is positive, since the integrand is positive. This is a contradiction. Therefore, we have proved that $f_{j}$ has at least one zero in $[x_k, x_{k+1}]$. 
\end{proof}

The following corollaries follow directly from Theorem \ref{slanalog}:
\begin{cor}
If $\lambda_i$ is the $i$th eigenvalue and $f_i$ is its eigenfunction, $f_i$ has exactly $i$ local extrema.
\end{cor}
\begin{cor}
If $\lambda_i, \lambda_{i+1}$ are consecutive eigenvalues with eigenfunctions $f_i, f_{i+1}$ respectively, then for each pair of consecutive zeros of $f_i$, ocurring at $x_k$ and $x_{k+1}$, $f_{i+1}$ has exactly one zero in $[x_k, x_{k+1}]$.
\end{cor}

In addition to this theorem, we can apply Proposition \ref{squishprop} to observe a pattern in the zeros and the values of the local extrema in certain eigenfunctions: 

\begin{cor}
If $g_i$ is the eigenfunction associated with the $i$th eigenvalue $\lambda_i$ on level $m-1$ and $x$ is one of its zeros, then for all $0\leq k\leq 3$, $F_k(x)$ is a zero of $f_{4i}$, the eigenfunction associated with the $(4i)$th eigenvalue $\lambda_{4i}$. The values of the local extrema of $f_{4i}$ are of form $\pm\left(\frac{p}{q}\right)^n \max(g_i(x))$.
\end{cor}

This gives us a nice description of the eigenfunctions, especially those of the form $\lambda_{4^n}$ and $2(4^n)$th eigenvalues $\lambda_{2(4^n)}$. The zeros of the $1$st and $2$nd eigenfunctions occur at $\{0, 1\}$ and $\{0, \frac{1}{2}, 1\}$ respectively, for all $p$, and are spaced evenly. This means that all eigenfunctions $f_{4^n}$ and $f_{2(4^n)}$ will have evenly spaced zeros for all $p$. See Figures \ref{fig:intervaleigfunction4} and \ref{fig:intervaleigfunction6} for the first 16 eigenfunctions for $p=0.1$ and $p=0.9$ with the zeros identified.

\section{Laplacians on the Sierpinski Gasket} \label{sglap}
In the standard theory, the Sierpinski Gasket is defined by an IFS consisting of three contractive mappings, with uniform measure throughout the cells and uniform resistance throughout the edges. As in the Interval case before, a larger set of symmetric, self-similar Laplacians can be generated through a modified IFS, defined
\begin{equation}
\{F_{jk} \mid F_{jk} = F_j\circ F_k\}
\end{equation}
where $F_i$ is a standard contractive mapping for the Sierpinski Gasket. Then, the Sierpinski Gasket can be equivalently defined
\begin{equation}
SG = \bigcup_{0\leq j, k\leq 2}F_{jk}(SG)
\end{equation}

This new IFS allows us to define a non-uniform, self-similar, symmetric measure for SG. Note that $F_{jk}(SG)$ gives the outer cells if $j=k$, the inner cells if $j \neq k$. In order to maintain symmetry, we must define the measure so that $\mu(F_{ii}(SG)) = \mu(F_{jj}(SG))$ for all $i, j$, and $\mu(F_{jk}) = \mu(F_{ih})$ for all $j\neq k$, $i\neq h$. Without loss of generality, we will set $\mu(SG) = 1$, meaning that if $\mu_0$ denotes the measure of an outer cell and $\mu_1$ denotes the measure of an inner cell,
\begin{equation}
3\mu_0 + 6\mu_1 = 1
\end{equation}
leaving us only one free measure parameter to vary. 

In order to compute the measure of $A = F_{j_1k_1}\circ F_{j_2k_2}\circ ... \circ F_{j_mk_m}(SG)$, an $m$-cell of SG, define $i(A)$ to be the number of $j_a$ such that $j_a=k_a$, i.e. the number of ``outer" mappings needed to obtain $A$. The number of ``inner" mappings needed is $m-i(A)$. Then,
\begin{equation}
\mu(A) = \mu_0^{i(A)}\mu_1^{m-i(A)} = \mu_0^{i(A)}\left(\frac{1-3\mu_0}{6}\right)^{m-i(A)}
\end{equation}

In addition to a non-uniform measure, we can also define a non-uniform resistance. In order to maintain symmetry, we must define the resistance of the edges of the outer cells to be equal and the same for the inner cells. Call the resistance of the outer cells $r_0$ and the inner cells $r_1$. We want to compute the resistances of the edges so that the resulting effective resistances between points in $V_0$ are equal; call this effective resistance $\rho$. Let $r_1=1$, $r=r_0$, and we will multiply by a constant at the end. The $\Delta-Y$ transforms shown in Figures \ref{fig:delta1} and \ref{fig:delta2} show that
\begin{align}
\rho &= \frac{9r^2+26r+15}{6(r+2)} \\
r_0 &= \frac{6r(r+2)}{9r^2+26r+15} \\
r_1 &= \frac{6(r+2)}{9r^2+26r+15}
\end{align}
and so there is only one free resistance parameter, $r = \frac{r_0}{r_1}$, to vary.

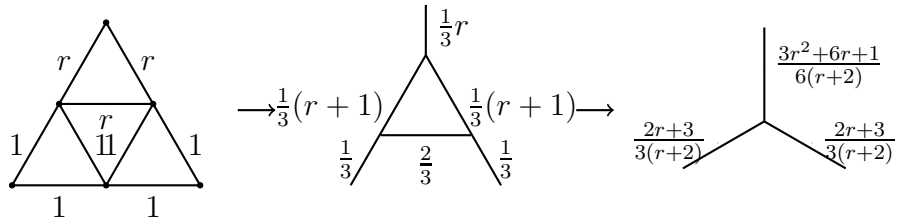
\begin{figure}
\centering
\begin{tikzpicture}

\begin{scope}[scale=0.5]
\draw[black,thick] (0,0) -- (5,0);
\draw[black,thick] (0,0) -- (2.5,4.3301);
\draw[black,thick] (5,0) -- (2.5,4.3301);
\draw[black,thick] (1.25,2.1651) -- (3.75,2.1651);
\draw[black,thick] (1.25,2.1651) -- (2.5,0);
\draw[black,thick] (3.75,2.1651) -- (2.5,0);

\filldraw[black] (1.25,0) circle (0pt) node[anchor=north] {$1$};
\filldraw[black] (3.75,0) circle (0pt) node[anchor=north] {$1$};
\filldraw[black] (2.5,2.15) circle (0pt) node[anchor=north] {$r$};
\filldraw[black] (0.625,1.0825) circle (0pt) node[anchor=east] {$1$};
\filldraw[black] (1.875,3.2475) circle (0pt) node[anchor=east] {$r$};
\filldraw[black] (3.125,1.0825) circle (0pt) node[anchor=east] {$1$};
\filldraw[black] (1.875,1.0825) circle (0pt) node[anchor=west] {$1$};
\filldraw[black] (4.375,1.0825) circle (0pt) node[anchor=west] {$1$};
\filldraw[black] (3.125,3.2475) circle (0pt) node[anchor=west] {$r$};

\filldraw[black] (0,0) circle (2pt) node[anchor=east] {};
\filldraw[black] (2.5,4.3301) circle (2pt) node[anchor=east] {};
\filldraw[black] (5,0) circle (2pt) node[anchor=north] {};
\filldraw[black] (3.75,2.1651) circle (2pt) node[anchor=west] {};
\filldraw[black] (1.25,2.1651) circle (2pt) node[anchor=east] {};
\filldraw[black] (2.5,0) circle (2pt) node[anchor=north] {};
\end{scope}

\draw[->,black,thick] (3,1) -- (3.5,1);

\begin{scope}[xshift=4.5cm,scale=0.2]

\draw[black,thick] (2,3.333) -- (8,3.333);
\draw[black,thick] (5,8.6603) -- (0,0);
\draw[black,thick] (10,0) -- (5,8.6603);
\draw[black,thick] (5,12) -- (5,8.6603);

\filldraw[black] (1,1.7321) circle (0pt) node[anchor=east]{$\frac{1}{3}$};
\filldraw[black] (9,1.7321) circle (0pt) node[anchor=west]{$\frac{1}{3}$};
\filldraw[black] (5,10.5) circle (0pt) node[anchor=west]{$\frac{1}{3}r$};
\filldraw[black] (3,5.1962) circle (0pt) node[anchor=east]{$\frac{1}{3}(r+1)$};
\filldraw[black] (7,5.1962) circle (0pt) node[anchor=west]{$\frac{1}{3}(r+1)$};
\filldraw[black] (5,3.333) circle (0pt) node[anchor=north]{$\frac{2}{3}$};

\end{scope}

\draw[->,black,thick] (7.5,1) -- (8,1);

\begin{scope}[xshift = 10cm,yshift =0.85cm,scale=0.25]

\draw[black,thick] (0,0) -- (0,5);
\draw[black,thick] (0,0) -- (-4.3301,-2.5);
\draw[black,thick] (0,0) -- (4.3301,-2.5);

\filldraw[black] (0,3) circle (0pt) node[anchor=west]{$\frac{3r^{2}+6r+1}{6(r+2)}$};
\filldraw[black] (-2.5,-1.15) circle (0pt) node[anchor=east]{$\frac{2r+3}{3(r+2)}$};
\filldraw[black] (2.5,-1.15) circle (0pt) node[anchor=west]{$\frac{2r+3}{3(r+2)}$};

\end{scope}

\end{tikzpicture}

\caption{The transformation on $\frac{1}{3}$ of the total gasket. Each arrow denotes a $\Delta - Y$ transform.} \label{fig:delta1}
\end{figure}

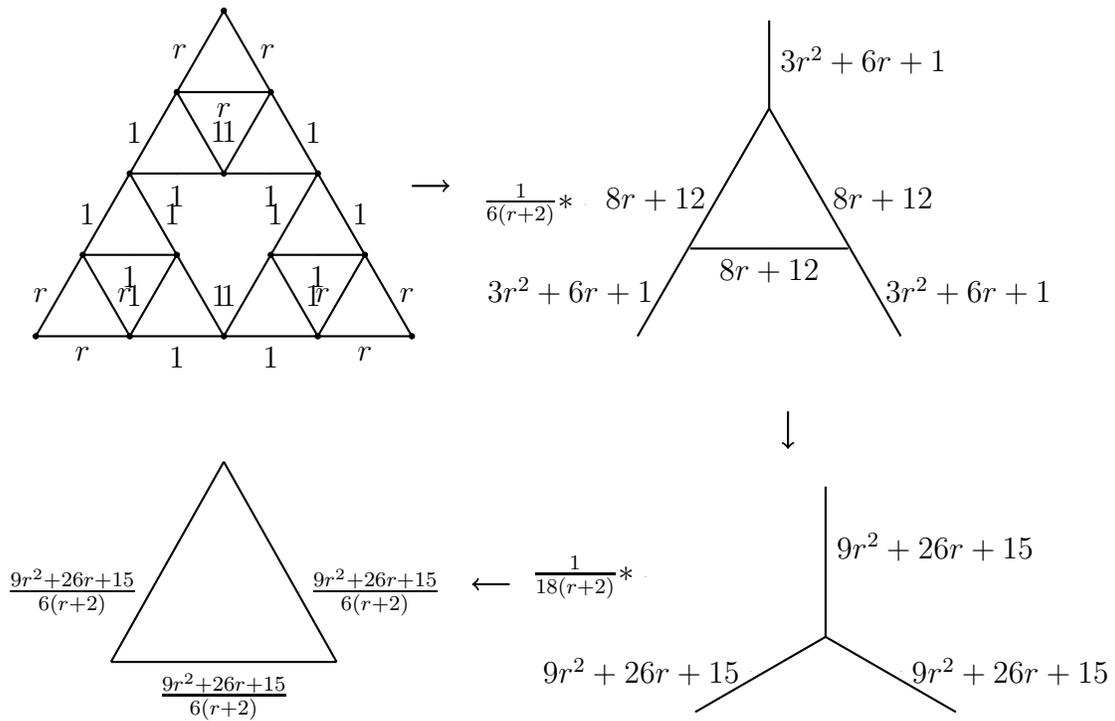
\begin{figure}
\centering
\begin{tikzpicture}

\begin{scope}[scale=0.5]
\draw[black,thick] (0,0) -- (10,0);
\draw[black,thick] (0,0) -- (5,8.6603);
\draw[black,thick] (5,8.6603) -- (10,0);
\draw[black,thick] (1.25,2.1651) -- (3.75,2.1651);
\draw[black,thick] (2.5,4.3301) -- (7.5,4.3301);
\draw[black,thick] (3.75,6.4952) -- (6.25,6.4952);
\draw[black,thick] (3.75,2.1651) -- (2.5,0);
\draw[black,thick] (7.5,4.3301) -- (5,0);
\draw[black,thick] (8.75,2.1651) -- (7.5,0);
\draw[black,thick] (1.25,2.1651) -- (2.5,0);
\draw[black,thick] (2.5,4.3301) -- (5,0);
\draw[black,thick] (3.75,6.4952) -- (5,4.3301);
\draw[black,thick] (6.25,6.4952) -- (5,4.3301);
\draw[black,thick] (6.25,2.1651) -- (8.75,2.1651);
\draw[black,thick] (6.25,2.1651) -- (7.5,0);

\filldraw[black] (1.25,0) circle (0pt) node[anchor=north] {$r$};
\filldraw[black] (3.75,0) circle (0pt) node[anchor=north] {$1$};
\filldraw[black] (6.25,0) circle (0pt) node[anchor=north] {$1$};
\filldraw[black] (8.75,0) circle (0pt) node[anchor=north] {$r$};
\filldraw[black] (2.5,2.15) circle (0pt) node[anchor=north] {$1$};
\filldraw[black] (7.5,2.15) circle (0pt) node[anchor=north] {$1$};
\filldraw[black] (3.75,4.3001) circle (0pt) node[anchor=north] {$1$};
\filldraw[black] (6.25,4.3001) circle (0pt) node[anchor=north] {$1$};
\filldraw[black] (5,6.4952) circle (0pt) node[anchor=north] {$r$};
\filldraw[black] (0.625,1.0825) circle (0pt) node[anchor=east] {$r$};
\filldraw[black] (1.875,3.2475) circle (0pt) node[anchor=east] {$1$};
\filldraw[black] (3.125,5.41) circle (0pt) node[anchor=east] {$1$};
\filldraw[black] (4.325,7.58) circle (0pt) node[anchor=east] {$r$};
\filldraw[black] (9.375,1.0825) circle (0pt) node[anchor=west] {$r$};
\filldraw[black] (8.125,3.2475) circle (0pt) node[anchor=west] {$1$};
\filldraw[black] (6.875,5.41) circle (0pt) node[anchor=west] {$1$};
\filldraw[black] (5.675,7.58) circle (0pt) node[anchor=west] {$r$};
\filldraw[black] (4.35,5.45) circle (0pt) node[anchor=west] {$1$};
\filldraw[black] (5.65,5.45) circle (0pt) node[anchor=east] {$1$};
\filldraw[black] (3.125,1.0825) circle (0pt) node[anchor=east] {$1$};
\filldraw[black] (1.875,1.0825) circle (0pt) node[anchor=west] {$r$};
\filldraw[black] (4.375,1.0825) circle (0pt) node[anchor=west] {$1$};
\filldraw[black] (3.125,3.2475) circle (0pt) node[anchor=west] {$1$};
\filldraw[black] (5.625,1.0825) circle (0pt) node[anchor=east] {$1$};
\filldraw[black] (6.8750,1.0825) circle (0pt) node[anchor=west] {$1$};
\filldraw[black] (8.125,1.0825) circle (0pt) node[anchor=east] {$r$};
\filldraw[black] (6.875,3.2475) circle (0pt) node[anchor=east] {$1$};

\filldraw[black] (0,0) circle (2pt) node[anchor=east] {};
\filldraw[black] (10,0) circle (2pt) node[anchor=west] {};
\filldraw[black] (5,8.6603) circle (2pt) node[anchor=south] {};
\filldraw[black] (2.5,4.3301) circle (2pt) node[anchor=east] {};
\filldraw[black] (7.5,4.3301) circle (2pt) node[anchor=west] {};
\filldraw[black] (5,0) circle (2pt) node[anchor=north] {};
\filldraw[black] (6.25,6.4952) circle (2pt) node[anchor=west] {};
\filldraw[black] (8.75,2.1651) circle (2pt) node[anchor=west] {};
\filldraw[black] (3.75,2.1651) circle (2pt) node[anchor=west] {};
\filldraw[black] (1.25,2.1651) circle (2pt) node[anchor=east] {};
\filldraw[black] (6.25,2.1651) circle (2pt) node[anchor=east] {};
\filldraw[black] (3.75,6.4952) circle (2pt) node[anchor=east] {};
\filldraw[black] (5,4.3301) circle (2pt) node[anchor=north] {};
\filldraw[black] (2.5,0) circle (2pt) node[anchor=north] {};
\filldraw[black] (7.5,0) circle (2pt) node[anchor=north] {};
\end{scope}

\draw[->,black,thick] (5,2) -- (5.5,2);

\begin{scope}[xshift = 8 cm,scale=0.35]

\draw[black,thick] (2,3.333) -- (8,3.333);
\draw[black,thick] (5,8.6603) -- (0,0);
\draw[black,thick] (10,0) -- (5,8.6603);
\draw[black,thick] (5,12) -- (5,8.6603);

\filldraw[black] (-2,5) circle (0pt) node[anchor=east]{$\frac{1}{6(r+2)}*$};
\filldraw[black] (1,1.7321) circle (0pt) node[anchor=east]{$3r^{2}+6r+1$};
\filldraw[black] (9,1.7321) circle (0pt) node[anchor=west]{$3r^{2}+6r+1$};
\filldraw[black] (5,10.5) circle (0pt) node[anchor=west]{$3r^{2}+6r+1$};
\filldraw[black] (3,5.1962) circle (0pt) node[anchor=east]{$8r+12$};
\filldraw[black] (7,5.1962) circle (0pt) node[anchor=west]{$8r+12$};
\filldraw[black] (5,3.333) circle (0pt) node[anchor=north]{$8r+12$};

\end{scope}

\begin{scope}[xshift = 10.5 cm, yshift = -4 cm,scale=0.4]

\draw[black,thick] (0,0) -- (0,5);
\draw[black,thick] (0,0) -- (-4.3301,-2.5);
\draw[black,thick] (0,0) -- (4.3301,-2.5);

\filldraw[black] (-6,2) circle (0pt) node[anchor=east]{$\frac{1}{18(r+2)}*$};
\filldraw[black] (0,3) circle (0pt) node[anchor=west]{$9r^{2}+26r+15$};
\filldraw[black] (-2.5,-1.15) circle (0pt) node[anchor=east]{$9r^{2}+26r+15$};
\filldraw[black] (2.5,-1.15) circle (0pt) node[anchor=west]{$9r^{2}+26r+15$};

\end{scope}

\draw[->,black,thick] (10,-1) -- (10,-1.5);
\draw[->,black,thick] (6.3,-3.3) -- (5.8,-3.3);

\begin{scope}[yshift = -6 cm,scale=0.5]

\draw[black,thick] (2,3.333) -- (8,3.333);
\draw[black,thick] (5,8.6603) -- (2,3.333);
\draw[black,thick] (8,3.333) -- (5,8.6603);

\filldraw[black] (3,5.1962) circle (0pt) node[anchor=east]{$\frac{9r^{2}+26r+15}{6(r+2)}$};
\filldraw[black] (7,5.1962) circle (0pt) node[anchor=west]{$\frac{9r^{2}+26r+15}{6(r+2)}$};
\filldraw[black] (5,3.333) circle (0pt) node[anchor=north]{$\frac{9r^{2}+26r+15}{6(r+2)}$};

\end{scope}

\end{tikzpicture}

\caption{The transformation of the entire gasket, using Figure \ref{fig:delta1} in Step 1.} \label{fig:delta2}
\end{figure}

We can compute the conductance and resistance in a similar way as the measure. Consider the $m$-cell $A = F_{j_1k_1}\circ F_{j_2k_2}\circ ... \circ F_{j_mk_m}(SG)$. Then for $x, y\in V_0$,
\begin{equation}
c(F_{j_1k_1}\circ ... \circ F_{j_mk_m}(x), F_{j_1k_1}\circ ... \circ F_{j_mk_m}(y)) = \frac{1}{r_0^{i(A)}}\left(\frac{1}{r_1^{m-i(A)}}\right)c(x, y)
\end{equation}

\begin{figure}
\centering
\begin{tiny}
\begin{tikzpicture}[scale=0.55]
\draw[black,thick] (0,0) -- (10,0);
\draw[black,thick] (0,0) -- (5,8.6603);
\draw[black,thick] (5,8.6603) -- (10,0);
\draw[black,thick] (1.25,2.1651) -- (3.75,2.1651);
\draw[black,thick] (2.5,4.3301) -- (7.5,4.3301);
\draw[black,thick] (3.75,6.4952) -- (6.25,6.4952);
\draw[black,thick] (3.75,2.1651) -- (2.5,0);
\draw[black,thick] (7.5,4.3301) -- (5,0);
\draw[black,thick] (8.75,2.1651) -- (7.5,0);
\draw[black,thick] (1.25,2.1651) -- (2.5,0);
\draw[black,thick] (2.5,4.3301) -- (5,0);
\draw[black,thick] (3.75,6.4952) -- (5,4.3301);
\draw[black,thick] (6.25,6.4952) -- (5,4.3301);
\draw[black,thick] (6.25,2.1651) -- (8.75,2.1651);
\draw[black,thick] (6.25,2.1651) -- (7.5,0);

\filldraw[black] (5,7) circle (0pt) node[anchor=south] {$F_{0}F_{0}K$};
\filldraw[black] (8.75,0.5) circle (0pt) node[anchor=south] {$F_{1}F_{1}K$};
\filldraw[black] (1.25,0.5) circle (0pt) node[anchor=south] {$F_{2}F_{2}K$};
\filldraw[black] (6.25,4.8) circle (0pt) node[anchor=south] {$F_{0}F_{1}K$};
\filldraw[black] (3.75,4.8) circle (0pt) node[anchor=south] {$F_{0}F_{2}K$};
\filldraw[black] (7.5,2.7) circle (0pt) node[anchor=south] {$F_{1}F_{0}K$};
\filldraw[black] (2.5,2.7) circle (0pt) node[anchor=south] {$F_{2}F_{0}K$};
\filldraw[black] (6.25,0.5) circle (0pt) node[anchor=south] {$F_{1}F_{2}K$};
\filldraw[black] (3.75,0.5) circle (0pt) node[anchor=south] {$F_{2}F_{1}K$};

\filldraw[black] (0,0) circle (2pt) node[anchor=east] {$q_{2}$};
\filldraw[black] (10,0) circle (2pt) node[anchor=west] {$q_{1}$};
\filldraw[black] (5,8.6603) circle (2pt) node[anchor=south] {$q_{0}$};
\filldraw[black] (2.5,4.3301) circle (2pt) node[anchor=east] {};
\filldraw[black] (7.5,4.3301) circle (2pt) node[anchor=west] {};
\filldraw[black] (5,0) circle (2pt) node[anchor=north] {};
\filldraw[black] (6.25,6.4952) circle (2pt) node[anchor=west] {};
\filldraw[black] (8.75,2.1651) circle (2pt) node[anchor=west] {};
\filldraw[black] (3.75,2.1651) circle (2pt) node[anchor=west] {};
\filldraw[black] (1.25,2.1651) circle (2pt) node[anchor=east] {};
\filldraw[black] (6.25,2.1651) circle (2pt) node[anchor=east] {};
\filldraw[black] (3.75,6.4952) circle (2pt) node[anchor=east] {};
\filldraw[black] (5,4.3301) circle (2pt) node[anchor=north] {};
\filldraw[black] (2.5,0) circle (2pt) node[anchor=north] {};
\filldraw[black] (7.5,0) circle (2pt) node[anchor=north] {};

\end{tikzpicture}
\begin{tikzpicture}[scale = 0.55]
\draw[black,thick] (0,0) -- (10,0);
\draw[black,thick] (0,0) -- (5,8.6603);
\draw[black,thick] (5,8.6603) -- (10,0);
\draw[black,thick] (1.25,2.1651) -- (3.75,2.1651);
\draw[black,thick] (2.5,4.3301) -- (7.5,4.3301);
\draw[black,thick] (3.75,6.4952) -- (6.25,6.4952);
\draw[black,thick] (3.75,2.1651) -- (2.5,0);
\draw[black,thick] (7.5,4.3301) -- (5,0);
\draw[black,thick] (8.75,2.1651) -- (7.5,0);
\draw[black,thick] (1.25,2.1651) -- (2.5,0);
\draw[black,thick] (2.5,4.3301) -- (5,0);
\draw[black,thick] (3.75,6.4952) -- (5,4.3301);
\draw[black,thick] (6.25,6.4952) -- (5,4.3301);
\draw[black,thick] (6.25,2.1651) -- (8.75,2.1651);
\draw[black,thick] (6.25,2.1651) -- (7.5,0);

\filldraw[black] (5,7) circle (0pt) node[anchor=south] {$\mu_{0}$};
\filldraw[black] (8.75,0.5) circle (0pt) node[anchor=south] {$\mu_{0}$};
\filldraw[black] (1.25,0.5) circle (0pt) node[anchor=south] {$\mu_{0}$};
\filldraw[black] (6.25,4.8) circle (0pt) node[anchor=south] {$\mu_{1}$};
\filldraw[black] (3.75,4.8) circle (0pt) node[anchor=south] {$\mu_{1}$};
\filldraw[black] (7.5,2.7) circle (0pt) node[anchor=south] {$\mu_{1}$};
\filldraw[black] (2.5,2.7) circle (0pt) node[anchor=south] {$\mu_{1}$};
\filldraw[black] (6.25,0.5) circle (0pt) node[anchor=south] {$\mu_{1}$};
\filldraw[black] (3.75,0.5) circle (0pt) node[anchor=south] {$\mu_{1}$};

\filldraw[black] (1.25,0) circle (0pt) node[anchor=north] {$r_{0}$};
\filldraw[black] (3.75,0) circle (0pt) node[anchor=north] {$r_{1}$};
\filldraw[black] (6.25,0) circle (0pt) node[anchor=north] {$r_{1}$};
\filldraw[black] (8.75,0) circle (0pt) node[anchor=north] {$r_{0}$};
\filldraw[black] (2.5,2.15) circle (0pt) node[anchor=north] {$r_{1}$};
\filldraw[black] (7.5,2.15) circle (0pt) node[anchor=north] {$r_{1}$};
\filldraw[black] (3.75,4.3001) circle (0pt) node[anchor=north] {$r_{1}$};
\filldraw[black] (6.25,4.3001) circle (0pt) node[anchor=north] {$r_{1}$};
\filldraw[black] (5,6.4952) circle (0pt) node[anchor=north] {$r_{0}$};
\filldraw[black] (0.625,1.0825) circle (0pt) node[anchor=east] {$r_{0}$};
\filldraw[black] (1.875,3.2475) circle (0pt) node[anchor=east] {$r_{1}$};
\filldraw[black] (3.125,5.41) circle (0pt) node[anchor=east] {$r_{1}$};
\filldraw[black] (4.325,7.58) circle (0pt) node[anchor=east] {$r_{0}$};
\filldraw[black] (9.375,1.0825) circle (0pt) node[anchor=west] {$r_{0}$};
\filldraw[black] (8.125,3.2475) circle (0pt) node[anchor=west] {$r_{1}$};
\filldraw[black] (6.875,5.41) circle (0pt) node[anchor=west] {$r_{1}$};
\filldraw[black] (5.675,7.58) circle (0pt) node[anchor=west] {$r_{0}$};
\filldraw[black] (4.35,5.45) circle (0pt) node[anchor=west] {$r_{1}$};
\filldraw[black] (5.65,5.45) circle (0pt) node[anchor=east] {$r_{1}$};
\filldraw[black] (3.125,1.0825) circle (0pt) node[anchor=east] {$r_{1}$};
\filldraw[black] (1.875,1.0825) circle (0pt) node[anchor=west] {$r_{0}$};
\filldraw[black] (4.375,1.0825) circle (0pt) node[anchor=west] {$r_{1}$};
\filldraw[black] (3.125,3.2475) circle (0pt) node[anchor=west] {$r_{1}$};
\filldraw[black] (5.625,1.0825) circle (0pt) node[anchor=east] {$r_{1}$};
\filldraw[black] (6.8750,1.0825) circle (0pt) node[anchor=west] {$r_{1}$};
\filldraw[black] (8.125,1.0825) circle (0pt) node[anchor=east] {$r_{0}$};
\filldraw[black] (6.875,3.2475) circle (0pt) node[anchor=east] {$r_{1}$};

\filldraw[black] (0,0) circle (2pt) node[anchor=east] {};
\filldraw[black] (10,0) circle (2pt) node[anchor=west] {};
\filldraw[black] (5,8.6603) circle (2pt) node[anchor=south] {};
\filldraw[black] (2.5,4.3301) circle (2pt) node[anchor=east] {};
\filldraw[black] (7.5,4.3301) circle (2pt) node[anchor=west] {};
\filldraw[black] (5,0) circle (2pt) node[anchor=north] {};
\filldraw[black] (6.25,6.4952) circle (2pt) node[anchor=west] {};
\filldraw[black] (8.75,2.1651) circle (2pt) node[anchor=west] {};
\filldraw[black] (3.75,2.1651) circle (2pt) node[anchor=west] {};
\filldraw[black] (1.25,2.1651) circle (2pt) node[anchor=east] {};
\filldraw[black] (6.25,2.1651) circle (2pt) node[anchor=east] {};
\filldraw[black] (3.75,6.4952) circle (2pt) node[anchor=east] {};
\filldraw[black] (5,4.3301) circle (2pt) node[anchor=north] {};
\filldraw[black] (2.5,0) circle (2pt) node[anchor=north] {};
\filldraw[black] (7.5,0) circle (2pt) node[anchor=north] {};

\end{tikzpicture}

\end{tiny}
\caption{Left: Construction of the twice-iterated SG Right: Assignment of measure and resistance} \label{fig:twiceitsg}
\end{figure}
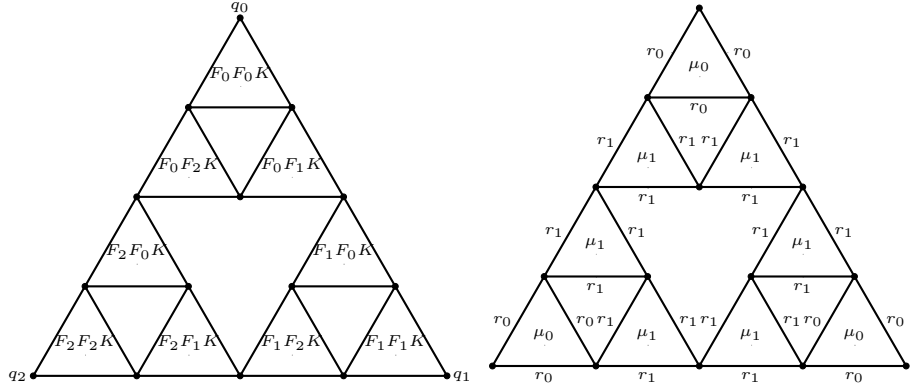

The measure and resistance distributions, along with a construction of the twice iterated SG, are shown in Figure \ref{fig:twiceitsg}. With the definition of conductance, we may define energy
\begin{align}
\mathcal{E}_m(u) &= \sum_{x\sim y} c^{(m)}(x, y)|u(x) - u(y)|^2 \\
\mathcal{E}(u) &= \lim_{m\to\infty} \mathcal{E}_m(u)
\end{align}
and a weak formulation of the Laplacian
\begin{equation}
\mathcal{E}(u, v) = -\int fv d\mu 
\end{equation}
where $u, v \in \text{dom}(\mathcal{E})$ and $f = \Delta^{(r)} u$, the Laplacian with parameter $r$. As usual, the pointwise formula for the Laplacian can be defined
\begin{align}  \label{pointwiselaplacian}
-\Delta_m^{(r)} u(x) &= \frac{1}{\int \Psi_m^{(x)} d\mu} \sum_{x\sim y} c^{(m)}(x, y) (u(x) - u(y)) && x\in V_m\setminus V_0\\ 
-\Delta^{(r)} u(x) &= \lim_{m\to\infty} -\Delta_m^{(r)} u(x) && x\in V_m\setminus V_0
\end{align}
where $\Psi_m^{(x)}(y) = \delta_{xy}$ for $y \in V_m$ and is piecewise harmonic in the complement of $V_{m}$. If the choice of $r$ is clear, we will abbreviate $\Delta_m^{(r)}$ and $\Delta^{(r)}$ to $\Delta_m$ or $\Delta$.

We can compute $\int \Psi_m^{(x)}d\mu$ (the \textit{pointmass} of $x$) explicitly. $\Psi_m^{(x)}$ has support in the two neighboring $m$-cells of $x$. Call these two cells $A$ and $B$. If $x_1, x_2, x_3 \in A$, then $\sum_i\int \Psi_m^{(x_i)}d\mu = \int_A 1 d\mu$. By symmetry, $\int \Psi_m^{(x_i)}d\mu$ are all equal. The same applies to the integrals for $B$. So 
\begin{equation}
\int \Psi_m^{(x)}d\mu = \frac{1}{3}(\mu(A) + \mu(B)) = \frac{1}{3}(\mu_0^{i(A)}\mu_1^{m-i(A)} + \mu_0^{i(B)}\mu_1^{m-i(B)})
\end{equation}

As in the standard case, our Laplacian is self-similar with the following identity
\begin{align}
-\Delta (u\circ F_{jj}) &= \frac{1}{r_0\mu_0} (-\Delta u)\circ F_{jj} \\
-\Delta (u\circ F_{jk}) &= \frac{1}{r_1\mu_1} (-\Delta u)\circ F_{jk} && j\neq k
\end{align} 

As on the interval, we will require that the renormalization factor is constant for any choice of contraction mappings we choose, i.e. $r_0\mu_0 = r_1\mu_1$. This means that there is only one choice of parameter that determines both the measure and the resistance. We will use $r=\frac{r_0}{r_1}$ as this parameter throughout this paper.

The renormalization factor $r_0\mu_0 = r_1\mu_1$ can be defined in terms of $r$. From now on, we will denote
\begin{equation}
\mu_{0}(r)r_{0}(r)=L(r) = \frac{2r(r+2)}{(2r+1)(9r^{2}+26r+15)}
\end{equation}
as the renormalization factor of $\Delta^{(r)}$. One interesting observation is that $L(r)$ has one global maximum at $r_{max}\approx 0.641677$ (and no other local extrema on $(0, \infty)$) which can be solved for analytically, but we will omit the calculations. This means that for all $r \neq r_{max}$, there exists exactly one $r' \neq r$ such that $L(r) = L(r')$. However, we did not find any significant properties relating $r'$ to $r$.

\section{Spectral Decimation on the Sierpinski Gasket}

We now seek to replicate the analysis performed in section \ref{intdec} on the more complicated structure of the Sierpinski Gasket. By lemma \ref{i(k)lemma}, which extends directly to SG, we can simplify the pointwise Laplacian formula \ref{pointwiselaplacian} based on the value of $i$ on adjacent cells. Letting $A_{0}$ and $A_{1}$ be $m$-cells with junction point $x$, and other vertices $y_{0},y_{1}$ and $y_{2},y_{3}$, respectively. Then the pointwise Laplacian of a function $f$ at $x$ can be written as

\begin{footnotesize}
\begin{align}
-\Delta_m f(x) &= \left(\frac{1}{\mu_{0}r_{0}}\right)^m \frac{3}{2}( (2f(x)-f(y_{0})-f(y_{1}))+(2f(x)-f(y_{2})-f(y_{3}))) &&\text{if }i(A_{0}) = i(A_{1}) \\
-\Delta_m f(x) &= \left(\frac{1}{\mu_{0}r_{0}}\right)^m\left(\frac{3}{2}\right) \frac{2}{\mu_{0}+\mu_{1}}(\mu_{0}( 2f(x)-f(y_{0})-f(y_{1}))+\mu_{1}(2f(x)-f(y_{2})-f(y_{3}))) &&\text{if } i(A_{0}) = i(A_{1})+1 \\
-\Delta_m f(x) &= \left(\frac{1}{\mu_{0}r_{0}}\right)^m\left(\frac{3}{2}\right) \frac{2}{\mu_{0}+\mu_{1}}(\mu_{1} (2f(x)-f(y_{0})-f(y_{1}))+\mu_{0}(2f(x)-f(y_{2})-f(y_{3}))) &&\text{if } i(A_{0}) = i(A_{1})-1
\end{align}
\end{footnotesize}

For now we will omit the renormalization factor $\left(\frac{1}{\mu_{0}r_{0}}\right)^m$ and later rescale our eigenvalues by this constant term. To develop an algorithm to extend an eigenfunction on $V_{m}$ with eigenvalue $\lambda_{m}$ to an eigenfunction on $V_{m+1}$ with new eigenvalue $\lambda_{m+1}$ we formulate an analagous system to \cite{Strichartz}. 

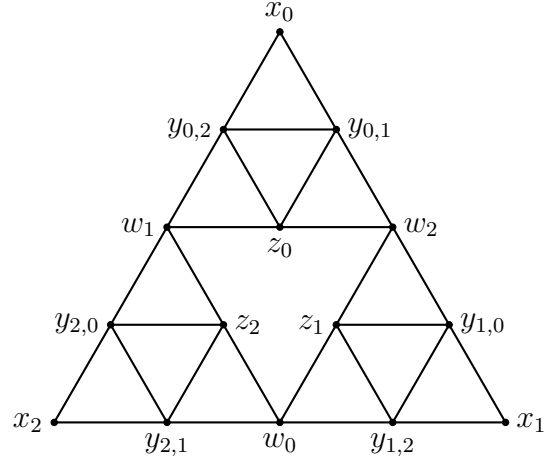
\begin{figure}
\centering
\begin{tikzpicture}[scale=0.6]
\draw[black,thick] (0,0) -- (10,0);
\draw[black,thick] (0,0) -- (5,8.6603);
\draw[black,thick] (5,8.6603) -- (10,0);
\draw[black,thick] (1.25,2.1651) -- (3.75,2.1651);
\draw[black,thick] (2.5,4.3301) -- (7.5,4.3301);
\draw[black,thick] (3.75,6.4952) -- (6.25,6.4952);
\draw[black,thick] (3.75,2.1651) -- (2.5,0);
\draw[black,thick] (7.5,4.3301) -- (5,0);
\draw[black,thick] (8.75,2.1651) -- (7.5,0);
\draw[black,thick] (1.25,2.1651) -- (2.5,0);
\draw[black,thick] (2.5,4.3301) -- (5,0);
\draw[black,thick] (3.75,6.4952) -- (5,4.3301);
\draw[black,thick] (6.25,6.4952) -- (5,4.3301);
\draw[black,thick] (6.25,2.1651) -- (8.75,2.1651);
\draw[black,thick] (6.25,2.1651) -- (7.5,0);

\filldraw[black] (0,0) circle (2pt) node[anchor=east] {$x_{2}$};
\filldraw[black] (10,0) circle (2pt) node[anchor=west] {$x_{1}$};
\filldraw[black] (5,8.6603) circle (2pt) node[anchor=south] {$x_{0}$};
\filldraw[black] (2.5,4.3301) circle (2pt) node[anchor=east] {$w_{1}$};
\filldraw[black] (7.5,4.3301) circle (2pt) node[anchor=west] {$w_{2}$};
\filldraw[black] (5,0) circle (2pt) node[anchor=north] {$w_{0}$};
\filldraw[black] (6.25,6.4952) circle (2pt) node[anchor=west] {$y_{0,1}$};
\filldraw[black] (8.75,2.1651) circle (2pt) node[anchor=west] {$y_{1,0}$};
\filldraw[black] (3.75,2.1651) circle (2pt) node[anchor=west] {$z_{2}$};
\filldraw[black] (1.25,2.1651) circle (2pt) node[anchor=east] {$y_{2,0}$};
\filldraw[black] (6.25,2.1651) circle (2pt) node[anchor=east] {$z_{1}$};
\filldraw[black] (3.75,6.4952) circle (2pt) node[anchor=east] {$y_{0,2}$};
\filldraw[black] (5,4.3301) circle (2pt) node[anchor=north] {$z_{0}$};
\filldraw[black] (2.5,0) circle (2pt) node[anchor=north] {$y_{2,1}$};
\filldraw[black] (7.5,0) circle (2pt) node[anchor=north] {$y_{1,2}$};

\end{tikzpicture}
\caption{A single $m$-cell with vertices $x_{0},x_{1},x_{2}$}\label{fig:sgmcell}
\end{figure}

Given an eigenfunction $u_{m}$ on level $m$, consider any $m$-cell with vertices $x_{0},x_{1},x_{2}$, as in Figure \ref{fig:sgmcell}. The extension to $V_{m+1}$ creates 12 new points. To ensure that our extension is an eigenfunction, we mandate that 

\begin{equation}
-\Delta_{m+1} u_{m+1}(x) = \lambda_{m+1}u_{m+1}(x)\quad \forall \ x\in V_{m+1}\setminus V_0
\end{equation}

Evaluating this equation on the 12 points produced via subdivision yields a 12-equation, 16-variable system that can be solved algebraically to yield $y_{i,j},w_{i}$, and $z_{i}$ as functions of $x_{0},x_{1},x_{2},\lambda_{m+1}$. For readability we will refer to $u(a)$ as $a$. These functions can be written as

\newpage

\begin{align}
w_{0}(x_{0},x_{1},x_{2},\lambda_{m+1},r) &= \frac{81x_{0}(-3+(2r+r^{2})(\lambda_{m+1}-9)+\lambda_{m+1})}{\gamma(r,\lambda_{m+1})} \nonumber
\\
&+\frac{9(x_{1}+x_{2})(-189+135\lambda_{m+1}-30\lambda_{m+1}^{2}+2\lambda_{m+1}^{3})}{\gamma(r,\lambda_{m+1})} \nonumber
\\
&+\frac{9(x_{1}+x_{2})(r^{2}(-81+177\lambda_{m+1}-30\lambda_{m+1}^{2}+2\lambda_{m+1}^{3}))}{\gamma(r,\lambda_{m+1})} \nonumber
\\
&+\frac{9(x_{1}+x_{2})(2r(-135+135\lambda-30\lambda_{m+1}^{2}+2\lambda_{m+1}^{3}))}{\gamma(r,\lambda_{m+1})} \label{ext1}
\\
z_{0}(x_{0},x_{1},x_{2},\lambda_{m+1},r) &= \frac{-9(x_{1}+x_{2})(54-27\lambda_{m+1}+3\lambda_{m+1}^{2}+r^{2}(81-36\lambda_{m+1}+3\lambda_{m+1}^{2}))}{\gamma(r,\lambda_{m+1})} \nonumber
\\
&+\frac{-9r(x_{1}+x_{2})(189-63\lambda_{m+1}+6\lambda_{m+1}^{2})}{\gamma(r,\lambda_{m+1})} \nonumber
\\
&+\frac{9x_{0}(-297+225\lambda_{m+1}+r^{2}(-81+171\lambda_{m+1}-54\lambda_{m+1}^{2}+4\lambda_{m+1}^{3}))}{\gamma(r,\lambda_{m+1})} \nonumber
\\
&+\frac{9x_{0}(-54\lambda_{m+1}^{2}+4\lambda_{m+1}^{3}+r(-324+432\lambda_{m+1}-108\lambda_{m+1}^{2}+8\lambda_{m+1}))}{\gamma(r,\lambda_{m+1})} \label{ext2}
\\
y_{0,1}(x_{0},x_{1},x_{2},\lambda_{m+1},r) &= \frac{-3x_{0}(\lambda_{m+1}-3)^{2}(135-48\lambda_{m+1}+4\lambda_{m+1}^{2})}{\gamma(r,\lambda_{m+1})} \nonumber
\\
&+\frac{-3r^{2}x_{0}(243-756\lambda_{m+1}+405\lambda_{m+1}^{2}-72\lambda_{m+1}^{3}+4\lambda_{m+1}^{4})}{\gamma(r,\lambda_{m+1})} \nonumber
\\
&+\frac{-3rx_{0}(1134-2106\lambda_{m+1}+900\lambda_{m+1}^{2}-144\lambda_{m+1}^{3}+8\lambda_{m+1}^{4})}{\gamma(r,\lambda_{m+1})} \nonumber
\\
&+\frac{-3r x_{2}(405-r(27\lambda_{m+1}-243)-81\lambda_{m+1})}{\gamma(r,\lambda_{m+1})} \nonumber
\\
&+\frac{-3r x_{1}(567-189\lambda_{m+1}+18\lambda_{m+1}^{2}+r(243-189\lambda_{m+1}+18\lambda_{m+1}^{2}))}{\gamma(r,\lambda_{m+1})} \label{ext3}
\end{align}

Similar equations for the remaining points can be obtained by permuting the indices $0,1,2$. The term $\gamma(r,\lambda_{m+1})$ is defined as

\begin{align}
\gamma(r,\lambda_{m+1}) &= (9-3(2+3r)\lambda_{m+1}+(1+r)\lambda_{m+1}^{2})(-405+279\lambda_{m+1}-60\lambda_{m+1}^{2}+4\lambda_{m+1}^{3}) \nonumber
\\
&+ r(9-3(2+3r)\lambda_{m+1}+(1+r)\lambda_{m+1}^{2})(-702+558\lambda_{m+1}-120\lambda_{m+1}^{2}+8\lambda_{m+1}^{3}) \nonumber
\\
&+ r^{2}(9-3(2+3r)\lambda_{m+1}+(1+r)\lambda_{m+1}^{2})(-243+243\lambda_{m+1}-60\lambda_{m+1}^{2}+4\lambda_{m+1}^{3})
\end{align}

a polynomial of degree 3 in $r$ and degree 5 in $\lambda_{m+1}$. Note that the extension equations are valid for any $r,\lambda_{m+1}$ such that $\gamma(r,\lambda_{m+1})\neq0$. I turns out that $\gamma$ belongs to teh special class of invertible quintics, and so in order to ensure $\gamma(r,\lambda_{m+1})\neq0$ we follow the strategy of \cite{Strichartz} and record forbidden eigenvalues $b_{1}(r),b_{2}(r),b_{3}(r),b_{4}(r),b_{5}(r)$ as the roots of $\gamma(r,\lambda_{m+1})$ in $\lambda_{m+1}$ as a function of $r$ such that 
\begin{equation}
b_1(1) < b_4(1) < b_5(1) < b_2(1) < b_3(1)
\end{equation}

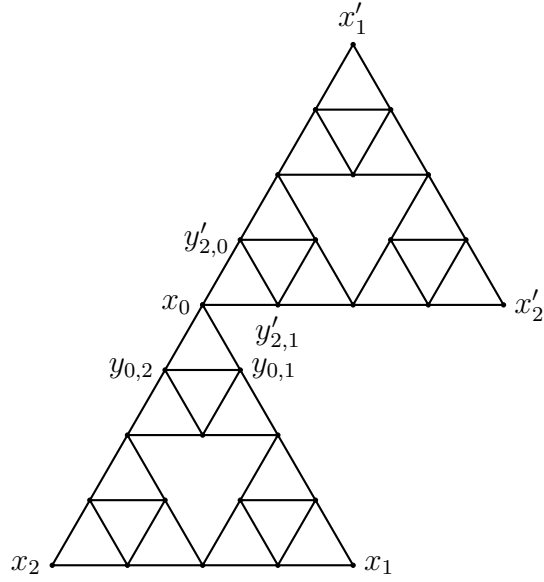
\begin{figure} 
\centering
\begin{tikzpicture}[scale=0.4]

\begin{scope}[xshift = 5cm,yshift=8.6603cm]

\draw[black,thick] (0,0) -- (10,0);
\draw[black,thick] (0,0) -- (5,8.6603);
\draw[black,thick] (5,8.6603) -- (10,0);
\draw[black,thick] (1.25,2.1651) -- (3.75,2.1651);
\draw[black,thick] (2.5,4.3301) -- (7.5,4.3301);
\draw[black,thick] (3.75,6.4952) -- (6.25,6.4952);
\draw[black,thick] (3.75,2.1651) -- (2.5,0);
\draw[black,thick] (7.5,4.3301) -- (5,0);
\draw[black,thick] (8.75,2.1651) -- (7.5,0);
\draw[black,thick] (1.25,2.1651) -- (2.5,0);
\draw[black,thick] (2.5,4.3301) -- (5,0);
\draw[black,thick] (3.75,6.4952) -- (5,4.3301);
\draw[black,thick] (6.25,6.4952) -- (5,4.3301);
\draw[black,thick] (6.25,2.1651) -- (8.75,2.1651);
\draw[black,thick] (6.25,2.1651) -- (7.5,0);

\filldraw[black] (0,0) circle (2pt) node[anchor=east] {};
\filldraw[black] (10,0) circle (2pt) node[anchor=west] {$x_{2}'$};
\filldraw[black] (5,8.6603) circle (2pt) node[anchor=south] {$x_{1}'$};
\filldraw[black] (2.5,4.3301) circle (2pt) node[anchor=east] {};
\filldraw[black] (7.5,4.3301) circle (2pt) node[anchor=west] {};
\filldraw[black] (5,0) circle (2pt) node[anchor=north] {};
\filldraw[black] (6.25,6.4952) circle (2pt) node[anchor=west] {};
\filldraw[black] (8.75,2.1651) circle (2pt) node[anchor=west] {};
\filldraw[black] (3.75,2.1651) circle (2pt) node[anchor=west] {};
\filldraw[black] (1.25,2.1651) circle (2pt) node[anchor=east] {$y_{2,0}'$};
\filldraw[black] (6.25,2.1651) circle (2pt) node[anchor=east] {};
\filldraw[black] (3.75,6.4952) circle (2pt) node[anchor=east] {};
\filldraw[black] (5,4.3301) circle (2pt) node[anchor=north] {};
\filldraw[black] (2.5,0) circle (2pt) node[anchor=north] {$y_{2,1}'$};
\filldraw[black] (7.5,0) circle (2pt) node[anchor=north] {};

\end{scope}

\draw[black,thick] (0,0) -- (10,0);
\draw[black,thick] (0,0) -- (5,8.6603);
\draw[black,thick] (5,8.6603) -- (10,0);
\draw[black,thick] (1.25,2.1651) -- (3.75,2.1651);
\draw[black,thick] (2.5,4.3301) -- (7.5,4.3301);
\draw[black,thick] (3.75,6.4952) -- (6.25,6.4952);
\draw[black,thick] (3.75,2.1651) -- (2.5,0);
\draw[black,thick] (7.5,4.3301) -- (5,0);
\draw[black,thick] (8.75,2.1651) -- (7.5,0);
\draw[black,thick] (1.25,2.1651) -- (2.5,0);
\draw[black,thick] (2.5,4.3301) -- (5,0);
\draw[black,thick] (3.75,6.4952) -- (5,4.3301);
\draw[black,thick] (6.25,6.4952) -- (5,4.3301);
\draw[black,thick] (6.25,2.1651) -- (8.75,2.1651);
\draw[black,thick] (6.25,2.1651) -- (7.5,0);

\filldraw[black] (0,0) circle (2pt) node[anchor=east] {$x_{2}$};
\filldraw[black] (10,0) circle (2pt) node[anchor=west] {$x_{1}$};
\filldraw[black] (5,8.6603) circle (2pt) node[anchor=east] {$x_{0}$};
\filldraw[black] (2.5,4.3301) circle (2pt) node[anchor=east] {};
\filldraw[black] (7.5,4.3301) circle (2pt) node[anchor=west] {};
\filldraw[black] (5,0) circle (2pt) node[anchor=north] {};
\filldraw[black] (6.25,6.4952) circle (2pt) node[anchor=west] {$y_{0,1}$};
\filldraw[black] (8.75,2.1651) circle (2pt) node[anchor=west] {};
\filldraw[black] (3.75,2.1651) circle (2pt) node[anchor=west] {};
\filldraw[black] (1.25,2.1651) circle (2pt) node[anchor=east] {};
\filldraw[black] (6.25,2.1651) circle (2pt) node[anchor=east] {};
\filldraw[black] (3.75,6.4952) circle (2pt) node[anchor=east] {$y_{0,2}$};
\filldraw[black] (5,4.3301) circle (2pt) node[anchor=north] {};
\filldraw[black] (2.5,0) circle (2pt) node[anchor=north] {};
\filldraw[black] (7.5,0) circle (2pt) node[anchor=north] {};

\end{tikzpicture}
\caption{The $m$-cells around $x_0$}\label{fig:sgmcells}
\end{figure}

These formulas were designed to extend an eigenfunction to $V_{m+1}$, but we must verify that the eigenfunction equation also holds on $V_{m}$ -- that is, the points $x_{0},x_{1},x_{2}$. Specifically consider $x_{0}$, and in addition to the $m$-cell considered above, consider the other $m$-cell with vertex $x_{0}$ as seen in Figure \ref{fig:sgmcells}. We must make the pointwise Laplacian hold at $x_{0}$ -- there are 3 distinct cases based upon the value of $i$ on the two $m$-cells, but the algebraic result is the same for all three cases. In the simplest case, we are given
\begin{equation}
\lambda_{m}x_{0} = \frac{3}{2}( (2x_{0}-x_{2}-x_{1})+(2x_{0}-x_{2}'-x_{1}'))
\end{equation}

by the $m$-level eigenvalue equation, and seek to verify that

\begin{align}
\lambda_{m+1}x_{0} &= \frac{3}{2} (2x_{0}-y_{0,1}(x_{0},x_{1},x_{2},\lambda_{m+1},r)-y_{0,2}(x_{0},x_{1},x_{2},\lambda_{m+1},r)) \nonumber
\\
&+\frac{3}{2}(2x_{0}-y_{2,0}(x_{1}',x_{2}',x_{0},\lambda_{m+1},r)-y_{2,1}(x_{1}',x_{2}',x_{0},\lambda_{m+1},r)) \label{m+1cell}
\end{align}

Mandating that both above conditions hold allows the formulation of a condition on $\lambda_{m}$ as a function of $\lambda_{m+1}$ and $r$ -- in particular we require that

\begin{align}
\lambda_{m}(\lambda_{m+1},r) &= \frac{-\lambda_{m+1}(\lambda_{m+1}-3)^{2}(135-48\lambda_{m+1}+4\lambda_{m+1}^{2})}{54r(-6+r(\lambda_{m+1}-3)+\lambda_{m+1})} \nonumber
\\
&+\frac{-r^{3}\lambda_{m+1}(1458-1701\lambda_{m+1}+603\lambda_{m+1}^{2}-84\lambda_{m+1}^{3}+4\lambda_{m+1}^{4})}{54r(-6+r(\lambda_{m+1}-3)+\lambda_{m+1})} \nonumber
\\
&+\frac{-r^{2}\lambda_{m+1}(4941-5022\lambda_{m+1}+1701\lambda_{m+1}^{2}-240\lambda_{m+1}^{3}+12\lambda_{m+1}^{4})}{54r(-6+r(\lambda_{m+1}-3)+\lambda_{m+1})} \nonumber
\\
&+\frac{-r\lambda_{m+1}(4536-4455\lambda_{m+1}+1557\lambda_{m+1}^{2}-228\lambda_{m+1}^{3}+12\lambda_{m+1}^{4})}{54r(-6+r(\lambda_{m+1}-3)+\lambda_{m+1})} \label{lambdamap}
\end{align}

holds. This equation is quintic in $\lambda_{m+1}$, and so we must compute numerical inverses. We denote these 5 inverses by $\Phi^{i}(\lambda_{m},r)$, with $\Phi^{i}(\lambda_{m},r) \leq \Phi^{i+1}(\lambda_{m},r) \ \forall \ \lambda_{m}$ for fixed $r$. We must also add a sixth forbidden eigenvalue, $b_{6}=9$. If $\lambda_{m+1}=9$, then (\ref{m+1cell}) is satisfied regardless of the value of $\lambda_{m}$. These equations provide a recipe for extension of existing eigenvalues; but just as in the case of the standard gasket, we must also account for eigenfunction/values that cannot be produced by this decimation process. These equations were derived independent of any boundary conditions. For the remainder of this section, we will provide a detailed analysis and counting argument for the eigenvalues and functions which are born on each level, to demonstrate that we have described the complete spectrum, under the Dirichlet boundary assumption. We believe that analagous analysis for the Neumann boundary assumption is possible.

\begin{figure}
\centering
  \begin{subfigure}{\linewidth}
  \centering
  \begin{tikzpicture}[scale=0.4]
\draw[black,thick] (0,0) -- (10,0);
\draw[black,thick] (0,0) -- (5,8.6603);
\draw[black,thick] (5,8.6603) -- (10,0);
\draw[black,thick] (1.25,2.1651) -- (3.75,2.1651);
\draw[black,thick] (2.5,4.3301) -- (7.5,4.3301);
\draw[black,thick] (3.75,6.4952) -- (6.25,6.4952);
\draw[black,thick] (3.75,2.1651) -- (2.5,0);
\draw[black,thick] (7.5,4.3301) -- (5,0);
\draw[black,thick] (8.75,2.1651) -- (7.5,0);
\draw[black,thick] (1.25,2.1651) -- (2.5,0);
\draw[black,thick] (2.5,4.3301) -- (5,0);
\draw[black,thick] (3.75,6.4952) -- (5,4.3301);
\draw[black,thick] (6.25,6.4952) -- (5,4.3301);
\draw[black,thick] (6.25,2.1651) -- (8.75,2.1651);
\draw[black,thick] (6.25,2.1651) -- (7.5,0);

\filldraw[black] (0,0) circle (2pt) node[anchor=east] {$0$};
\filldraw[black] (10,0) circle (2pt) node[anchor=west] {$0$};
\filldraw[black] (5,8.6603) circle (2pt) node[anchor=south] {$0$};
\filldraw[black] (2.5,4.3301) circle (2pt) node[anchor=east] {$1$};
\filldraw[black] (7.5,4.3301) circle (2pt) node[anchor=west] {$1$};
\filldraw[black] (5,0) circle (2pt) node[anchor=north] {$1$};
\filldraw[black] (6.25,6.4952) circle (2pt) node[anchor=west] {$a$};
\filldraw[black] (8.75,2.1651) circle (2pt) node[anchor=west] {$a$};
\filldraw[black] (3.75,2.1651) circle (2pt) node[anchor=west] {$1$};
\filldraw[black] (1.25,2.1651) circle (2pt) node[anchor=east] {$a$};
\filldraw[black] (6.25,2.1651) circle (2pt) node[anchor=east] {$1$};
\filldraw[black] (3.75,6.4952) circle (2pt) node[anchor=east] {$a$};
\filldraw[black] (5,4.3301) circle (2pt) node[anchor=north] {$1$};
\filldraw[black] (2.5,0) circle (2pt) node[anchor=north] {$a$};
\filldraw[black] (7.5,0) circle (2pt) node[anchor=north] {$a$};

\end{tikzpicture}
\begin{tikzpicture}[scale=0.4]
\draw[black,thick] (0,0) -- (10,0);
\draw[black,thick] (0,0) -- (5,8.6603);
\draw[black,thick] (5,8.6603) -- (10,0);
\draw[black,thick] (1.25,2.1651) -- (3.75,2.1651);
\draw[black,thick] (2.5,4.3301) -- (7.5,4.3301);
\draw[black,thick] (3.75,6.4952) -- (6.25,6.4952);
\draw[black,thick] (3.75,2.1651) -- (2.5,0);
\draw[black,thick] (7.5,4.3301) -- (5,0);
\draw[black,thick] (8.75,2.1651) -- (7.5,0);
\draw[black,thick] (1.25,2.1651) -- (2.5,0);
\draw[black,thick] (2.5,4.3301) -- (5,0);
\draw[black,thick] (3.75,6.4952) -- (5,4.3301);
\draw[black,thick] (6.25,6.4952) -- (5,4.3301);
\draw[black,thick] (6.25,2.1651) -- (8.75,2.1651);
\draw[black,thick] (6.25,2.1651) -- (7.5,0);

\filldraw[black] (5,7) circle (0pt) node[anchor=south] {$0$};
\filldraw[black] (8.75,0.5) circle (0pt) node[anchor=south] {$0$};
\filldraw[black] (1.25,0.5) circle (0pt) node[anchor=south] {$0$};
\filldraw[black] (6.25,4.65) circle (0pt) node[anchor=south] {$A$};
\filldraw[black] (3.75,4.65) circle (0pt) node[anchor=south] {$-A$};
\filldraw[black] (7.5,2.45) circle (0pt) node[anchor=south] {$-A$};
\filldraw[black] (2.5,2.45) circle (0pt) node[anchor=south] {$A$};
\filldraw[black] (6.25,0.15) circle (0pt) node[anchor=south] {$A$};
\filldraw[black] (3.75,0.15) circle (0pt) node[anchor=south] {$-A$};

\filldraw[black] (0,0) circle (2pt) node[anchor=east] {$ $};
\filldraw[black] (10,0) circle (2pt) node[anchor=west] {$ $};
\filldraw[black] (5,8.6603) circle (2pt) node[anchor=south] {$ $};
\filldraw[black] (2.5,4.3301) circle (2pt) node[anchor=east] {$ $};
\filldraw[black] (7.5,4.3301) circle (2pt) node[anchor=west] {$ $};
\filldraw[black] (5,0) circle (2pt) node[anchor=north] {$ $};
\filldraw[black] (6.25,6.4952) circle (2pt) node[anchor=west] {$ $};
\filldraw[black] (8.75,2.1651) circle (2pt) node[anchor=west] {$ $};
\filldraw[black] (3.75,2.1651) circle (2pt) node[anchor=west] {$ $};
\filldraw[black] (1.25,2.1651) circle (2pt) node[anchor=east] {$ $};
\filldraw[black] (6.25,2.1651) circle (2pt) node[anchor=east] {$ $};
\filldraw[black] (3.75,6.4952) circle (2pt) node[anchor=east] {$ $};
\filldraw[black] (5,4.3301) circle (2pt) node[anchor=north] {$ $};
\filldraw[black] (2.5,0) circle (2pt) node[anchor=north] {$ $};
\filldraw[black] (7.5,0) circle (2pt) node[anchor=north] {$ $};

\end{tikzpicture} 
  \caption{Left: Eigenfunction for $b_{1},b_{2}$ Right: Eigenfunction for $b_{3},b_{4},b_{5}$}
  \label{fig:eigfuncdec1}
  \end{subfigure}
  \begin{subfigure}{\linewidth}
  \centering
    \begin{tikzpicture}[scale=0.4]
\draw[black,thick] (0,0) -- (10,0);
\draw[black,thick] (0,0) -- (5,8.6603);
\draw[black,thick] (5,8.6603) -- (10,0);
\draw[black,thick] (1.25,2.1651) -- (3.75,2.1651);
\draw[black,thick] (2.5,4.3301) -- (7.5,4.3301);
\draw[black,thick] (3.75,6.4952) -- (6.25,6.4952);
\draw[black,thick] (3.75,2.1651) -- (2.5,0);
\draw[black,thick] (7.5,4.3301) -- (5,0);
\draw[black,thick] (8.75,2.1651) -- (7.5,0);
\draw[black,thick] (1.25,2.1651) -- (2.5,0);
\draw[black,thick] (2.5,4.3301) -- (5,0);
\draw[black,thick] (3.75,6.4952) -- (5,4.3301);
\draw[black,thick] (6.25,6.4952) -- (5,4.3301);
\draw[black,thick] (6.25,2.1651) -- (8.75,2.1651);
\draw[black,thick] (6.25,2.1651) -- (7.5,0);

\draw[dashed] (1.25,2.5) -- (1.25,-0.5);
\draw[dashed] (3.75,2.5) -- (3.75,-0.5);
\draw[dashed] (6.25,2.5) -- (6.25,-0.5);
\draw[dashed] (8.75,2.5) -- (8.75,-0.5);

\filldraw[black] (5,7) circle (0pt) node[anchor=south] {$0$};
\filldraw[black] (8.75,0.15) circle (0pt) node[anchor=south] {$-A$};
\filldraw[black] (1.25,0.15) circle (0pt) node[anchor=south] {$A$};
\filldraw[black] (6.25,4.8) circle (0pt) node[anchor=south] {$0$};
\filldraw[black] (3.75,4.8) circle (0pt) node[anchor=south] {$0$};
\filldraw[black] (7.5,2.7) circle (0pt) node[anchor=south] {$0$};
\filldraw[black] (2.5,2.7) circle (0pt) node[anchor=south] {$0$};
\filldraw[black] (6.25,0.15) circle (0pt) node[anchor=south] {$A$};
\filldraw[black] (3.75,0.15) circle (0pt) node[anchor=south] {$-A$};

\filldraw[black] (0,0) circle (2pt) node[anchor=east] {$ $};
\filldraw[black] (10,0) circle (2pt) node[anchor=west] {$ $};
\filldraw[black] (5,8.6603) circle (2pt) node[anchor=south] {$ $};
\filldraw[black] (2.5,4.3301) circle (2pt) node[anchor=east] {$ $};
\filldraw[black] (7.5,4.3301) circle (2pt) node[anchor=west] {$ $};
\filldraw[black] (5,0) circle (2pt) node[anchor=north] {$ $};
\filldraw[black] (6.25,6.4952) circle (2pt) node[anchor=west] {$ $};
\filldraw[black] (8.75,2.1651) circle (2pt) node[anchor=west] {$ $};
\filldraw[black] (3.75,2.1651) circle (2pt) node[anchor=west] {$ $};
\filldraw[black] (1.25,2.1651) circle (2pt) node[anchor=east] {$ $};
\filldraw[black] (6.25,2.1651) circle (2pt) node[anchor=east] {$ $};
\filldraw[black] (3.75,6.4952) circle (2pt) node[anchor=east] {$ $};
\filldraw[black] (5,4.3301) circle (2pt) node[anchor=north] {$ $};
\filldraw[black] (2.5,0) circle (2pt) node[anchor=north] {$ $};
\filldraw[black] (7.5,0) circle (2pt) node[anchor=north] {$ $};

\end{tikzpicture}
\begin{tikzpicture}[scale=0.4]
\draw[black,thick] (0,0) -- (10,0);
\draw[black,thick] (0,0) -- (5,8.6603);
\draw[black,thick] (5,8.6603) -- (10,0);
\draw[black,thick] (1.25,2.1651) -- (3.75,2.1651);
\draw[black,thick] (2.5,4.3301) -- (7.5,4.3301);
\draw[black,thick] (3.75,6.4952) -- (6.25,6.4952);
\draw[black,thick] (3.75,2.1651) -- (2.5,0);
\draw[black,thick] (7.5,4.3301) -- (5,0);
\draw[black,thick] (8.75,2.1651) -- (7.5,0);
\draw[black,thick] (1.25,2.1651) -- (2.5,0);
\draw[black,thick] (2.5,4.3301) -- (5,0);
\draw[black,thick] (3.75,6.4952) -- (5,4.3301);
\draw[black,thick] (6.25,6.4952) -- (5,4.3301);
\draw[black,thick] (6.25,2.1651) -- (8.75,2.1651);
\draw[black,thick] (6.25,2.1651) -- (7.5,0);

\begin{scope}[cm={cos(-60) ,-sin(-60) ,sin(-60) ,cos(-60) ,(1.8 cm,-1.05 cm)}]
\draw[dashed] (1.25,2.5) -- (1.25,-0.5);
\draw[dashed] (3.75,2.5) -- (3.75,-0.5);
\draw[dashed] (6.25,2.5) -- (6.25,-0.5);
\draw[dashed] (8.75,2.5) -- (8.75,-0.5);
\end{scope}

\begin{scope}[xshift = -0.3 cm,yshift = -0.3 cm]
\filldraw[black] (2.5,2.7) circle (0pt) node[anchor=south] {$A$};
\filldraw[black] (5,7) circle (0pt) node[anchor=south] {$A$};
\filldraw[black] (1.25,0.5) circle (0pt) node[anchor=south] {$-A$};
\filldraw[black] (3.75,4.8) circle (0pt) node[anchor=south] {$-A$};
\end{scope}

\filldraw[black] (8.75,0.5) circle (0pt) node[anchor=south] {$0$};
\filldraw[black] (6.25,4.8) circle (0pt) node[anchor=south] {$0$};
\filldraw[black] (7.5,2.7) circle (0pt) node[anchor=south] {$0$};
\filldraw[black] (6.25,0.5) circle (0pt) node[anchor=south] {$0$};
\filldraw[black] (3.75,0.5) circle (0pt) node[anchor=south] {$0$};

\filldraw[black] (0,0) circle (2pt) node[anchor=east] {$ $};
\filldraw[black] (10,0) circle (2pt) node[anchor=west] {$ $};
\filldraw[black] (5,8.6603) circle (2pt) node[anchor=south] {$ $};
\filldraw[black] (2.5,4.3301) circle (2pt) node[anchor=east] {$ $};
\filldraw[black] (7.5,4.3301) circle (2pt) node[anchor=west] {$ $};
\filldraw[black] (5,0) circle (2pt) node[anchor=north] {$ $};
\filldraw[black] (6.25,6.4952) circle (2pt) node[anchor=west] {$ $};
\filldraw[black] (8.75,2.1651) circle (2pt) node[anchor=west] {$ $};
\filldraw[black] (3.75,2.1651) circle (2pt) node[anchor=west] {$ $};
\filldraw[black] (1.25,2.1651) circle (2pt) node[anchor=east] {$ $};
\filldraw[black] (6.25,2.1651) circle (2pt) node[anchor=east] {$ $};
\filldraw[black] (3.75,6.4952) circle (2pt) node[anchor=east] {$ $};
\filldraw[black] (5,4.3301) circle (2pt) node[anchor=north] {$ $};
\filldraw[black] (2.5,0) circle (2pt) node[anchor=north] {$ $};
\filldraw[black] (7.5,0) circle (2pt) node[anchor=north] {$ $};

\end{tikzpicture}
  \caption{Two independent eigenfunctions for $b_{3},b_{4},b_{5}$}
  \label{fig:eigfuncdec2}
  \end{subfigure}
    \begin{subfigure}{\linewidth}
    \centering
\begin{tikzpicture}[scale=0.4]
\draw[black,thick] (0,0) -- (10,0);
\draw[black,thick] (0,0) -- (5,8.6603);
\draw[black,thick] (5,8.6603) -- (10,0);
\draw[black,thick] (1.25,2.1651) -- (3.75,2.1651);
\draw[black,thick] (2.5,4.3301) -- (7.5,4.3301);
\draw[black,thick] (3.75,6.4952) -- (6.25,6.4952);
\draw[black,thick] (3.75,2.1651) -- (2.5,0);
\draw[black,thick] (7.5,4.3301) -- (5,0);
\draw[black,thick] (8.75,2.1651) -- (7.5,0);
\draw[black,thick] (1.25,2.1651) -- (2.5,0);
\draw[black,thick] (2.5,4.3301) -- (5,0);
\draw[black,thick] (3.75,6.4952) -- (5,4.3301);
\draw[black,thick] (6.25,6.4952) -- (5,4.3301);
\draw[black,thick] (6.25,2.1651) -- (8.75,2.1651);
\draw[black,thick] (6.25,2.1651) -- (7.5,0);

\filldraw[black] (0,0) circle (2pt) node[anchor=east] {$0$};
\filldraw[black] (10,0) circle (2pt) node[anchor=west] {$0$};
\filldraw[black] (5,8.6603) circle (2pt) node[anchor=south] {$0$};
\filldraw[black] (2.5,4.3301) circle (2pt) node[anchor=east] {$0$};
\filldraw[black] (7.5,4.3301) circle (2pt) node[anchor=west] {$0$};
\filldraw[black] (5,0) circle (2pt) node[anchor=north] {$2$};
\filldraw[black] (6.25,6.4952) circle (2pt) node[anchor=west] {$0$};
\filldraw[black] (8.75,2.1651) circle (2pt) node[anchor=west] {$1$};
\filldraw[black] (3.75,2.1651) circle (2pt) node[anchor=west] {$-1$};
\filldraw[black] (1.25,2.1651) circle (2pt) node[anchor=east] {$1$};
\filldraw[black] (6.25,2.1651) circle (2pt) node[anchor=east] {$-1$};
\filldraw[black] (3.75,6.4952) circle (2pt) node[anchor=east] {$0$};
\filldraw[black] (5,4.3301) circle (2pt) node[anchor=north] {$0$};
\filldraw[black] (2.5,0) circle (2pt) node[anchor=north] {$-1$};
\filldraw[black] (7.5,0) circle (2pt) node[anchor=north] {$-1$};

\end{tikzpicture}
\begin{tikzpicture}[scale=0.4]
\draw[black,thick] (0,0) -- (10,0);
\draw[black,thick] (0,0) -- (5,8.6603);
\draw[black,thick] (5,8.6603) -- (10,0);
\draw[black,thick] (1.25,2.1651) -- (3.75,2.1651);
\draw[black,thick] (2.5,4.3301) -- (7.5,4.3301);
\draw[black,thick] (3.75,6.4952) -- (6.25,6.4952);
\draw[black,thick] (3.75,2.1651) -- (2.5,0);
\draw[black,thick] (7.5,4.3301) -- (5,0);
\draw[black,thick] (8.75,2.1651) -- (7.5,0);
\draw[black,thick] (1.25,2.1651) -- (2.5,0);
\draw[black,thick] (2.5,4.3301) -- (5,0);
\draw[black,thick] (3.75,6.4952) -- (5,4.3301);
\draw[black,thick] (6.25,6.4952) -- (5,4.3301);
\draw[black,thick] (6.25,2.1651) -- (8.75,2.1651);
\draw[black,thick] (6.25,2.1651) -- (7.5,0);

\filldraw[black] (0,0) circle (2pt) node[anchor=east] {$0$};
\filldraw[black] (10,0) circle (2pt) node[anchor=west] {$0$};
\filldraw[black] (5,8.6603) circle (2pt) node[anchor=south] {$0$};
\filldraw[black] (2.5,4.3301) circle (2pt) node[anchor=east] {$0$};
\filldraw[black] (7.5,4.3301) circle (2pt) node[anchor=west] {$0$};
\filldraw[black] (5,0) circle (2pt) node[anchor=north] {$0$};
\filldraw[black] (6.25,6.4952) circle (2pt) node[anchor=west] {$1$};
\filldraw[black] (8.75,2.1651) circle (2pt) node[anchor=west] {$-1$};
\filldraw[black] (3.75,2.1651) circle (2pt) node[anchor=west] {$0$};
\filldraw[black] (1.25,2.1651) circle (2pt) node[anchor=east] {$1$};
\filldraw[black] (6.25,2.1651) circle (2pt) node[anchor=east] {$0$};
\filldraw[black] (3.75,6.4952) circle (2pt) node[anchor=east] {$-1$};
\filldraw[black] (5,4.3301) circle (2pt) node[anchor=north] {$0$};
\filldraw[black] (2.5,0) circle (2pt) node[anchor=north] {$-1$};
\filldraw[black] (7.5,0) circle (2pt) node[anchor=north] {$1$};

\end{tikzpicture}
  \caption{Left: Eigenfunction for $b_{6}=6$ Right: Eigenfunction for $b_{7}$}
  \label{fig:eigfuncdec3}
  \end{subfigure}
  \caption{}
\end{figure}

Eigenvalues $b_{1}$ and $b_{2}$ are born on level 1 only. Both are associated with eigenfunctions according to Figure \ref{fig:eigfuncdec1}. For $b_{1}$, the value taken on at points marked $a$ is $\frac{4r}{r+\sqrt{r(8+9r)}}$, while for $b_{2}$ the value at $a$ is $\frac{4r}{r-\sqrt{r(8+9r)}}$.

Eigenvalues $b_{3},b_{4}$, and $b_{5}$ are all present with multiplicity 2 and a  basis of two skew-symmetric functions (linearly independent via rotation about the gasket) on level 1. Analytic formulas for these functions exist as functions of $r$ and can be obtained via symbolic eigenvector computations on a symbolic level 1 Laplacian matrix, and verify that such functions exist and are skew-symmetric, but are otherwise too large and unwieldy to be analyzed. If we denote these level 1 functions by $A$, then on higher levels we may create the eigenfunction see in Figure \ref{fig:eigfuncdec1}, by placing $A$ and $-A$ around an `empty' cell. We may also create, independently, the two eigenfunctions seen in Figure \ref{fig:eigfuncdec2} by gluing $A$ together along the boundary. The number of `empty' cells on level $m+1$ is given by $\sum_{i=0}^{2m-1} 3^i$, so that adding 2 gives the lower bound for the multiplicity for all three of these eigenvalues: $(\sum_{i=0}^{2m-1}3^i)+2 = \frac{3^{2m}+3}{2}$.

For $\lambda = b_6 = 9$, the multiplicity and eigenfunctions are independent of $r$. Consider Figure \ref{fig:eigfuncdec3}. We see that this is the same as the eigenfunction in the standard case described by \cite{Strichartz}. Computations show that the same construction as on the standard case, i.e. rotating and fitting this function on level $m=1$ so that $u(x) = 2$ for some $x \in V_{m+1}$ to get an eigenfunction with the same eigenvalue on level $m+1$, is successful. Since this construction is exactly the same as the construction on the standard SG shown by \cite{Strichartz}, we will take the same lower bound for the multiplicity of $\lambda = 9$ on level $m+1$, which is $\frac{3^{2m+2}-3}{2}$.

Consider $\lambda = \frac{9+6r}{1+r}$, which we will call $b_7$. On level 1, this eigenvalue corresponds to this function in Figure \ref{fig:eigfuncdec3}. This function can be miniaturized into each cell of the next level, creating $9$ new eigenfunctions. These new functions are all independent of each other, meaning that the multiplicity of $b_{7}$ on level $m+1$ is  $\geq 9^{m}$, the number of such cells in level $m+1$.

We need to confirm that the eigenvalues born on level $m+1$ are not decimated to through the eigenvalue extension mapping. We see that applying (\ref{lambdamap}) to these eigenvalues yields $\lambda_m(b_2), \lambda_m(b_3), \lambda_m(b_4), \lambda_m(b_5), \lambda_m(b_6) \leq 0$, which is impossible for Dirichlet eigenvalues. For $b_7$, we need to consider the eigenfunction extension mapping. In order for the function shown in \ref{fig:eigfuncdec3} to arise in level $m+1$, the function would have been uniformly $0$ on level $m$, which is also impossible for Dirichlet eigenfunctions.

Now we compile these results into a counting argument to show that we have acquired every eigenfunction for any level $m>1$. The sum of the multiplicities of the eigenvalues that are born on level $m>1$ is $9^{m}+3\left(\frac{3^{2m}+3}{2}\right)+\frac{3^{2m}-1}{2}$. We know that, with $\#(V_{m}\setminus V_0)$ points, level $m$ will have $\#(V_{m}\setminus V_0) = \frac{3^{2m+1}-3}{2}$ eigenvalues. Each of these will decimate to $5$ new values on level $m$, except for the $\frac{3^{2m}-3}{2}$ eigenvalues on level $m$ corresponding $\lambda = 6$, which will only decimate to $3$ new values ($2$ of the $5$ always lead to forbidden eigenvalues $b_{1}$ and $b_{2}$). This means that the number of eigenvalues we produce via decimation is $5\left(\frac{3^{2m+1}-3}{2} - \frac{3^{2m}-3}{2}\right) + 3\left(\frac{3^{2m}-3}{2}\right)$. Adding the number of eigenvalues we have identified that are born, we compute
\begin{align*}
\frac{5(3^{2m+1}-3)}{2} - \frac{2(3^{2m}-3)}{2}+9^{m}+3\left(\frac{3^{2m}+3}{2}\right)+\frac{3^{2m+2}-3}{2}& \\
=\frac{3^{2m+3}-3}{2}&
\end{align*}
which is $\#(V_{m+1}\setminus V_0)$, confirming that we have accounted for all of the eigenvalues on level $m+1$.

We now, similar to the Interval case, define

\begin{equation}
\lambda = \lim_{m \to \infty} \left(\frac{1}{\mu_{0}r_{0}}\right)^{m}\lambda_{m}
\end{equation}

with $\lambda_{m}$ a sequence defined by repeated application of the $\Phi$ mappings, with all but a finite number $\Phi_{1}$. Expressing $\Phi_{1}$ in Taylor Series form is not as simple as on the Interval; as a solution to a quintic equation we lack a closed algebraic form. However, application of the Lagrange Inversion Theorem allows simple formulation of a Taylor Series for the inverse of the aforementioned quintic near $\lambda_{m}=0$

\begin{equation}
\Phi_{1}(x) = \frac{2r(2+r)}{(1+2r)(15+26r+9r^{2})}x+O(x^{2})
\end{equation}

and so as $\lambda \to 0$, the higher order terms will fall away, causing $\lambda_{m} = O\left(\left(\frac{2r(2+r)}{(1+2r)(15+26r+9r^{2})}\right)^{m}\right)$ as $m \to \infty$. Computing the renormaliztion factor explicitly as a function of $r$ yields

\begin{equation}
\frac{1}{\mu_{0}r_{0}} = \frac{(1+2r)(15+26r+9r^{2})}{2r(2+r)}
\end{equation}

and so the limit defined above clearly exists.

The ordering of eigenvalues that are born and decimated from the $\Phi$ maps on SG is, unlike the Interval, dependent on the parameter $r$. Lacking algebraic closed forms for thte $\Phi_{i}$ is less than ideal, but we can still provide a complete description using the properties of the quintic function, $\lambda_{m+1}(\lambda_{m},r)$, that they are solutions to. For brevity we present only results and not algebraic proofs of each - these are not difficult to show for each case.

For very small $r$, the ordering follows

\begin{align*}
\Phi_1(\lambda_{1, p}^{(m)}) &, ..., \Phi_1(\lambda_{s, p}^{(m)}), b_{1}, \Phi_2(\lambda_{s, p}^{(m)}), ..., \Phi_2(\lambda_{1, p}^{(m)}), b_{4}, \\
&\Phi_3(\lambda_{1, p}^{(m)}), ..., \Phi_3(\lambda_{s, p}^{(m)}), b_{2}, \Phi_4(\lambda_{s, p}^{(m)}), ..., \Phi_4(\lambda_{1, p}^{(m)}),b_{5},\Phi_5(\lambda_{s, p}^{(m)}), ..., \Phi_5(\lambda_{1, p}^{(m)}),b_{3},b_{7},b_{6}
\end{align*}

Then at the solution to $b_{2}(r)=b_{4}(r) \approx 0.28$ a local inversion occurs around $\Phi_{4}$ to give

\begin{align*}
\Phi_1(\lambda_{1, p}^{(m)}) &, ..., \Phi_1(\lambda_{s, p}^{(m)}), b_{1}, \Phi_2(\lambda_{s, p}^{(m)}), ..., \Phi_2(\lambda_{1, p}^{(m)}), b_{4}, \\
&\Phi_3(\lambda_{1, p}^{(m)}), ..., \Phi_3(\lambda_{s, p}^{(m)}), b_{5}, \Phi_4(\lambda_{1, p}^{(m)}), ..., \Phi_4(\lambda_{s, p}^{(m)}),b_{2},\Phi_5(\lambda_{s, p}^{(m)}), ..., \Phi_5(\lambda_{1, p}^{(m)}),b_{3},b_{7},b_{6}
\end{align*}

The next major change occurs at $r=1$ when the direction of $\Phi_3$ inverts, and $b_{7}$ descends past $b_{3}$

\begin{align*}
\Phi_1(\lambda_{1, p}^{(m)}) &, ..., \Phi_1(\lambda_{s, p}^{(m)}), b_{1}, \Phi_2(\lambda_{s, p}^{(m)}), ..., \Phi_2(\lambda_{1, p}^{(m)}), b_{4}, \\
&\Phi_3(\lambda_{s, p}^{(m)}), ..., \Phi_3(\lambda_{1, p}^{(m)}), b_{5}, \Phi_4(\lambda_{1, p}^{(m)}), ..., \Phi_4(\lambda_{s, p}^{(m)}),b_{2},\Phi_5(\lambda_{s, p}^{(m)}), ..., b_7, \Phi_5(\lambda_{1,p}^{(m)}), b_3, b_6
\end{align*}

Then $b_{7}$ continues to descend past $b_{2}$ to yield

\begin{align*}
\Phi_1(\lambda_{1, p}^{(m)}) &, ..., \Phi_1(\lambda_{s, p}^{(m)}), b_{1}, \Phi_2(\lambda_{s, p}^{(m)}), ..., \Phi_2(\lambda_{1, p}^{(m)}), b_{4}, \\
&\Phi_3(\lambda_{s, p}^{(m)}), ..., \Phi_3(\lambda_{1, p}^{(m)}), b_{5}, \Phi_4(\lambda_{1, p}^{(m)}), ..., \Phi_4(\lambda_{s, p}^{(m)}),b_{7},b_{2},\Phi_5(\lambda_{s, p}^{(m)}),..., \Phi_5(\lambda_{1, p}^{(m)}),b_{3},b_{6}
\end{align*}

There are many small gaps of varying sizes interweaved between the above decimation ordering, but their exact size and location is very difficult to describe analytically. We summarize the above results in another conclusive theorem.

\begin{thm}[SG Spectral Decimation]

For any $r$, given $u_{m}$, an eigenfunction with eigenvalue $\lambda_{m}$ on $V_{m}$, we may choose $\lambda_{m+1}$ as a solution to \ref{lambdamap}, given that $\lambda_{m+1} \neq b_{1},b_{2},b_{3},b_{4},b_{5},b_{6},b_{7}$. We can then extend $u_{m}$ to $V_{m+1}$ according to \ref{ext1},\ref{ext2},\ref{ext3} to obtain an eigenfunction on level $m+1$. Then using counting arguments and constructions detailed above, this process, taken together with known eigenfunctions born on each level, produces a complete spectrum on level $m+1$.

\end{thm}

\section{Data on the Sierpinski Gasket}
In this section, we will present the experimental data produced for SG for $r=0.5,3$.

\subsection{Eigenvalues and Eigenfunctions}
Table \ref{sgeigvals} shows a portion of the spectra for $r=0.5, 1$, and $r=3$ on SG for the first three levels. There doesn't seem to be an obvious pattern to the spectra for certain parameter values of $r$ than are observed in the Interval case. Since $\#(V_m \setminus V_0) = \frac{3^{2m+1}-3}{2}$, we have $\frac{3^{2m+1}}{2}$ eigenvalues at level $m$.
\newpage
\begin{center}
\tiny
    \begin{tabular}{l|l|l|l}%
    $n$ & $m = 1$ & $m = 2$ & $m = 3$ 
    \begin{filecontents*}{r05eigvals.csv}
,,,
1,24.20000,25.05695,25.09281
2,55.36441,60.24744,60.45549
3,55.36441,60.24744,60.45549
4,133.1000,170.4184,172.1013
5,133.1000,170.4184,172.1013
6,145.2000,192.2647,194.4113
7,174.5356,255.3789,259.1902
8,174.5356,255.3789,259.1902
9,193.6000,309.8982,315.5415
10,217.8000,433.1401,444.3046
11,217.8000,433.1401,444.3046
12,217.8000,433.1401,444.3046
13,,735.8874,769.1626
14,,812.6530,853.5739
15,,812.6530,853.5739
16,,911.9137,964.0106
17,,949.4321,1006.141
18,,949.4321,1006.141
19,,1174.440,1263.477
20,,1174.440,1263.477
21,,1265.407,1369.883
22,,1339.819,1457.988
23,,1339.819,1457.988
24,,1339.819,1457.988
25,,1339.819,1457.988
26,,1339.819,1457.988
27,,1339.819,1457.988
28,,2392.064,2825.629
29,,2446.852,2904.264
30,,2446.852,2904.264
31,,2553.162,3059.388
32,,2553.162,3059.388
33,,2566.728,3079.432
34,,2596.932,3124.265
35,,2596.932,3124.265
36,,2614.705,3150.783
37,,2635.380,3181.757
38,,2635.380,3181.757
39,,2635.380,3181.757
40,,3221.020,4124.126
41,,3221.020,4124.126
42,,3221.020,4124.126
43,,3221.020,4124.126
44,,3221.020,4124.126
45,,3221.020,4124.126
46,,3241.324,4159.370
47,,3270.497,4210.355
48,,3270.497,4210.355
49,,3361.988,4372.997
50,,3361.988,4372.997
51,,3379.323,4404.298
52,,3426.166,4489.691
53,,3426.166,4489.691
54,,3461.300,4554.529
55,,3959.160,5559.953
56,,3959.160,5559.953
57,,3959.160,5559.953
58,,4005.369,5663.113
59,,4036.030,5732.676
60,,4036.030,5732.676
61,,4076.930,5826.922
62,,4092.159,5862.450
63,,4092.159,5862.450
64,,4175.123,6060.421
65,,4175.123,6060.421
66,,4203.307,6129.463
67,,4223.761,6180.169
68,,4223.761,6180.169
69,,4223.761,6180.169
70,,4223.761,6180.169
71,,4223.761,6180.169
72,,4223.761,6180.169
73,,4685.120,7499.536
,,...,...
\end{filecontents*}
\csvreader[head to column names]{r05eigvals.csv}{}
    {\\\hline\csvcoli&\csvcolii&\csvcoliii&\csvcoliv}
    \end{tabular}
    \quad
    \begin{tabular}{l|l|l|l}%
    $n$ & $m = 1$ & $m = 2$ & $m = 3$ 
    \begin{filecontents*}{r3eigvals.csv}
,,,
1,11.46517,11.53001,11.53161
2,72.55942,75.30424,75.37256
3,72.55942,75.30424,75.37256
4,203.1230,227.7845,228.4117
5,203.1230,227.7845,228.4117
6,274.0500,323.1067,324.3712
7,323.4848,396.8281,398.7384
8,333.3176,412.3704,414.4339
9,333.3176,412.3704,414.4339
10,365.4000,2271.921,2337.182
11,365.4000,2271.921,2337.182
12,365.4000,2271.921,2337.182
13,,2343.077,2412.602
14,,2343.077,2412.602
15,,2364.160,2434.976
16,,2465.819,2543.036
17,,2601.450,2687.663
18,,2601.450,2687.663
19,,2828.745,2931.217
20,,2828.745,2931.217
21,,2927.756,3037.780
22,,2945.912,3057.352
23,,2945.912,3057.352
24,,2945.912,3057.352
25,,2945.912,3057.352
26,,2945.912,3057.352
27,,2945.912,3057.352
28,,7417.620,8208.932
29,,7417.620,8208.932
30,,7417.620,8208.932
31,,7452.209,8251.684
32,,7452.209,8251.684
33,,7463.317,8265.424
34,,7523.311,8339.725
35,,7624.700,8465.641
36,,7624.700,8465.641
37,,7900.887,8810.892
38,,7900.887,8810.892
39,,8155.336,9131.941
40,,8246.795,9248.050
41,,8246.795,9248.050
42,,8246.795,9248.050
43,,8246.795,9248.050
44,,8246.795,9248.050
45,,8246.795,9248.050
46,,8762.523,9909.999
47,,9110.091,10363.26
48,,9110.091,10363.26
49,,9595.838,11006.79
50,,9595.838,11006.79
51,,9818.552,11305.92
52,,9967.039,11506.82
53,,9996.189,11546.40
54,,9996.189,11546.40
55,,10090.78,11675.16
56,,10090.78,11675.16
57,,10090.78,11675.16
58,,11126.43,13118.13
59,,11126.43,13118.13
60,,11126.43,13118.13
61,,11126.43,13118.13
62,,11126.43,13118.13
63,,11126.43,13118.13
64,,11126.43,13118.13
65,,11126.43,13118.13
66,,11126.43,13118.13
67,,13175.44,16176.94
68,,13175.44,16176.94
69,,13187.95,16196.55
70,,13248.50,16291.69
71,,13329.52,16419.46
72,,13329.52,16419.46
73,,13464.26,16633.14
,,...,...
\end{filecontents*}
\csvreader[head to column names]{r3eigvals.csv}{}
    {\\\hline\csvcoli&\csvcolii&\csvcoliii&\csvcoliv}
    \end{tabular}
    \captionof{table}{Eigenvalues on SG (truncated for length). Left: $r=0.5$ Right: $r=3$}
    \label{sgeigvals}
\end{center}

Figure \ref{fig:sgeigfns} show the first 4 eigenfunctions for each of the $r$ values. 

\begin{figure}
\centering
\includegraphics[width=0.65\textwidth]{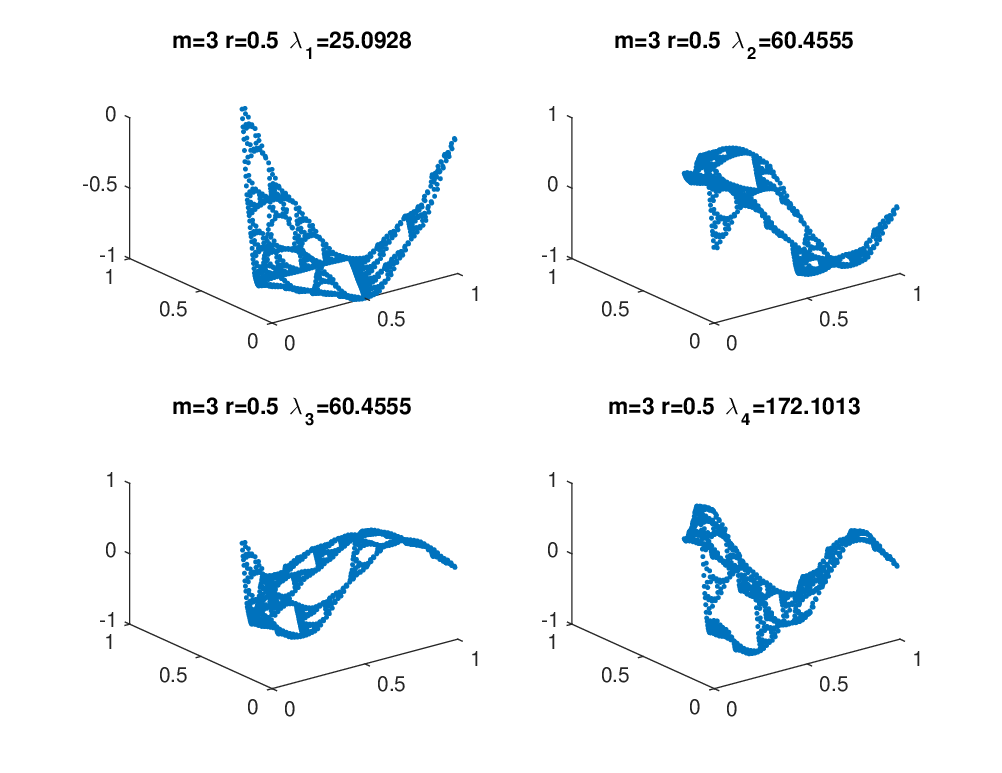}
\includegraphics[width=0.65\textwidth]{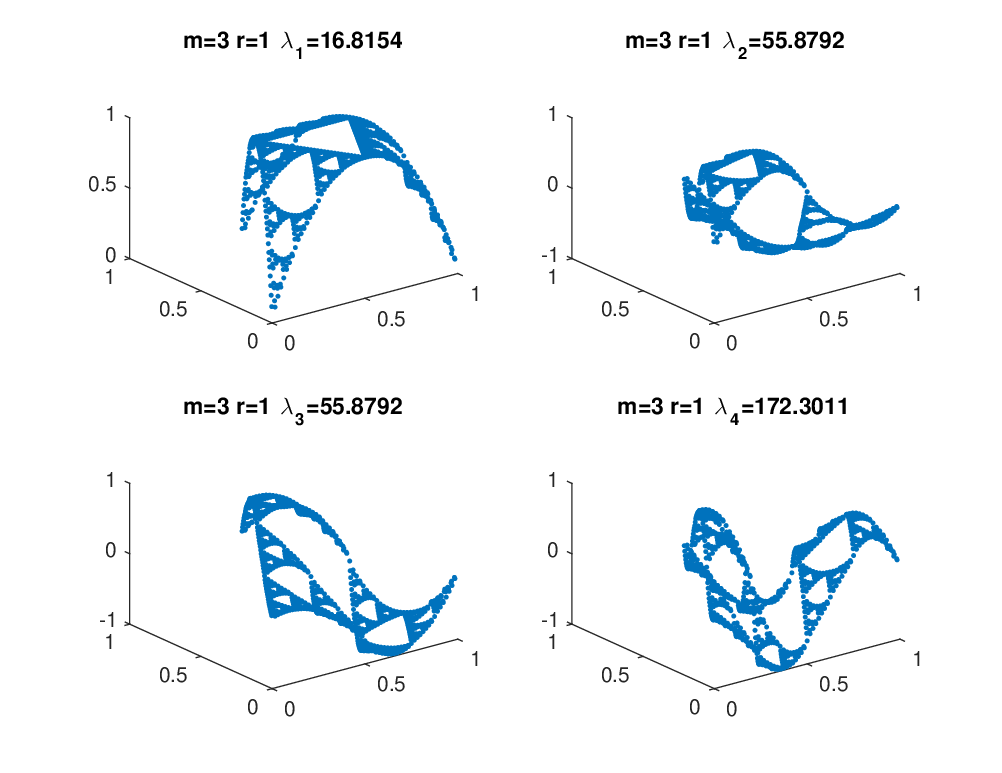}
\includegraphics[width=0.65\textwidth]{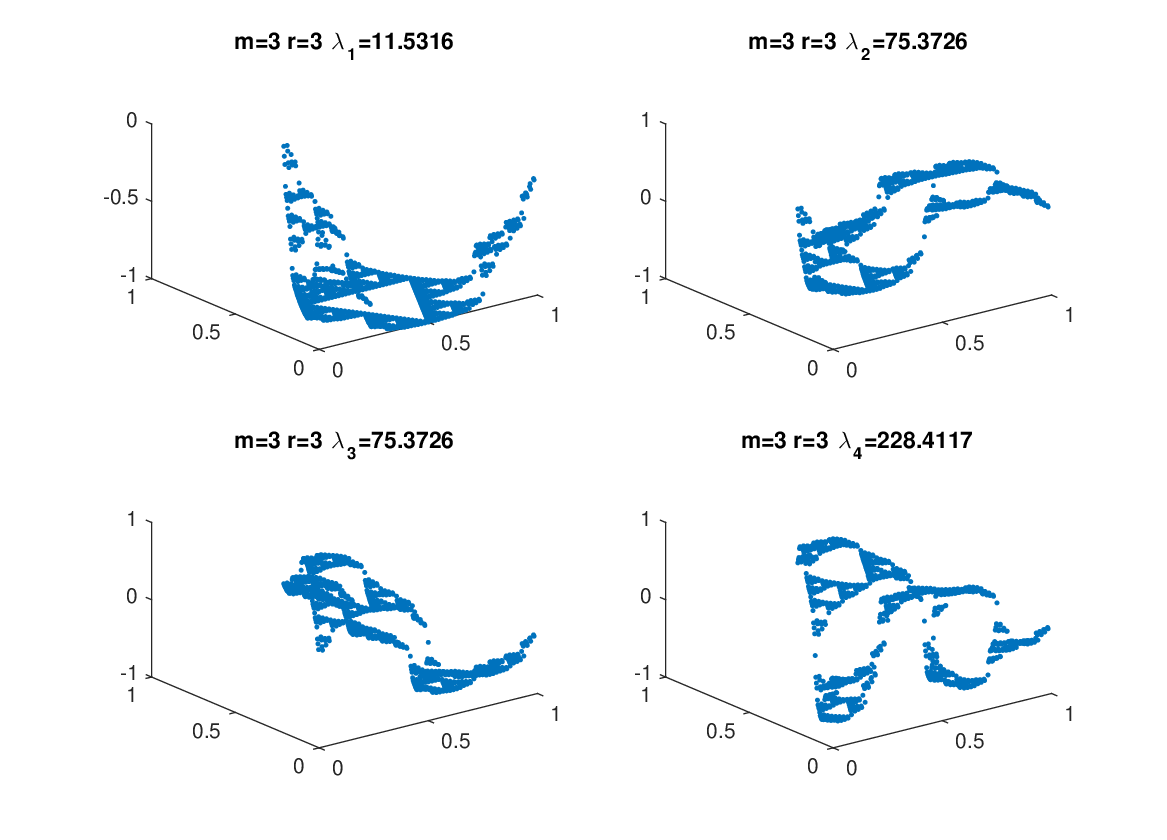}
\caption{First 4 eigenfunctions on $m=3$. Top: $r=0.5$, Middle: $r=1$ (standard), Bottom: $r=3$}
\label{fig:sgeigfns}
\end{figure}

\subsection{Eigenvalue Counting Functions and Weyl Plots}

\begin{figure}
\centering
\includegraphics[width=0.3\textwidth]{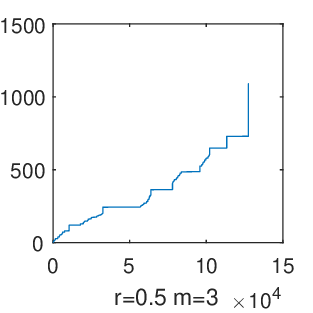}
\includegraphics[width=0.3\textwidth]{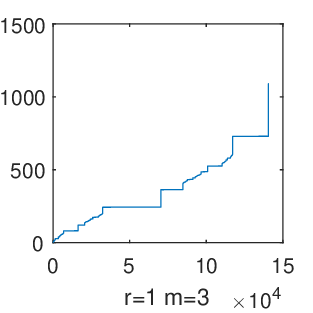}
\includegraphics[width=0.3\textwidth]{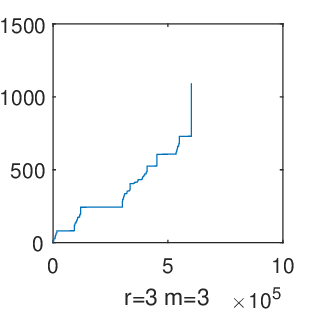}
\caption{Graphs of counting functions on $m=3$. Left: $r=0.5$, Center: $r=1$ (standard), Right: $r=3$}
\label{fig:sgcounting}
\end{figure}

We give graphs of counting functions for the SG for $r=0.5, 1,$ and $3$ at the first three levels in Figure \ref{fig:sgcounting}. The counting function is defined 
\begin{equation}
N(x) = \#\{\lambda | \lambda\leq x, \text{ for all eigenvalues }\lambda\}
\end{equation}
as before, and is known to follow a power law according to Weyl asymptotics. 

\begin{figure}
\centering
\includegraphics[width=0.3\textwidth]{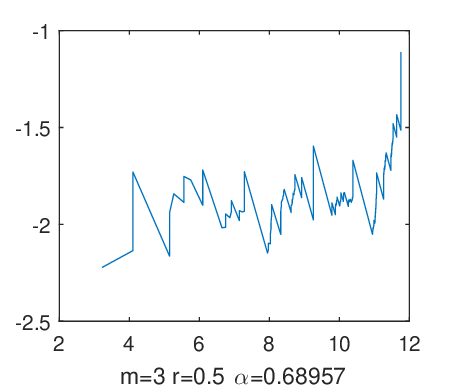}
\includegraphics[width=0.3\textwidth]{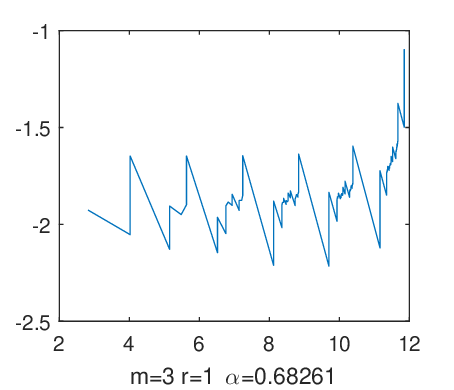}
\includegraphics[width=0.3\textwidth]{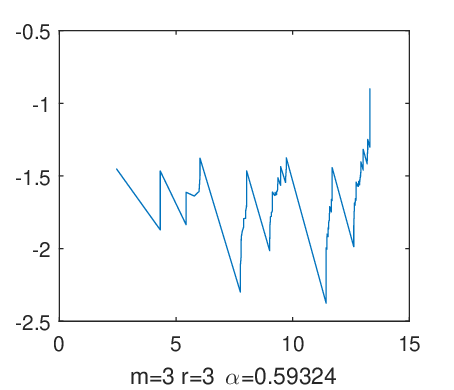}
\caption{Weyl plots on $m=3$. Left: $r=0.5$, Center: $r=1$ (standard), Right: $r=3$}
\label{fig:sgweyl}
\end{figure}

We can also generate Weyl plots similarly as on the Interval. $\alpha$, the exponent in the power law relationship, can be computed in the same way as before. We can just replace the renormalization factor and the number of cells in the formula for $\alpha$ for the Interval to obtain the correct constant for the gasket. This gives
\begin{equation}
\alpha = \frac{\log(9)}{\log(L(r)^{-1})}
\end{equation}
where $L(r)$ is the renormalization factor as defined in section 6. Equipped with $\alpha$, we plot 
\begin{equation}
W(\lambda) = \frac{N(\lambda)}{\lambda^\alpha}
\end{equation}
on a log-log scale and confirm that periodicity occurs in our Weyl plots in Figure \ref{fig:sgweyl}. This behavior is established in \cite{Kigami3}.

\subsection{Limiting Laplacians}

We would now like to, similar to the previous analysis on the Interval, examine the limiting behavior of our Laplacian on SG, as either $r \to 0$ or $r \to \infty$. As before, the eigenfunctions can be prohibitively complicated, but the eigenvalues more readily offer themselves up for analysis.

Again, for both limiting cases, the renormalization constant $\frac{1}{\mu_{0}r_{0}}$ is unbounded, and so any eigenvalues that we wish to remain bounded (with respect to $r$) must be very small. Considering the set $\{b_{1},b_{2},b_{3},b_{4},b_{5},b_{6},b_{7}\}$ of eigenvalues born on each ($m=1$ for $b_{1},b_{2}$) level, the only eigenvalue that tends towards $0$ in the limiting cases is $b_{1}=\frac{3(2+3r-\sqrt{8r+9r^{2}})}{2(1+r)}$. In the case of the first eigenvalue,

\begin{equation}\lim_{r \to 0,m \to \infty} \left(\frac{1}{\mu_{0}r_{0}}\right)^{m}\Phi^{m}_{1}\left(\frac{3(2+3r-\sqrt{8r+9r^{2}})}{2(1+r)}\right)= \infty
\end{equation}

\begin{equation}
\lim_{r \to \infty,m \to \infty} \left(\frac{1}{\mu_{0}r_{0}}\right)^{m}\Phi^{m}_{1}\left(\frac{3(2+3r-\sqrt{8r+9r^{2}})}{2(1+r)}\right)= 9
\end{equation}

the behavior is different for the two limiting cases. This is verified experimentally in Table \ref{tab:sglim}.

\begin{center}
\tiny
    \begin{tabular}{c|c|c|c}%
    $ $ & $r=10^{-2}$ & $r=10^{-4}$ & $r=10^{-5}$\\
    \hline
    \hline
    $\lambda_{1}$ & $1.0096*10^{3}$ & $1.10958*10^{5}$ & $1.12002*10^{6}$  \\
    \hline
    $\lambda_{2}$ & $1.15446*10^{3}$ & $1.12529*10^{5}$ & $1.12503*10^{6}$ \\
    \hline
    $\lambda_{3}$ & $1.15446*10^{3}$ & $1.12529*10^{5}$ & $1.12503*10^{6}$  \\
    \hline
    $\lambda_{4}$ & $1.34118*10^{3}$ & $1.14141*10^{5}$ & $1.13008*10^{6}$   \\
    \hline
    \end{tabular}
    \quad
    \begin{tabular}{c|c|c|c}%
    $ $ & $r=10^{2}$ & $r=10^{4}$ & $r=10^{5}$\\
    \hline
    \hline
    $\lambda_{1}$ & $9.0750$ & $9.0008$ & $8.9994$  \\
    \hline
    $\lambda_{2}$ & $1381.52$ & $1.35031*10^{5}$ & $1.35003*10^{6}$ \\
    \hline
    $\lambda_{3}$ & $1381.52$ & $1.35031*10^{5}$ & $1.35003*10^{6}$  \\
    \hline
    $\lambda_{4}$ & $4141.61$ & $4.05091*10^{5}$ & $4.05009*10^{6}$   \\
    \hline
    \end{tabular}
    \captionof{table}{Limiting eigenvalues}
    \label{tab:sglim}
\end{center}

We will use ratios to further characterize these eigenvalues in the upcoming section. Similarly to the Interval case, the eigenfunction extension algorithm is related to the sequence of $\Phi$ maps used on each individual eigenvalue. However, we can examine the ground state eigenfunction to provide an interesting example. We can explicitly take the limits of the eigenvalue extension formulas as $r \to \infty$. Writing the extensions to $w_{0},z_{0},$ and $y_{0,1}$ as functions of values $x_{0},x_{1},x_{2},r,\lambda$,

\begin{align}
    \lim_{r\to \infty} w_{0}(x_{0},x_{1},x_{2},r,\Phi_{1}(b_{1})) &= \frac{x_{0}+x_{1}+x_{2}}{3} \nonumber\\
    \lim_{r\to \infty} z_{0}(x_{0},x_{1},x_{2},r,\Phi_{1}(b_{1})) &= \frac{x_{0}+x_{1}+x_{2}}{3} \nonumber\\
    \lim_{r\to \infty} y_{0,1}(x_{0},x_{1},x_{2},r,\Phi_{1}(b_{1})) &= \frac{x_{0}+x_{1}+x_{2}}{3}
\end{align}

Here we use $\Phi_{1}(b_{1})$ as the eigenvalue because this will lead us to $\lambda = 9$. While we do not have an a closed algebraic form for $\Phi_{1}$, since $b_{1}\to 0$, it suffices to use the Lagrangian Inversion Polynomial discussed eariler as a substitute for an analytical form of $\Phi_{1}$. It is interesting to note that this limiting direction is analogous to the limiting direction producing $\lambda_{1} = 4$ on the Interval. If we consider the outer 3 cells of the twice-iterated gasket as the ``outside", then both $p \to 0$ and $r \to \infty$ assign all measure to the ``inside" core and all resistance to the ``outside" shell of the structure, be it Interval or SG. This is in contrast to the case where all measure is assigned to the ``outside" and resistance to the ``inside" - on the Interval the limiting structure is simply the $\frac{1}{2}$-Cantor set.

Similar to the Interval example, the ground state eigenfunction on level 1 has value of uniformly 1 in the limiting case. Then application of the above extension algorithm yields a Cantor-like function on SG. On the inside 6 level 1 cells the function is uniformly 1, but on the outside 3 cells a step-function is formed, similar to that formed by the classic Cantor function. The most significant difference is the dependence on the values of three points - thus the largest `tier' occurs at $\frac{1}{3}(0+1+1)=\frac{2}{3}$, as opposed to $\frac{1}{2}$ on the Interval. Figure \ref{fig:cantorsg} contains images of this function, an interesting extension of the Devil's Staircase to SG.

\begin{figure}
\centering
  \includegraphics[width=0.4\textwidth]{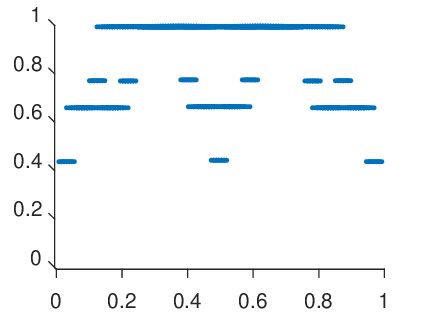}
\includegraphics[width=0.4\textwidth]{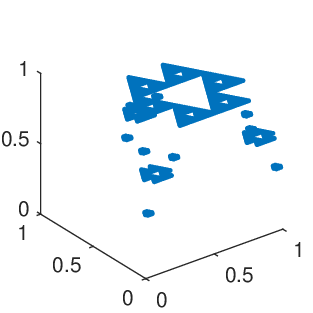}
  \caption{Ground state eigenfunction for $r=10^{4}$, $m=4$, with side-view (left) to show similarity to Cantor function}
  \label{fig:cantorsg}
\end{figure}

\subsection{Ratios of Eigenvalues}

Ratios also provide an effective method of eigenvalue analysis on SG, and the limiting case is more interesting than on the Interval. Some gaps have been shown to exist even in the standard case \cite{Bockelman}. In our case, $\{b_{1},b_{2},b_{3},b_{4},b_{5},b_{6},b_{7}\}$, the eigenvalues being born on each level, have explicit limits for extreme r, given by

\begin{align}
    \lim_{r\to 0} \{b_{1},b_{2},b_{3},b_{4},b_{5},b_{6},b_{7}\} &= \{3,3,\frac{15}{2},3,\frac{9}{2},9,9\} \nonumber\\
    \lim_{r\to \infty} \{b_{1},b_{2},b_{3},b_{4},b_{5},b_{6},b_{7}\} &= \{0,9,9,\frac{3}{2},\frac{9}{2},6,9\}
\end{align}

We can also, as before, examine the limiting behavior of $\{\Phi_{1},\Phi_{2},\Phi_{3},\Phi_{4},\Phi_{5}\}$

\begin{align}
    \lim_{r\to 0} \{\Phi_{1},\Phi_{2},\Phi_{3},\Phi_{4},\Phi_{5}\} &= \{0,3,3,\frac{3}{2},\frac{15}{2}\} \nonumber\\
    \lim_{r\to \infty} \{\Phi_{1},\Phi_{2},\Phi_{3},\Phi_{4},\Phi_{5}\} &= \{0,\frac{3}{2},\frac{9}{2},6,9\}
\end{align}

Of significant interest is that, unlike on the Interval, the limiting behavior of the eigenvalues which are born and the eigenvalue extension maps is different for the two limiting directions of the parameter. We would expect this difference to show up in the observed ratios between eigenvalues. For $r \to 0$, we expect to observe ratios of the set $\{3,3,\frac{3}{2},\frac{15}{2}\}$, or explicitly $\{\frac{1}{3},\frac{2}{5},\frac{1}{2},\frac{3}{5},\frac{2}{3},\frac{5}{6},1,\frac{6}{5},\frac{3}{2},\frac{5}{3},2,\frac{5}{2},3\}$. For $r \to \infty$, we expect to observe ratios of the set $\{\frac{3}{2},\frac{9}{2},6,9\}$, or explicitly $\{\frac{1}{6},\frac{1}{4},\frac{1}{3},\frac{1}{2},\frac{2}{3},\frac{3}{4},1,\frac{4}{3},\frac{3}{2},2,3,5,6\}$. In fact these are precisely the ratios we observe numerically for both limiting cases, as seen in Figure \ref{fig:sgratio}. Of course, the gaps in the ratios of eigenvalues imply gaps of the form $\limsup_{n\to\infty} \frac{\lambda_{n+1}}{\lambda_n} > 1$.

\begin{figure}
\centering
    \includegraphics[width=0.3\textwidth]{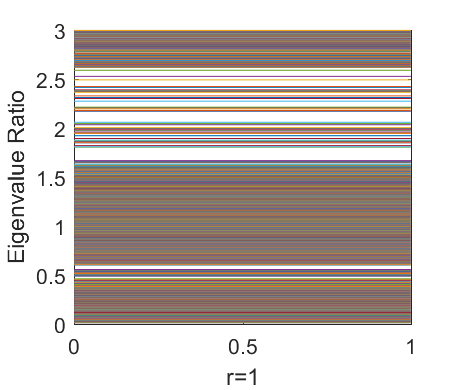}
    \includegraphics[width=0.3\textwidth]{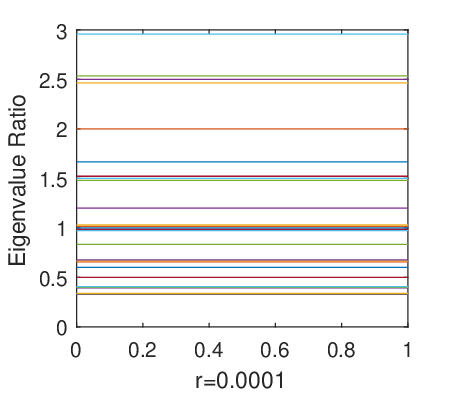}
    \includegraphics[width=0.3\textwidth]{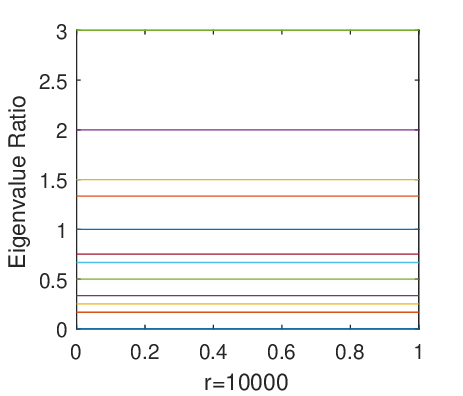}
  \caption{Eigenvalue ratios for $r=1$, $r=10^{-4}$, $r=10^{4}$}
  \label{fig:sgratio}
\end{figure}

\section{Threshold Subdivision}
In the next two sections, we will generalize the Laplacians that we have been studying in this paper to create different families of the Laplacian, though this will eliminate self-similarity at any level $m$. One way is to change the division scheme of the cells when extending to level $m$ from level $m-1$. Until now, the cells on level $m$ could all be written in form $F_w(K)$ where $|w| = m$ and $K = SG$ or $I$. For a threshold subdivision, choose a cutoff value $c$. Then given a partition into cells $C_{m} = \{A_1, ..., A_N\}$ on level $m$, we take
\begin{align}
A_n &\in C_{m+1} &\text{if } \mu(A_n) < c^{m+1} \\
F_i(A_n) &\in C_{m+1} \ \forall \ i & \text{otherwise}
\end{align}
i.e. divide the cell at level $m$ if its measure is greater than $c^{m+1}$. For the right choice of measure and cutoff value, this will give us a more uniform distribution of measure throughout $K$.

\begin{figure}
\centering
\includegraphics[width=\textwidth]{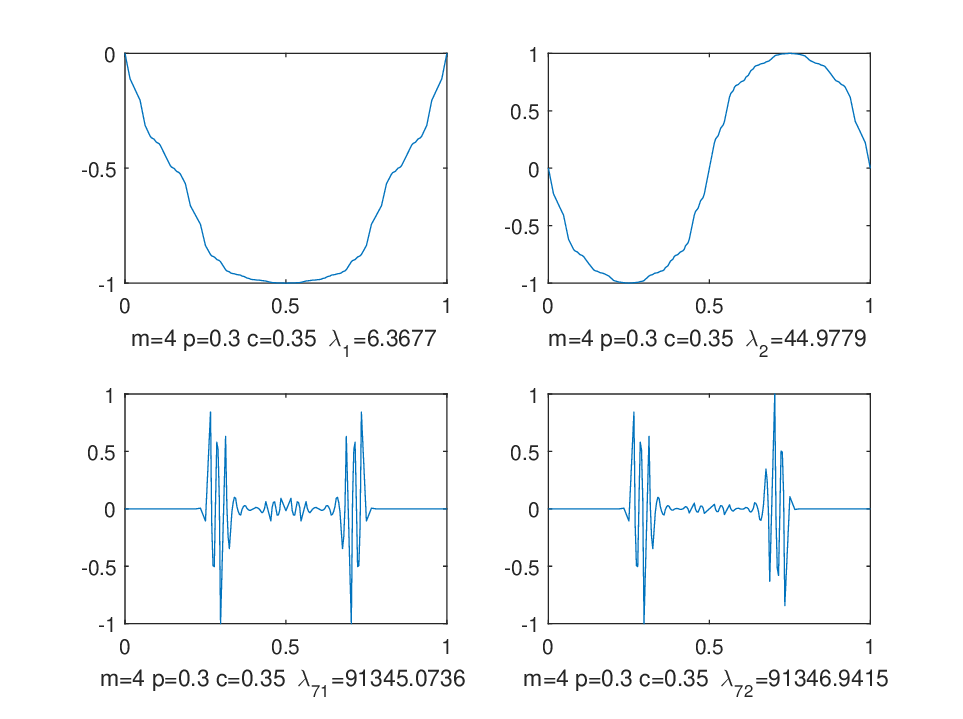}
\includegraphics[width=\textwidth]{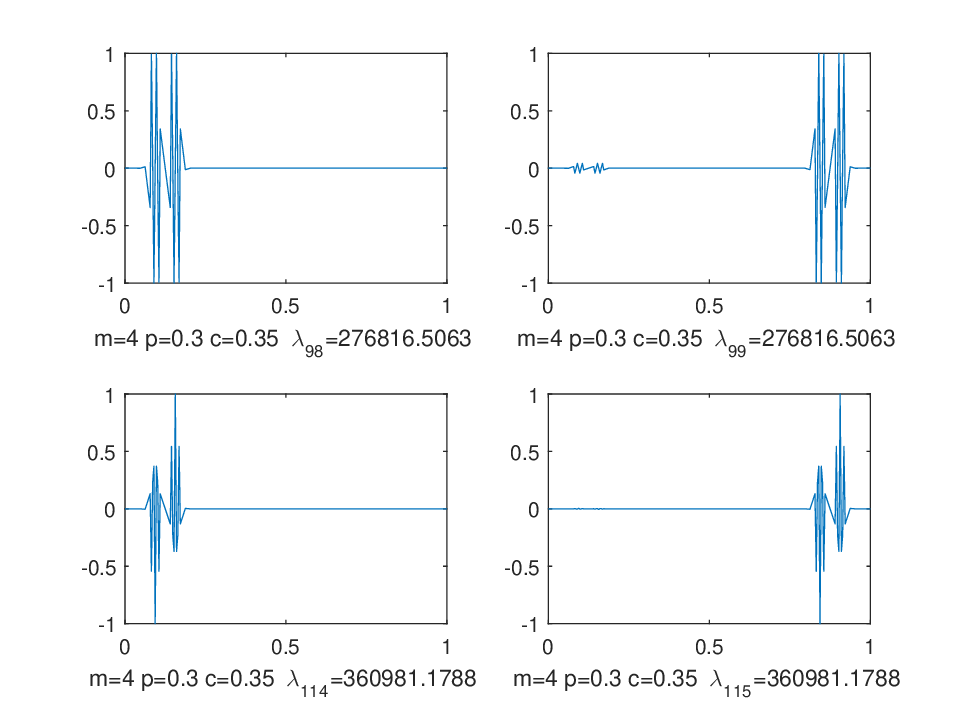}
\caption{Select eigenfunctions for threshold division, $m=4$, $p=0.3$, $c=0.35$}
\label{fig:threshfns}
\end{figure}

\begin{center}
\begin{table}\label{threshspec}
\tiny
\begin{minipage}[t]{0.33\textwidth}
\vspace{0pt}
    \begin{tabular}{l|l|l|l}%
    \bfseries $n$ & \bfseries $m = 1$ & \bfseries $m = 2$ & \bfseries $m = 3$ 
    \begin{filecontents*}{p3c0.csv}
,,,
1,6.222475,6.360901,6.368233
2,38.09524,44.62337,44.98851
3,69.96800,101.8143,103.7496
4,,118.5233,121.1600
5,,354.9600,380.5864
6,,475.0876,523.0972
7,,629.7550,719.5578
8,,725.6236,849.9690
9,,821.4921,988.2072
10,,976.1596,1231.269
11,,1096.287,1441.680
12,,1332.724,1939.319
13,,1349.433,1980.481
14,,1406.624,2129.940
15,,1444.886,2238.704
16,,,2257.587
17,,,6280.032
18,,,6449.775
19,,,6692.325
20,,,6761.142
21,,,7675.375
22,,,8118.797
23,,,8689.125
24,,,9049.287
25,,,9419.944
26,,,10054.94
27,,,10600.28
28,,,11995.33
29,,,12131.09
30,,,12707.14
31,,,13401.88
32,,,13821.40
33,,,14240.92
34,,,14935.67
35,,,15511.72
36,,,15647.47
37,,,17042.52
38,,,17587.86
39,,,18222.86
40,,,18593.52
41,,,18953.68
42,,,19524.01
43,,,19967.43
44,,,20881.66
45,,,20950.48
46,,,21193.03
47,,,21362.77
48,,,25385.22
49,,,25404.10
50,,,25512.86
51,,,25662.32
52,,,25703.48
53,,,26201.12
54,,,26411.53
55,,,26654.60
56,,,26792.83
57,,,26923.25
58,,,27119.71
59,,,27262.22
60,,,27521.64
61,,,27539.05
62,,,27597.81
63,,,27636.43
\end{filecontents*}
\csvreader[head to column names]{p3c0.csv}{}
    {\\\hline\csvcoli&\csvcolii&\csvcoliii&\csvcoliv}
    \end{tabular}
\end{minipage}
\begin{minipage}[t]{0.33\textwidth}
\vspace{0pt}
    \begin{tabular}{l|l|l|l}%
    \bfseries $n$ & \bfseries $m = 1$ & \bfseries $m = 2$ & \bfseries $m = 3$ 
    \begin{filecontents*}{p3c35.csv}
,,,
1,6.222475,6.360901,6.354989
2,38.09524,44.62337,44.79475
3,69.96800,101.8143,100.1729
4,,118.5233,116.5333
5,,354.9600,373.9338
6,,475.0876,521.7032
7,,629.7550,740.8326
8,,725.6236,822.3015
9,,821.4921,909.5239
10,,976.1596,1194.766
11,,1096.287,1453.684
12,,1332.724,2120.754
13,,1349.433,3505.048
14,,1406.624,3505.582
15,,1444.886,4265.552
16,,,4269.176
17,,,4781.933
18,,,4783.277
19,,,6525.961
20,,,6866.018
21,,,8212.002
22,,,9174.324
23,,,9881.015
24,,,9882.112
25,,,10153.79
26,,,12175.47
27,,,12983.19
28,,,14074.93
29,,,14532.81
30,,,14533.42
31,,,15208.65
32,,,15822.08
33,,,17646.69
34,,,18651.30
35,,,18951.58
36,,,18951.63
37,,,19557.36
38,,,20905.12
39,,,21208.02
40,,,25434.87
41,,,25565.03
42,,,25565.13
43,,,25634.94
44,,,25786.45
45,,,26432.80
46,,,26809.97
47,,,26896.26
48,,,26896.26
49,,,27127.68
50,,,27525.33
51,,,27599.89
\end{filecontents*}
\csvreader[head to column names]{p3c35.csv}{}
    {\\\hline\csvcoli&\csvcolii&\csvcoliii&\csvcoliv}
    \end{tabular}
    \end{minipage}
\begin{minipage}[t]{0.33\textwidth}
\vspace{0pt}
    \begin{tabular}{l|l|l|l}%
    \bfseries $n$ & \bfseries $m = 1$ & \bfseries $m = 2$ & \bfseries $m = 3$ 
    \begin{filecontents*}{p3c5.csv}
,,,
1,6.222475,6.023543,6.360901
2,38.09524,42.94625,44.62337
3,69.96800,187.1441,101.8143
4,,245.1589,118.5233
5,,518.6735,354.9600
6,,762.9425,475.0876
7,,994.9449,629.7550
8,,1342.166,725.6236
9,,1412.052,821.4921
10,,,976.1596
11,,,1096.287
12,,,1332.724
13,,,1349.433
14,,,1406.624
15,,,1444.886
\end{filecontents*}
\csvreader[head to column names]{p3c5.csv}{}
    {\\\hline\csvcoli&\csvcolii&\csvcoliii&\csvcoliv}
    \end{tabular}
\end{minipage}
    \centering
    \captionof{table}{Eigenvalues for threshold division, $p=0.3$. Left: $c=0.0$ (standard), Center: $c=0.35$, Right: $c=0.5$}
    \label{threshtbl}
\end{table}
\end{center}

\begin{figure}
\centering
\includegraphics[width=0.32\textwidth]{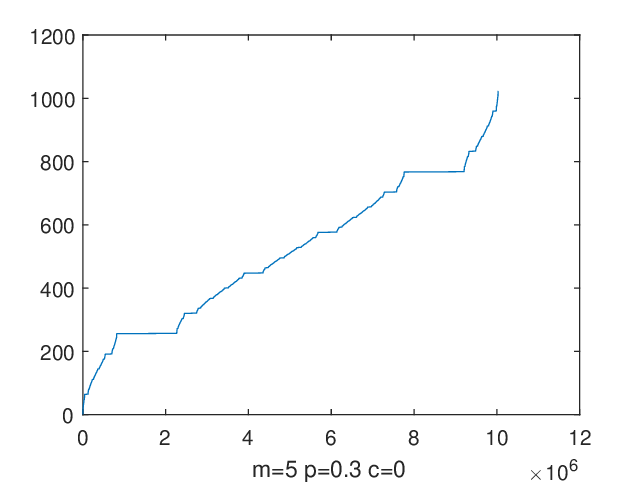}
\includegraphics[width=0.32\textwidth]{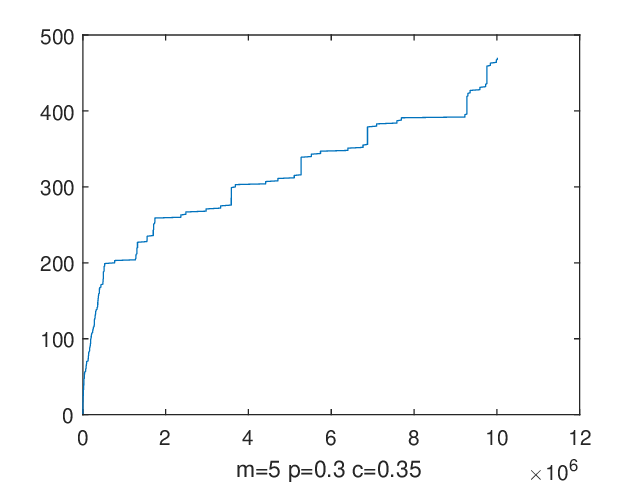}
\includegraphics[width=0.32\textwidth]{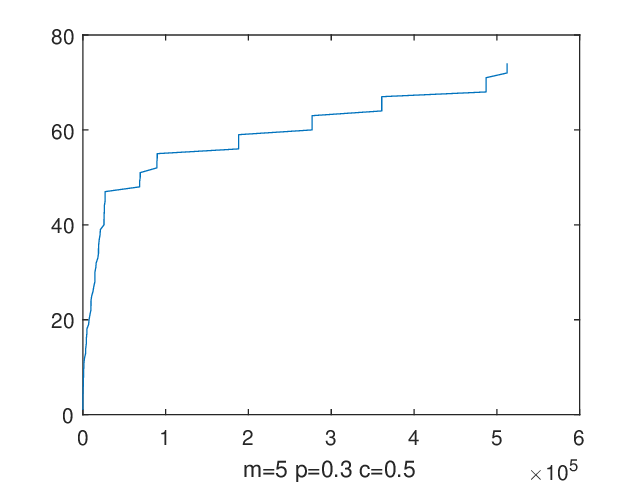}
\caption{Eigenvalue counting functions for threshold division, $m=5$, $p=0.3$. Left: $c=0$ (standard), Center: $c=0.35$, Right: $c=0.5$}
\label{fig:threshcounting}
\end{figure}

We computed the spectra and eigenfunctions on the Interval numerically for various values for $p$ and cutoff values $c$. For eigenvalues high enough on the spectrum, we observed eigenfunctions that are asymmetric about $x=\frac{1}{2}$, which is unprecedented in our study of self-similar, symmetric Laplacians. Some examples are given in Figure \ref{fig:threshfns} for $p=0.3$ and $c=0.35=\frac{1-p}{2}$. As the figure suggests, the lower portion of the spectra seems to give eigenfunctions that are pointwise close to the eigenfunctions obtained for the self-similar Laplacians from previous sections. As we progress to higher portions of the spectra, we begin to see more and more eigenfunctions that are asymmetric about $x=\frac{1}{2}$, as well as eigenvalues that seem to hint at multiplicities $\geq 1$. However, we have yet to see eigenvalues that have multiplicities $>2$. We suspect that this implies that these discrete approximations yield a different Laplacian than those in previous sections, as the discrete eigenfunctions observed seem to approximate a different set of eigenfunctions as in the standard division scheme. We give the full spectra at the first three levels and the graphs of the corresponding eigenvalue counting functions for $p=0.3$ and more values of $c$ in Table \ref{threshtbl} and Figure \ref{fig:threshcounting} respectively. Though the data in this paper only corresponds to the Interval, the same scheme can be used to generate a family of Laplacians on SG as well.

\section{Hierarchical Laplacians}
Another way to construct the Laplacian on $K=I$ or $SG$ is to use a sequence of parameters instead of a single value for $p$ or $r$ such that the measure and resistance of each $m$-cell will be determined according to the $m$th parameter in the sequence. For example, if $A = F_w(I)$ is an $m$-cell on the Interval with the sequence of parameters $\{p_i\}$, we can determine measure and resistance on $F_i(A)$, an $(m+1)$-cell, in the following way. If $i=0, 3$,
\begin{align}
\mu(F_i(A)) &= \left(\frac{p_{m+1}}{2}\right)\mu(A) \\
R(F_i(A)) &= \left(\frac{1-p_{m+1}}{2}\right)R(A)
\end{align}

If $i=1, 2$,
\begin{align}
\mu(F_i(A)) &= \left(\frac{1-p_{m+1}}{2}\right)\mu(A) \\
R(F_i(A)) &= \left(\frac{p_{m+1}}{2}\right)R(A)
\end{align}

The same construction can be used to construct hierarchical Laplacians using a sequence of $r$ values on SG as well. Note that decimation still holds for this hierarchical Laplacian, though a different extension mapping will be used at every level $m$.

\begin{figure}
\centering
\includegraphics[width=0.49\textwidth]{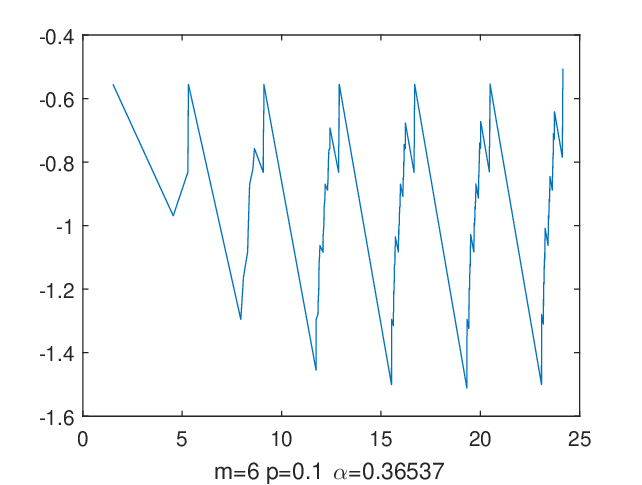}
\includegraphics[width=0.49\textwidth]{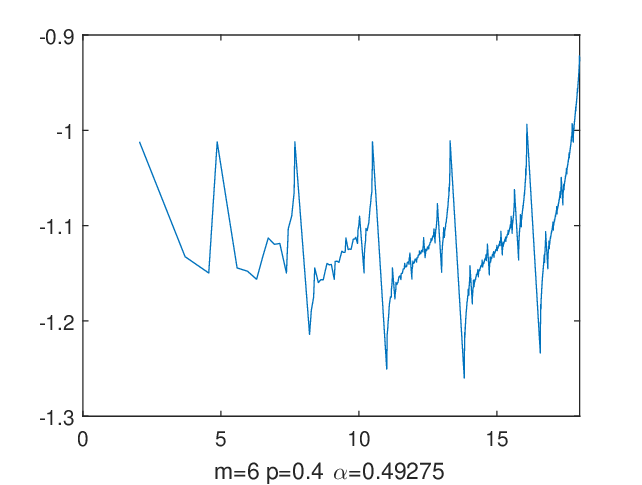}
\caption{Weyl plots for standard division scheme, $p=0.1$ (left) and $p=0.4$ (right)}
\label{fig:stdweyl}
\end{figure}

\begin{figure}
\centering
\includegraphics[width=0.49\textwidth]{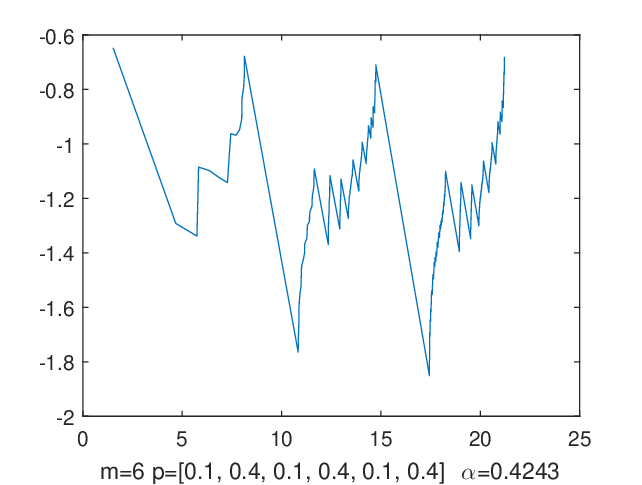}
\includegraphics[width=0.49\textwidth]{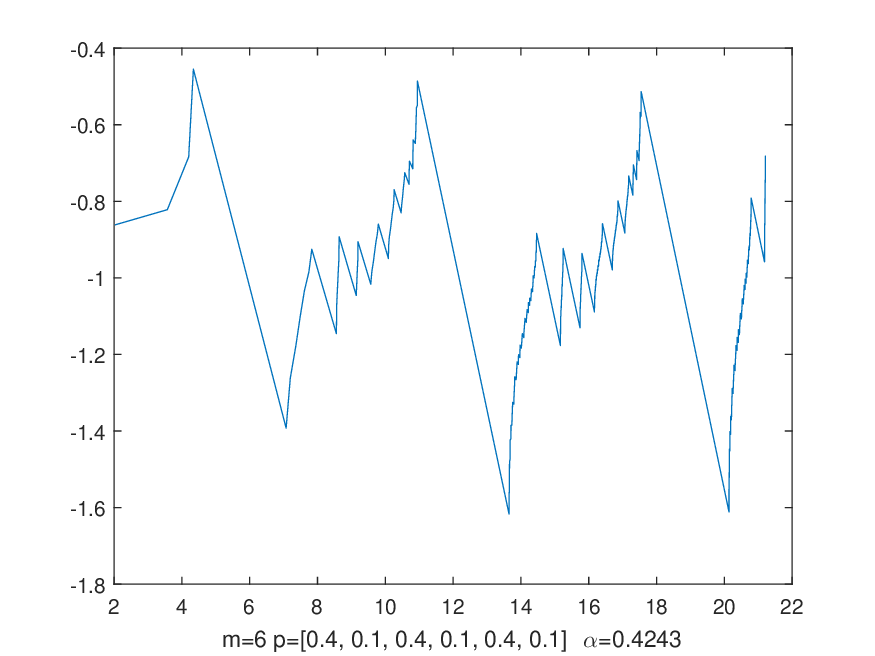}
\caption{Weyl plots for Hierarchical scheme, $p=\{0.1, 0.4, 0.1, ...\}$ (left) and $p=\{0.4, 0.1, 0.4, ...\}$ (right)}
\label{fig:hierchweyl}
\end{figure}

The spectra for these hierarchical Laplacians on the Interval yielded some interesting, though purely experimental, patterns. For example, we examine the spectra for a sequence of parameters $\{p_1, p_2, p_1, p_2, ...\}$ and compare it to the spectra of the Laplacians with a single value for $p$, i.e. $\Delta^{(p_1)}$ and $\Delta^{(p_2)}$. The Weyl plot for the hierarchical spectra with $\{p_1, p_2, p_1, p_2, ...\}$ visually seems to be a sort of mix of the Weyl plots for the spectra of $\Delta^{(p_1)}$ and $\Delta^{(p_2)}$. For example, take $p_1=0.1$ and $p_2 = 0.4$. Figure \ref{fig:stdweyl} shows the Weyl plots for the Laplacian$\Delta^{(p_i)}$, and each of the plots in Figure \ref{fig:hierchweyl} seems to be mix of the preceding Weyl plots in Figure \ref{fig:stdweyl}. We are unable to give an analytic explanation of this phenomena at the moment, but foresee that this way of constructing Laplacians may lead to more concrete results in the future. A related example in \cite{Drenning} yields more decisive graphs.

\section{Solutions of Spacetime Equations}

The methods outlined in this paper provide a framework for the computation of a finite approximation of the spectra for both the Interval and the Sierpinski Gasket. One of the most significant applications of a strong grip on relevant spectra are solutions to spacetime equations through the application of the spectral operator. The general framework is the same as in the case of the Kigami Laplacian (\cite{Kigami1}, \cite{Strichartz}). We sketch the ideas for the convenience of the readers. The main interest of this section are the figures obtained in the general case.

Given an orthonormal basis of eigenfunctions and eigenvalues of the Laplacian, $u_{j}$ and $\lambda_{j}$, ordered such that $\lambda_{j+1} \geq \lambda_{j}$, we define the spectral operator as 

\begin{equation}f(-\Delta)u = \sum_{j=1}^{\infty}f(\lambda_{j})\langle u,u_{j} \rangle u_{j}
\end{equation}

where the inner product $\langle a,b \rangle$ is defined as

\begin{equation}
\langle f,g \rangle=\int a(y)b(y)d\mu(y)
\end{equation}

This spectral operator is used in the classical solution to some spacetime partial differential equations. Unable to numerically compute the infinite series described above, we were reduced to computing numerical approximations using a finite number of eigenfunction and eigenvalue pairs.

Implementation of the spectral operator is not far from previously described computations. Previous sections describe finding eigenfunctions themselves; only orthonormality must be verified. Normalizing each function is a trivial task accomplished by rescaling such that

\begin{equation}
\int u_{j}(y)^{2} d\mu(y)=1
\end{equation}

for all eigenfunctions. On the Interval, orthogonality is provided by the properties of linear algebra - eigenfunctions associated with distinct eigenvalues are guaranteed to be orthogonal. On the Sierpinski Gasket, eigenfunctions associated with eigenvalues of high multiplicity are not neccesarily orthogonal. Clever solutions have been propose to this problem \cite{Allan}, but we chose a simple and direct implementation of the Gram-Schmidt algorithm to provide the necessary orthonormal basis for each eigenvalue on the Sierpinski Gasket.

\subsection{Heat Equation}

Our first application of the spectral operator is to the heat equation, where we seek $u(x,t)$ such that 

\begin{equation}
\frac{\partial u(x,t)}{\partial t} = \Delta_{x}u(x,t)
\end{equation}

and

\begin{equation}
u(x,0) = f(x)
\end{equation}

with Neumann boundary conditions

\begin{equation}
\partial_n u(x,t)|_{\partial \Omega} = 0
\end{equation}

for some region $\Omega$ and initial heat distribtuion described by $f(x)$.
The classical solution to this problem is given by the spectral operator

\begin{equation}u(x,t) = \sum_{\lambda_{j}}e^{-\lambda_{j} t}u_{j}(x)\int u_{j}(y)f(y)d\mu(y)
\end{equation}

where the eigenfunction basis used corresponds to Neumann boundary conditions (the Dirichlet problem can be solved using a Dirichlet basis). The heat equation earns its name describing the flow of heat across space as a function of time, but has deep-rooted connections to Brownian Motion, probability, and random walks.

The simplest choice of $f(x)$ for our analysis is a delta function, with support limited to a single point in the graph approximation. Using Neumann eigenvalues and eigenvectors, we have $\lambda_{1}=0$, with $u_{1}=1$. Then the large $t$ limit is determined wholly by the value of $\int f(y)d\mu(y)$. If $f(x)$ is a delta function with support restricted to $f(x_{0})=d$, this integral further reduces to $\int f(y)d\mu(y)=\frac{d}{m}$ where $m$ is the pointmass assigned to point $x_{0}$. In Figure \ref{fig:heat} several numerical solutions are displayed to delta functions at the center of the interval for different values of $p$.

\begin{figure}
\centering
\includegraphics[width=0.49\textwidth]{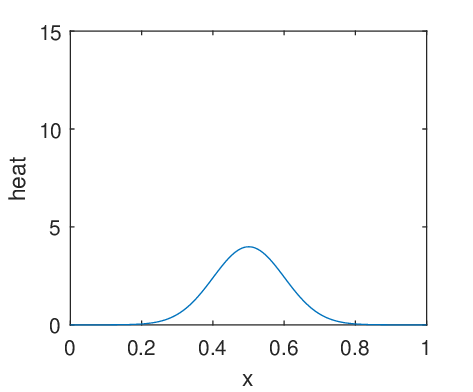}
\includegraphics[width=0.49\textwidth]{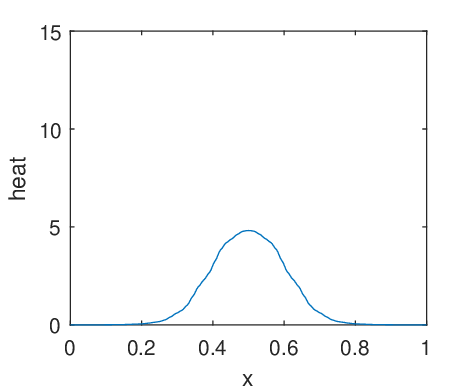}
\newline
\includegraphics[width=0.49\textwidth]{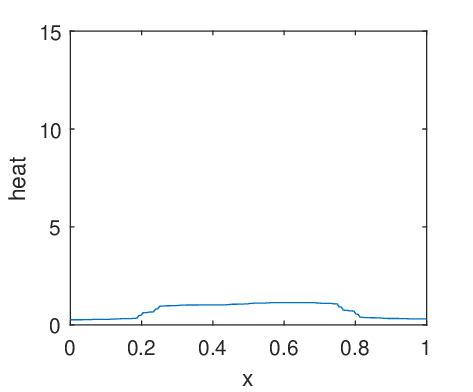}
\includegraphics[width=0.49\textwidth]{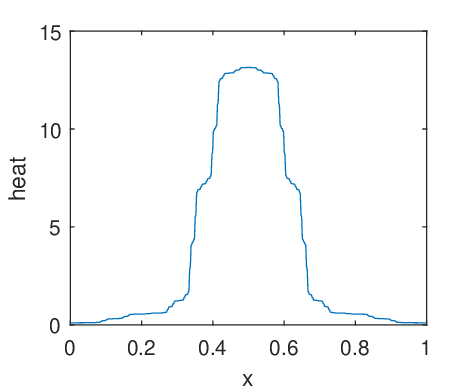}
\caption{Top Row: Heat Solutions for $p=0.5$ (left) and $p=0.6$ Bottom Row: Heat Solutions for $p=0.1$ (left) and $p=0.9$ (right). All at $t=0.005$}

\label{fig:heat}
\end{figure}

The consequence of this is that, considering that the solution will converge to a uniform distribution regardless of the value of $r$ or $p$, $\lim_{t \to \infty} u(x,t) = \frac{d}{m}$ and is therefore very dependent on the choice of parameter.

Furthermore, we can describe the rate of convergence towards these uniform functions. As $t$ moves away from $t=0$, $\lambda_{1}$ begins to dominate, but the last term to fall away will be the second smallest eigenvalue, $\lambda_{2}$. Recalling that eigenvalues tend to infinity at the edge of parameter space, it seems that that heat equation solutions using these laplacians will relax to the ground state the fastest.

\subsection{Wave Equation}

We can also use these methods to solve the wave equation

\begin{equation}
\frac{\partial^{2} u(x,t)}{\partial t^{2}} = \Delta_{x}u(x,t)
\end{equation}

with initial conditions
\begin{align}
u(x,0)=&0\\
\frac{\partial u}{\partial t}u(x,0) =& f(x)
\end{align}

and Dirichlet boundary conditions

\begin{equation}
u(x,t) \big| _{\partial \Omega} = 0
\end{equation}

for some region $\Omega$.

The classical solution is given by

\begin{equation}
u(x,t) = \sum_{\lambda_{j}}\frac{\sin{t\sqrt{\lambda_{j}}}}{\sqrt{\lambda_{j}}}u_{j}(x)\int u_{j}(y)f(y)d\mu(y)
\end{equation}

The wave equation can be solved by numerical methods very similar to those used to solve the heat equation, but the results can be very difficult to analyze quantitatively. Some qualitative results on the Interval are displayed in Figures \ref{fig:wave1}, \ref{fig:wave2}, \ref{fig:wave3}. In the standard case we form the classic traveling wave formation, but varying $p$ slightly to $p=0.52$ produces a small wake behind the leading peak. This wake remains as the leading peak inverts along the right boundary and returns to the origin.

As the value of $p$ becomes more extreme, oscillations become focused onto specific regions of the Interval - in particular those with the smallest measure, as seen in the second row of Figures \ref{fig:wave1}, \ref{fig:wave2}, \ref{fig:wave3}. Similar behavior occurs on SG, but is more difficult to analyze and display because of the increased complexity of the underlying structure. The infinite propagation speed proved in \cite{Lee} and \cite{Ngai} holds here, and is visible in Figures \ref{fig:wave1}, \ref{fig:wave2}, \ref{fig:wave3}.

\begin{figure}
\centering
\includegraphics[width=0.49\textwidth]{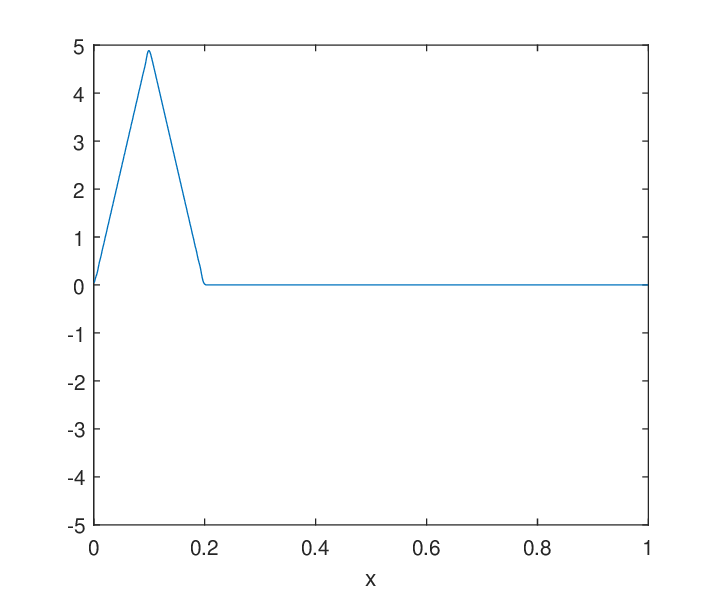}
\includegraphics[width=0.49\textwidth]{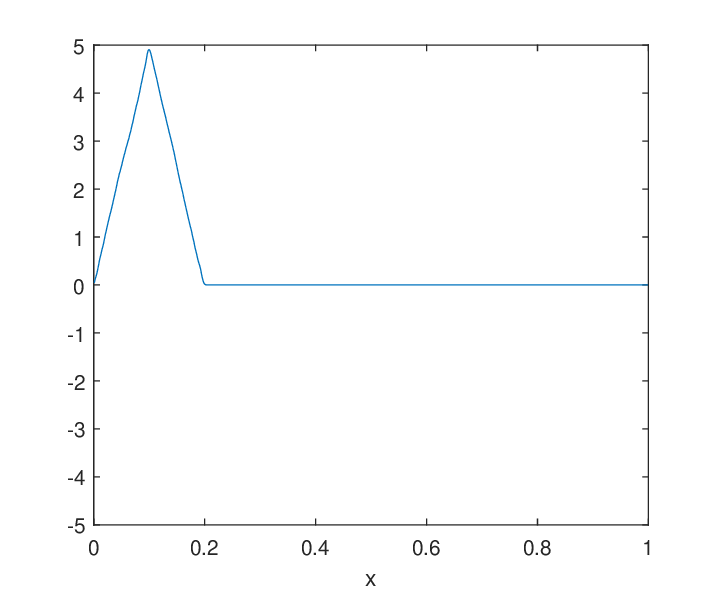}
\newline
\includegraphics[width=0.49\textwidth]{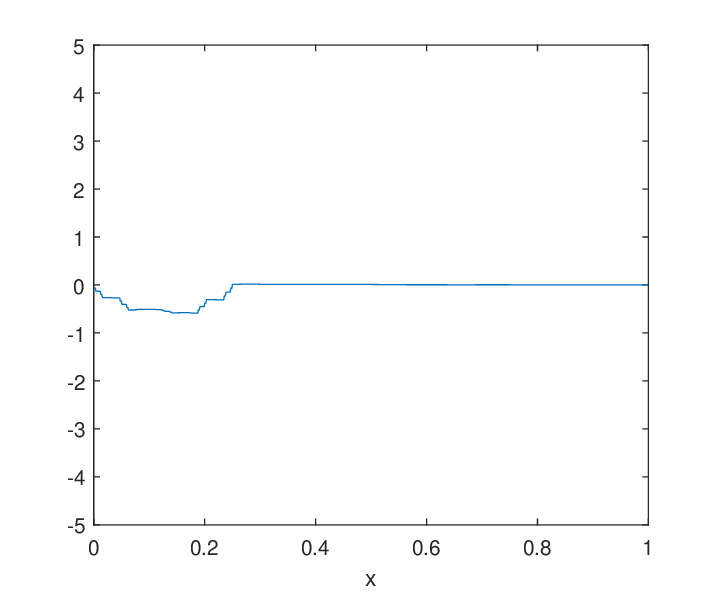}
\includegraphics[width=0.49\textwidth]{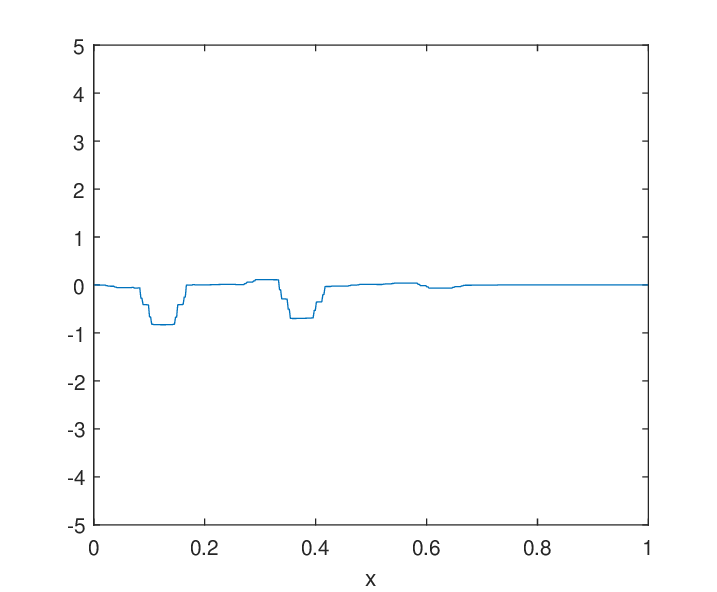}
\caption{Top Row: Wave Solutions for $p=0.5$ (left) and $p=0.52$ Bottom Row: Wave Solutions for $p=0.01$ (left) and $p=0.99$ (right). All at $t=0.1$}

\label{fig:wave1}
\end{figure}

\begin{figure}
\centering
\includegraphics[width=0.49\textwidth]{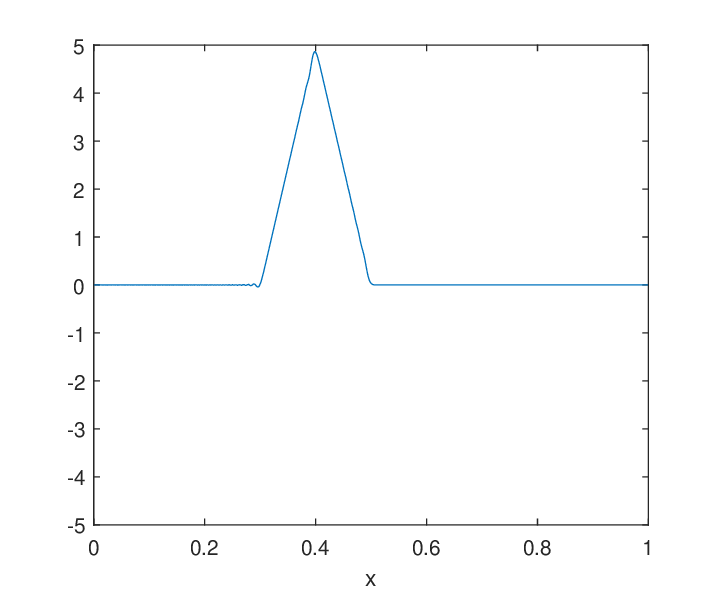}
\includegraphics[width=0.49\textwidth]{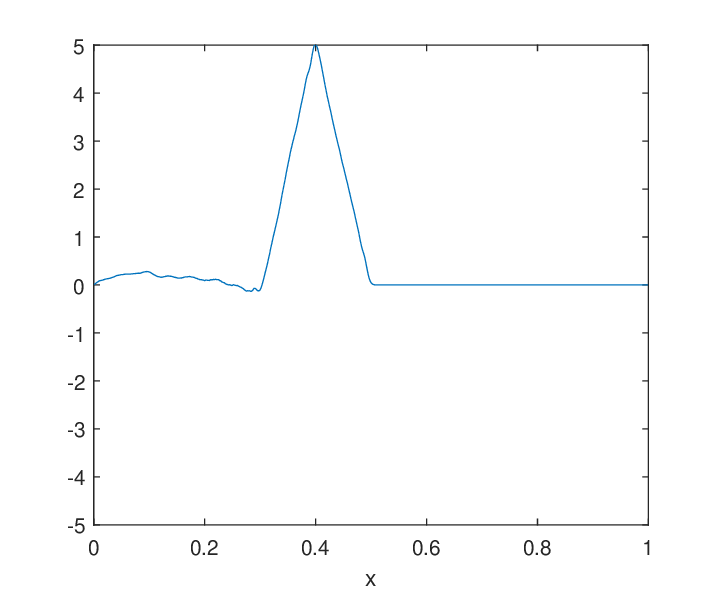}
\newline
\includegraphics[width=0.49\textwidth]{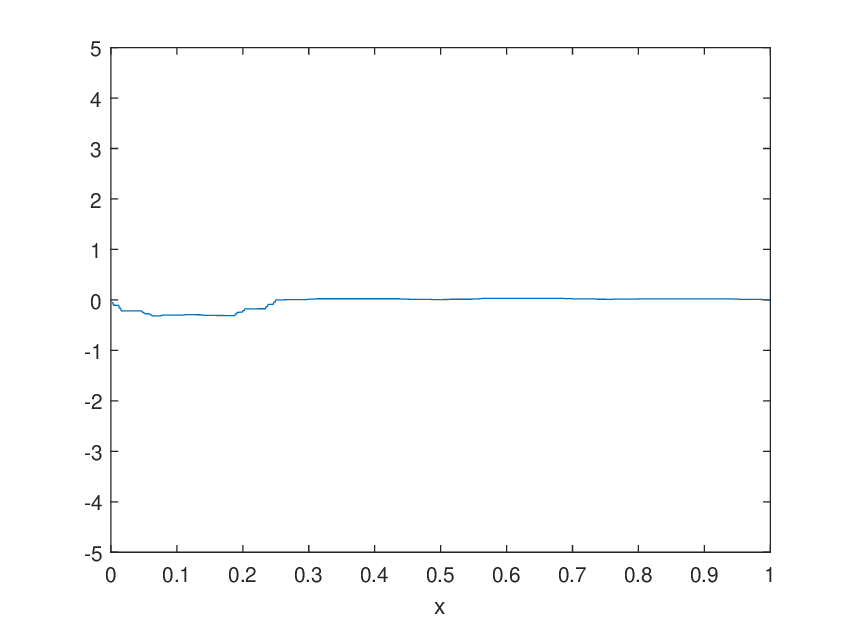}
\includegraphics[width=0.49\textwidth]{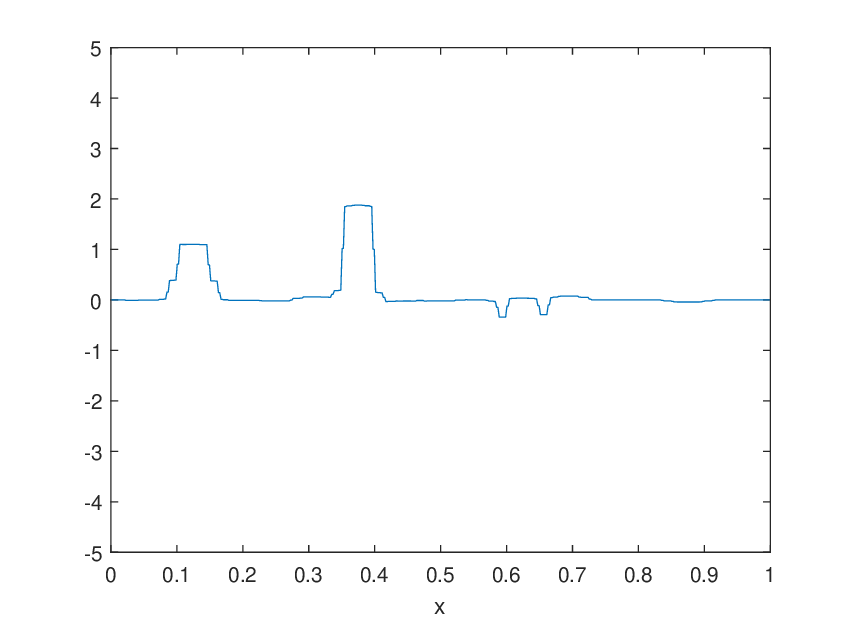}
\caption{Solutions above at $t=0.4$}

\label{fig:wave2}
\end{figure}

\begin{figure}
\centering
\includegraphics[width=0.49\textwidth]{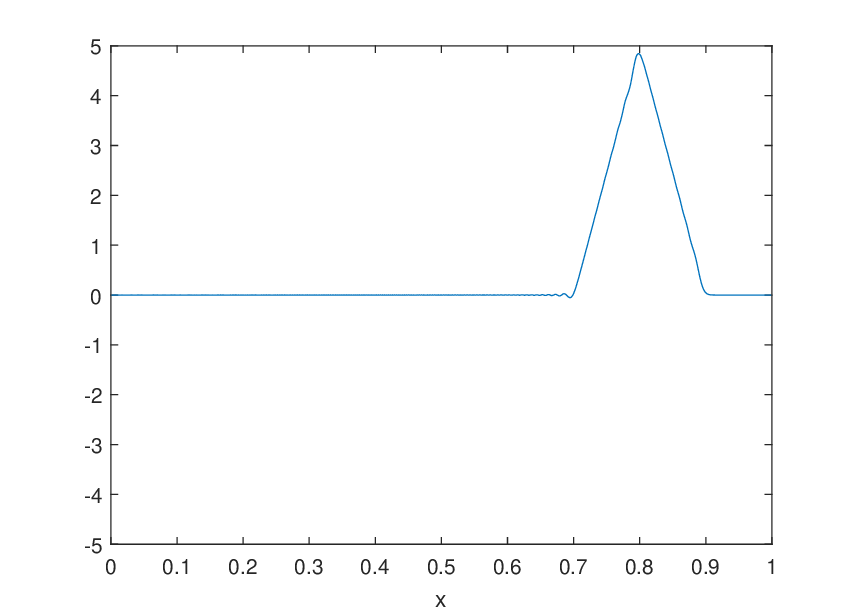}
\includegraphics[width=0.49\textwidth]{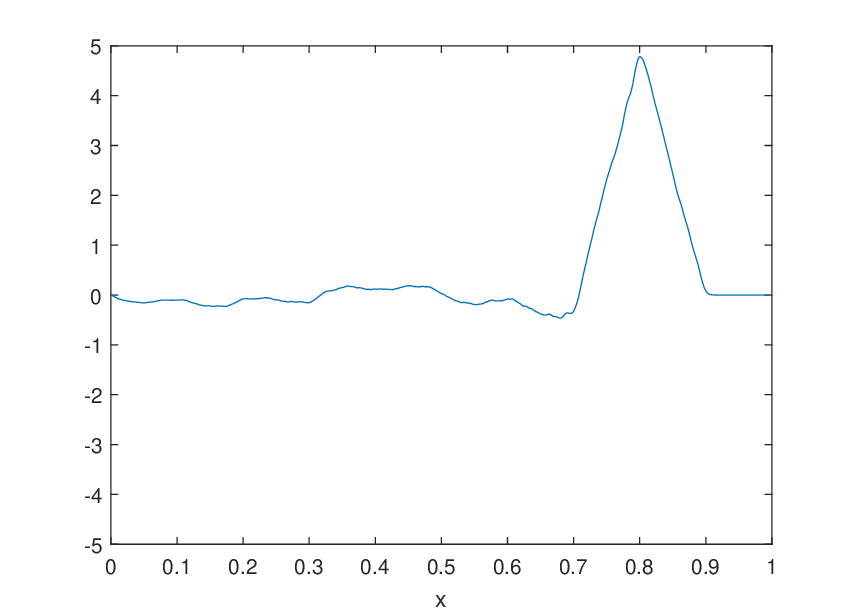}
\newline
\includegraphics[width=0.49\textwidth]{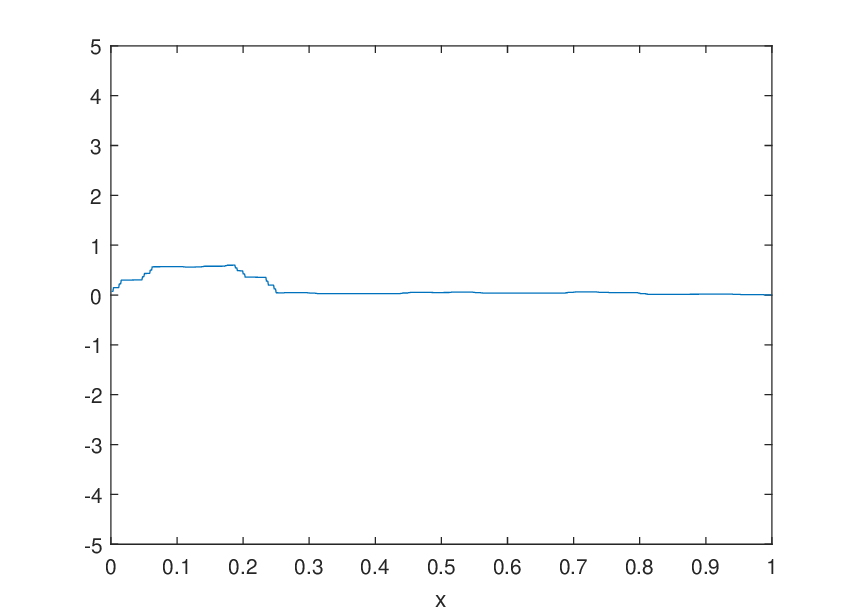}
\includegraphics[width=0.49\textwidth]{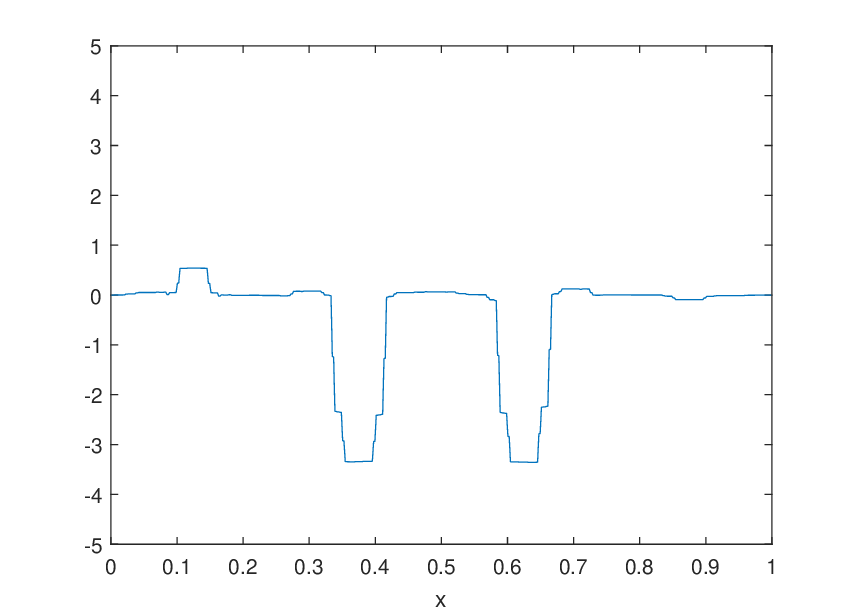}
\caption{Solutions above at $t=0.8$}

\label{fig:wave3}
\end{figure}

\begin{tikzpicture}

\end{tikzpicture}

\begin{tikzpicture}

\end{tikzpicture}

\bigskip
\bigskip
\footnotesize
\noindent\textit{Acknowledgments.}
D. A. King's research was supported by the Nancy Susan Reynolds Scholarship Committee. E. Lee's research was supported by the Summer Program for Undergraduate Research at Cornell University.
\baselineskip=17pt
\printbibliography

\end{document}